\documentclass[10pt,oneside]{amsart}
\usepackage[english]{babel}
\usepackage{natbib}
\usepackage{url}
\usepackage{inputenc}
\usepackage{amsfonts}
\usepackage{a4wide}
\usepackage[scale = .8]{geometry}
\usepackage{amsmath}
\usepackage{empheq}
\usepackage{graphicx}
\graphicspath{{images/}}
\usepackage{parskip}
\usepackage{fancyhdr}
\usepackage{bbold}
\usepackage{amssymb}
\usepackage{vmargin}
\usepackage{array}
\usepackage{amsthm}
\usepackage[ruled,vlined]{algorithm2e}
\usepackage{hyperref}
\usepackage{caption}
\usepackage{float}
\usepackage{fourier-orns}
\usepackage{listings}
\usepackage{epsfig}
\usepackage{mathrsfs}
\usepackage{enumitem}

\usepackage{amsmath}    
\newtheorem{theorem}{Theorem}[section]

\newtheorem{definition}[theorem]{Definition}
\newtheorem{example}[theorem]{Example}
\newtheorem{lemma}[theorem]{Lemma}

\newtheorem{proposition}[theorem]{Proposition}
\newtheorem{remark}[theorem]{Remark}

\DeclareMathOperator{\diag}{diag}
\usepackage{listings}

\newcommand{\bu}{\bar{u}}
\newcommand{\bv}{\bar{v}}

\newcommand{\R}{{\mathbb{R}}}
\newcommand{\Enu}{{E_k^{(\nu)}}}
\newcommand{\Z}{{\mathbb{Z}}}
\newcommand{\N}{{\mathbb{N}}}
\newcommand{\C}{\mathbb{C}}

\newcommand{\eps}{\varepsilon}

\newcommand{\cS}{\mathcal{S}}
\newcommand{\cA}{\mathcal{A}}
\newcommand{\cK}{\mathcal{K}}
\newcommand{\cI}{\mathcal{I}}
\newcommand{\cP}{\mathcal{P}}

\newcommand{\cM}{\mathcal{M}}

\newcommand{\cD}{\mathcal{D}}
\newcommand{\cF}{\mathcal{F}}
\newcommand{\cG}{\mathcal{G}}
\newcommand{\cC}{\mathcal{C}}
\newcommand{\cB}{\mathcal{B}}

\newcommand{\cL}{\mathcal{L}}
\newcommand{\cQ}{\mathcal{Q}}
\newcommand{\cV}{\mathcal{V}}

\newcommand{\bU}{\bar{U}}
\newcommand{\bcV}{\bar{\cV}}
\newcommand{\bcM}{\bar{\cM}}

\newcommand{\B}{\mathscr{B}}
\newcommand{\fB}{\mathfrak{B}}

\newcommand{\pa}{\partial}

\newcommand{\tcL}{\tilde{\cL}}
\newcommand{\bydef}{\stackrel{\mbox{\tiny\textnormal{\raisebox{0ex}[0ex][0ex]{def}}}}{=}}

\usepackage[dvipsnames]{xcolor}

\newcommand{\blue}[1]{{\color{blue} #1}}

\begin{document}


\title[Existence and stability of stationary solutions to parabolic PDEs]{Constructive existence proofs and stability of stationary solutions to parabolic PDEs using Gegenbauer polynomials}

\author[M. Breden]{Maxime Breden}
\address{CMAP, CNRS, Ecole polytechnique, Institut Polytechnique de Paris, 91120 Palaiseau, France}
\email{maxime.breden@polytechnique.edu}

\author[M. Cadiot]{Matthieu Cadiot}
\address{CMAP, CNRS, Ecole polytechnique, Institut Polytechnique de Paris, 91120 Palaiseau, France}
\email{matthieu.cadiot@polytechnique.edu}

\author[A. Zurek]{Antoine Zurek}
\address{CNRS – Université de Montréal CRM – CNRS, Montréal, Canada and Universit\'e de Technologie de Compi\`egne, LMAC, 60200 Compi\`egne, France}
\email{antoine.zurek@utc.fr}

\date{\today}

\begin{abstract}
In this paper, we present a computer-assisted framework for constructive proofs of existence for stationary solutions to one-dimensional parabolic PDEs and the rigorous determination of their linear stability. By expanding solutions in Gegenbauer polynomials, we first develop a general approach for boundary value problems (BVPs), corresponding to the stationary part of the PDE. This yields a computationally efficient sparse structure for both differential and multiplication operators. By deriving sharp, explicit and quantitative estimates for the inverse of differential operators, we implement a Newton-Kantorovich approach. Specifically, given a numerical approximation  
$\bar{u}$, we prove the existence of a true stationary solution $\tilde{u}$ within a small, rigorously quantified neighborhood of $\bar{u}$. A key advantage of this approach is that the sharp control over the defect $\tilde{u}-\bar{u}$, integrated with the spectral properties of the Gegenbauer basis, enables an accurate enclosure of the linearization's spectrum around $\tilde{u}$. This allows for a definitive conclusion regarding the (in)stability of the verified solution, which is the main contribution of the paper. We demonstrate the efficacy of this method through several applications, capturing both stable and unstable equilibrium states.
\bigskip
		
		
\medskip
		
\end{abstract}

\maketitle

\tableofcontents


\section{Introduction}

We are interested in the following class of parabolic PDEs: 
\begin{align}\label{eq : parabolic PDE original}
    \begin{cases}
        \displaystyle \partial_t\, v = (-1)^{m+1} \, \partial_x^{2m}\, v + \cQ(v)   + \psi(x) &(t,x) \in (0,T) \times (-1,1),\\
       \displaystyle  \sum_{j=0}^{2m-1} \beta_{j,i} \, \partial^j_x \,v(t,-1) = 0 &\text{ for all } i \in \{1, 2, \dots, m\},\\
        \displaystyle  \sum_{j=0}^{2m-1} \gamma_{j,i}\, \partial^j_x \, v(t,1) = 0 &\text{ for all } i \in \{1, 2, \dots, m\},
    \end{cases}
\end{align}
complemented with some initial conditions and where $\beta_{j,i}$ and $\gamma_{j,i}$ are constant coefficients. 
We assume that $\cQ(v)$ is a polynomial in $v$ and in its partial derivatives in $x$ up to the order $2m-1$. More specifically, we assume that
\begin{align}\label{eq : nonlinear part}
    \cQ(v) = v^{j_0} (\partial_x v)^{j_1} \dots (\partial_x^{2m-1} v)^{j_{2m-1}}, \text{ where } (j_0, \dots, j_{2m-1}) \in (\mathbb{N}_0)^{2m},
\end{align}
where $\mathbb{N}_0$ denotes the set of nonnegative integers. Under \eqref{eq : nonlinear part}, we have that \eqref{eq : parabolic PDE original} is autonomous and semi-linear. Finally, the source term $\psi = \psi(x)$ is assumed to be a time-independent polynomial in $x$.

In order to understand the global dynamics of~\eqref{eq : parabolic PDE original}, a first step is to study its equilibria and their stability. More precisely,  the equilibria solve the stationary version of \eqref{eq : parabolic PDE original} given by 
\begin{align}\label{eq : stationary PDE}
    (-1)^{m+1}\partial_x^{2m}\, v + \cQ(v)  + \psi = 0, \qquad x \in  (-1,1),
\end{align}
with the boundary conditions
\begin{align}\label{eq : BC stationary PDE}
    \begin{cases}
       \displaystyle  \sum_{j=0}^{2m-1} \beta_{j,i}\, \partial^j_x\,v(-1) = 0 &\text{ for all } i \in \{1, 2, \dots, m\},\\
        \displaystyle  \sum_{j=0}^{2m-1} \gamma_{j,i}\, \partial^j_x\,v(1) = 0 &\text{ for all } i \in \{1, 2, \dots, m\}.
    \end{cases}
\end{align}

However, studying elliptic equations of the form~\eqref{eq : stationary PDE}--\eqref{eq : BC stationary PDE} can already be very challenging, if only because such equations can have multiple non-trivial solutions. 

In the past decades, computer-assisted methods have emerged to study PDEs. One of the interesting features of these techniques is that they not only provide existence results, but also a quantitative description of the obtained solutions, as the resulting theorems take the form: ``there exists a solution $v$ in a small and explicit neighborhood of a given function $\bar{v}$ which is explicitly known'' (and typically obtained by numerically solving the PDE). A popular approach for proving such statements is to use a fixed-point argument in the spirit of the Newton-Kantorovich theorem~\cite{BerLes15}, although more geometrical or topological approaches can also be very successful~\cite{CAPD,WilZgl20,WilZgl24}. For a broader overview of computer-assisted proofs in PDEs, we refer to the survey~\cite{Gom19}, the book~\cite{NakPluWat19} and the references therein.

Among the computer-assisted proofs for elliptic PDEs based on the Newton-Kantorovich theorem, most works use either finite elements or Fourier series to approximate solutions. While finite elements are much more flexible regarding the geometry of the domain, Fourier series are still a very natural tool to use for many problems, and when they can be used they not only lead to accurate numerics, but also to more efficient and streamlined computer-assisted proofs~\cite{HunLesMir16,Ber17}. In particular, once a steady state of a semilinear parabolic PDE has been obtained through a computer-assisted proof using Fourier series, it is often straightforward to also obtain the Morse index of the solution, i.e., the number of unstable eigenvalues of the linearization at the equilibrium, essentially because the computer-assisted proof of the steady state already requires a relatively sharp control of that linear operator~\cite{AriKoc10,MirRei19}. For another approach based on cone conditions and self-consistent bounds, see for a general description~\cite{Zgl09}, and~\cite{Zgl02bis} for an early PDE example. However, Fourier series can only be used on the torus, or for equations posed on (hyper-)rectangular domains with homogeneous Dirichlet or Neumann boundary conditions and appropriate symmetries. For instance, even in dimension 1 and for a simple equation like
\begin{align}\label{eq : toy problem}
\begin{cases}
    \partial_x^2 v + \alpha\, (v^2 +1 ) = 0 \qquad x\in(0,1),\\
    v(0) = v(1) = 0,
\end{cases}
\end{align}
Fourier series do not seem to be a suitable option, as the boundary conditions require first extending the solution $v$ as an odd function on $(-1,1)$, and therefore using sine series, but the term $\alpha (v^2 +1 )$ is not odd and therefore very poorly approximated using sine series. For such an example, or for more complicated boundary conditions, a natural alternative to Fourier series are Chebyshev series, which have similar properties but are not restricted to periodic (or periodizable) problems. Chebyshev-based approximations are well-known to be very efficient for approaching smooth functions, and in particular for solving boundary value problems~\cite{chebfun}.
Computer-assisted proofs for nonlinear boundary value problems using Chebyshev series can already be found in~\cite{LesRei14,BerShe21} and many other works, and also easily accommodate complicated boundary conditions~\cite{BreChaZur21}. 
However, it should be noted that in all these works, equations of the form~\eqref{eq : stationary PDE} are first rewritten as systems of first order equations, thereby losing the higher regularizing properties of~\eqref{eq : stationary PDE}, which leads to suboptimal estimates and unnecessarily expensive computer-assisted proofs. In this paper, we show that by using different families of Gegenbauer polynomials (which generalize Chebyshev polynomials), one can easily retain the high order structure of~\eqref{eq : stationary PDE}, leading to computer-assisted proofs with very similar estimates to what one would obtain with Fourier series on the torus. We note that this strategy is also known to be useful from a purely numerical point of view~\cite{OlvTow13}.

The second and main contribution of this paper is to introduce an approach allowing to also study the spectral stability of the steady states obtained with such computer-assisted proofs based on Gegenbauer polynomials, that is, to rigorously count the number of unstable eigenvalues (because the linearized equation is elliptic and we work on a bounded domain, there are only eigenvalues in the spectrum). More precisely, given a steady state solution $\tilde{v}$ of~\eqref{eq : parabolic PDE original}, we want to obtain computable and rigorous enclosures for each $\lambda\in\C$ such that there exists a non-trivial function $v:[-1,1]\to\C$ satisfying
\begin{align}\label{eq : evp functions}
    \begin{cases}
        (-1)^{m+1} \, \partial_x^{2m}\, v + D\cQ(\tilde{v}) v  = \lambda v &x \in  (-1,1),\\
       \displaystyle  \sum_{j=0}^{2m-1} \beta_{j,i} \, \partial^j_x \,v(-1) = 0 &\text{ for all } i \in \{1, 2, \dots, m\},\\
        \displaystyle  \sum_{j=0}^{2m-1} \gamma_{j,i}\, \partial^j_x \, v(1) = 0 &\text{ for all } i \in \{1, 2, \dots, m\}.
    \end{cases}
\end{align}
In the sequel, we say that such an eigenvalue $\lambda$ is stable if $\Re(\lambda)<0$ and unstable if $\Re(\lambda)>0$. Situations with eigenvalues having zero real part require extra work and are not considered here, but we emphasize that our approach can be used without knowing a priori that there are no eigenvalues with zero real part.

In the context of this paper, note that spectral (in)stability implies nonlinear (in)stability since we work with parabolic-semilinear PDEs, see, e.g.,~\cite[Section 5.1]{henry_semilinear}. our goal is to obtain sufficiently sharp quantitative enclosures to be able to rigorously determine the number of unstable eigenvalues. The main reason we will be able to do so is because our computer-assisted proof for the steady state solution $\tilde{v}$ already provides us with a lot of quantitative information about the linearized operator. However, it should be noted that, in contrast to similar problems that can be written using Fourier series, we have to deal here with a generalized eigenvalue problem. This significantly complicates the analysis. Indeed, if the boundary conditions and symmetries in~\eqref{eq : evp functions} allowed for a Fourier series expansion, one could simply write~\eqref{eq : evp functions} in Fourier space, and use a Gershgorin argument to directly control the spectrum of $(-1)^{m+1} \partial_x^{2m} + D\cQ(\tilde{v})$. Such arguments can be found, e.g., in~\cite{BrePayReiTan25,Cad25}, see also~\cite{AriKoc10} for earlier computer-assisted stability proofs of the same flavor.
In our setting, it seems a generalized eigenvalue problem is unavoidable, because the more general boundary conditions cannot easily be incorporated in the function space, and because $\partial_x^{2m}$ does not have a simple expression in the Chebyshev basis. As we will see in Section~\ref{subsec:diffop}, $\partial_x^{2m}$ does have a simple expression when considered between different Gegenbauer bases, but this would require adding a change of basis on the $\lambda v$ term, and therefore would also lead to a generalized eigenvalue problem. We eventually recast this generalized eigenvalue problem as a standard eigenvalue problem, but we then end up having to control the spectrum of a compact operator whose eigenvalues accumulate at zero, which makes distinguishing between stable and unstable eigenvalues much more challenging. We overcome this issue by combining a computer-assisted Gershgorin argument with a priori bounds on the eigenvalues. We also propose an alternative procedure, based on directly studying the generalized eigenvalue problem, which is slightly more technical but sometimes more efficient in practice.

In summary, we showcase in this paper that the existence and the stability of steady states of equations of the form~\eqref{eq : parabolic PDE original} can be studied with efficient computer-assisted proofs based on spectral methods, thus generalizing the known Fourier techniques from the literature that were only applicable on the torus (or for problems admitting a smooth periodic extension). In particular, this work represents a first step towards the application of such methods for the study of applied models relying on PDEs with possibly intricate boundary conditions, and for which more classical computer-assisted techniques are not applicable. If, for pedagogical purposes, we will only consider here some academic examples of PDEs (see Section~\ref{sec : applications}), we intend to study more complex models in future work. We have in mind, for instance, to pursue the analysis done in~\cite{BreChaZur21} to study the stability of traveling waves solutions for the so-called Diffusion Poisson Coupled Model (DPCM). This model is a one-dimensional free-boundary problem involving drift-diffusion equations coupled with a Poisson equation and some nonlinear Fourier boundary conditions. It arises in the general study of the long term safety of the geological repository of nuclear wastes. Another possible field of application of the methods developed in this work is the uniqueness or non-uniqueness of one-dimensional stationary solutions for semiconductor equations~\cite[Chapter 3]{Mar86}. These systems are also composed of drift-diffusion equations coupled with a Poisson equation and either Dirichlet or mixed Dirichlet/Neumann boundary conditions. If the uniqueness of stationary solutions is known for small applied potentials~\cite[Section 3.4]{Mar86}, only a few theoretical results are known in full generality; see~\cite{Ala92} and references therein.

The two above examples are not exactly of the form~\eqref{eq : stationary PDE}--\eqref{eq : BC stationary PDE}. However, this framework is used throughout this paper for the sake of simplifying the presentation and because the point of this paper is to focus on one main difficulty shared by these examples: boundary conditions that are not compatible with Fourier expansions. Let us also emphasize that the setup we introduce is flexible enough to accommodate a broader class of problems that could contain, for instance,
\begin{itemize}
    \item nonhomogeneous terms or boundary conditions,
    \item nonlinear boundary conditions,
    \item equations containing a sum of several terms of the form~\eqref{eq : nonlinear part},
    \item systems instead of scalar equations,
    \item analytic but non polynomial nonlinearities.
\end{itemize}
The crucial assumptions are that the equation (or system) is elliptic and semilinear, and posed on a bounded interval. The extension to higher dimensional rectangular domains is non-trivial but will be studied in a future work. Let us also note that, even though we focus here on boundary value problems, our setup could also be used to enclose solutions of ODEs with initial conditions, without having to first rewrite them as a system of first order equations, as suggested in~\cite{LesRei14}. For a slightly different setup also using Chebyshev-based approximation to rigorously integrate (linear) ODEs, see~\cite{BreBriJol18,Bre18bis}.

The remainder of the paper is organized as follows. In Section~\ref{sec : gegenbauer} we give a sequence space reformulation of the steady state problem~\eqref{eq : stationary PDE}-\eqref{eq : BC stationary PDE} using Gegenbauer bases, which is amenable to computer-assisted proofs, and derive related quantitative compactness estimates. The usual Newton-Kantorovich strategy for computer-assisted proof is recalled in Section~\ref{sec:NK}, and applied to the steady state problem~\eqref{eq : stationary PDE}-\eqref{eq : BC stationary PDE}. In Section~\ref{sec : stability}, we then present two separate approaches for studying the stability of steady states of~\eqref{eq : parabolic PDE original}. The whole procedure is then applied to two examples in Section~\ref{sec : applications}. Finally, Appendix~\ref{App : explicit inverses} contains explicit formulas for the Dirichlet and Neumann Laplacian operators in sequence spaces, and Appendix~\ref{App : rough enclosure spectrum KS} the proof of an a priori estimate regarding possible unstable eigenvalues of the Kuramoto--Sivashinsky model.

\section{Formulation of the problem using Gegenbauer polynomials}\label{sec : gegenbauer}

We first recall the definition of Gegenbauer polynomials in Section~\ref{subsec: gegenbauer def}, which allows us to fix some notation, and recall some of their basic properties. For a more thorough treatment, we refer to~\cite[\href{https://dlmf.nist.gov/18}{(18)}]{DLMF} and \cite{boyd_cheb_fourier}. Suitable Banach spaces and transformations between sequences of coefficients and function series are discussed in Section~\ref{subsec:spaces}, and formula for differential operators in sequence spaces are then derived in Section~\ref{subsec:diffop}. We then first express and solve linear boundary value problems in coefficient space in Section~\ref{ssec : linear problem}, and then finally rewrite the whole problem~\eqref{eq : stationary PDE}-\eqref{eq : BC stationary PDE} in coefficient space in Section~\ref{ssec : zero finding}. Section~\ref{ssec : compactness} contains quantitative compactness estimates that play a critical role in our computer-assisted proofs.

\subsection{Introduction to the Gegenbauer polynomials}
\label{subsec: gegenbauer def}

In the entire paper, $\N_0$ denotes the set of nonnegative integers and $\N$ the set of positive integers. For any $k \in \mathbb{N}_0$, the Gegenbauer polynomials of order $k$, denoted $(G^{(k)}_n)_{n=0}^\infty$, are orthogonal polynomials on the interval $[-1,1]$ associated with the weight $\omega_k : (-1,1) \to \R$ defined as
\begin{align} \label{def : weight}
\omega_k(x) \bydef (1-x^2)^{k-\frac{1}{2}}.
\end{align}
More precisely, we have
\begin{align}\label{eq : inner product Gegenbauer}
    \int_{-1}^1 G^{(k)}_i(x) \, G^{(k)}_j(x) \, \omega_k(x) \, dx = \Omega_{k,j} \delta_{i,j},
\end{align}
where $\delta_{i,j}$ is the usual Kronecker delta, and the normalizing sequence factor $(\Omega_{k,j})_{j \in \N_0}$ is given by 
\begin{equation}\label{eq: normalization}
\begin{aligned}
     \Omega_{k,j} \bydef \begin{cases}
         \dfrac{\pi}{2- \delta_{0,j}} & \text{ if } k=0,\\\\
        \dfrac{\pi\, 2^{1-2k}\, \Gamma(j + 2k)}{j! \,(j+k) \,\Gamma(k)^2} & \text{ if } k \in \mathbb{N}.
    \end{cases}
\end{aligned}
\end{equation}
Note that the cases $k=0$ and $k=1$ correspond to the Chebyshev polynomials of the first and of the second kind respectively. More generally, Gegenbauer polynomials are a specific case of Jacobi polynomials, for which the weight $(1-x)^\alpha(1+x)^\beta$ is taken even (i.e., $\alpha=\beta$). We recall below some well-known formulas involving Gegenbauer polynomials that will prove useful in our work, and refer to~\cite[\href{https://dlmf.nist.gov/18}{(18)}]{DLMF} for a more exhaustive description of their properties and further references.


%
%

\begin{lemma}[Evaluation, derivatives and change of basis]\label{lem : formulas change of basis and derivatives}
    Given $n \in \mathbb{N}_0$, we have the following properties:
    \begin{align}\label{eq : evaluation at 1 or -1}
    G^{(k)}_{n}(1) = \begin{cases}
    1 &\text{ if } k=0,\\
        \dfrac{\Gamma(2k+n)}{\Gamma(2k)\, n!} &\text{ if } k \in \mathbb{N},
    \end{cases} ~~ \text{ and }  ~~ G^{(k)}_{n}(-1) =  \begin{cases}
    (-1)^n &\text{ if } k=0,\\
        (-1)^n\dfrac{\Gamma(2k+n)}{\Gamma(2k)\, n!} &\text{ if } k \in \mathbb{N},
    \end{cases}
    \end{align}
    \begin{equation}\label{eq : derivative}
        \frac{d}{dx} G^{(k)}_n = \begin{cases}
            n\, G^{(1)}_{n-1}(x) &\text{ if } k=0,\\
            2k\, G^{(k+1)}_{n-1}(x) &\text{ if } k \in \mathbb{N},
        \end{cases}
    \end{equation}
    and 
    \begin{equation}\label{eq : change of basis}
        G^{(k)}_n= \begin{cases}
            \dfrac{1}{2}\left(G_n^{(1)} - G^{(1)}_{n-2}\right) &\text{ if } k=0,\\ 
            \dfrac{k}{n+k}\left(G_n^{(k+1)} - G^{(k+1)}_{n-2}\right) &\text{ if } k \in \mathbb{N},
        \end{cases}
    \end{equation}
    with the convention that $G^{(k)}_{-1} = G^{(k)}_{-2} =0$ for all $k \in \mathbb{N}_0.$
\end{lemma}

\subsection{Banach algebra and different representations}
\label{subsec:spaces}



%
In this work, we ultimately reformulate problem~\eqref{eq : stationary PDE}-\eqref{eq : BC stationary PDE} in terms of Chebyshev coefficients only, but more general Gegenbauer expansions prove useful in various intermediate steps, and we often go back and forth between sequences of Gegenbauer coefficients and the corresponding function series. In order to keep track of these different formulations, we introduce a Fourier type transform associated to the Gegenbauer polynomials. For this purpose, we define $L^2_{\omega_k}$ as the $\omega_k$ weighted $L^2(-1,1)$ space, i.e., the set of measurable functions $u:(-1,1)\to\R$ such that
\begin{align*}
    \|u\|^2_{L^2_{\omega_k}} \bydef \int_{-1}^1 |u(x)|^2 \, \omega_k(x) \, dx < \infty.
\end{align*}
Similarly, recalling definition~\eqref{eq: normalization} of the sequence $\Omega_k = (\Omega_{k,n})_{n\in\N_0}$, we define $\ell^2_{\Omega_k}$ as the $\Omega_k$ weighted $\ell^2(\mathbb{N}_0)$ sequence space, i.e., the set of sequences $U\in\R^{\N_0}$ such that
\begin{align*}
    \|U\|^2_{\ell^2_{\Omega_k}} \bydef \sum_{n \in \N_0} |U_n|^2 \, \Omega_{k,n} < \infty.
\end{align*}
Then, for any $k \in \N_0$, define  the map $\mathcal{G}_k : \ell^2_{\Omega_k} \to L^2_{\omega_k}$ as 
\begin{align}\label{def : fourier transform}
    \mathcal{G}_k(U) = U_0 + 2 \sum_{n=1}^\infty U_n \, G_n^{(k)} \quad \mbox{ for all } U = (U_n)_{n \in \mathbb{N}_0} \mbox{ in }\ell^2_{\Omega_k}.
\end{align}
Using \eqref{eq : inner product Gegenbauer}, we observe that $\mathcal{G}_k : \ell^2_{\Omega_k} \to L^2_{\omega_k}$ has an inverse
 $\mathcal{G}_k^{-1}: L^2_{\omega_k} \to \ell^2_{\Omega_k}$ given by
\begin{align}\label{def : fourier transform inverse}
    \left(\mathcal{G}_k^{-1}(u)\right)_n = \begin{cases}
      \displaystyle \frac{1}{\Omega_{k,0}} \int_{-1}^1 u(x) \, G_0^{(k)}(x) \, \omega_k(x) \, dx  &\text{ if } n=0,\\
      \displaystyle \frac{1}{2\Omega_{k,n}} \int_{-1}^1 u(x) \, G_n^{(k)}(x) \, \omega_k(x) dx  &\text{ if } n\geq 1.
    \end{cases} 
\end{align}
The maps $\cG_k$ send a sequence of Gegenbauer coefficients to the corresponding Gegenbauer series, and vice versa for $\cG_k^{-1}$. Throughout the paper, we will systematically use uppercase letters like $U$ to denote sequences, and lowercase symbols like $u$ to denote the associated function.

While $\ell^2$ and $L^2$ spaces are the natural ones to define the above transformations, solutions of~\eqref{eq : stationary PDE}-\eqref{eq : BC stationary PDE} are expected to be much more regular, and other choices of space can greatly simplify the analysis required to rigorously validate approximate solutions. Given $\nu \geq 1$, we define the Banach space $\ell^1_\nu$ by
\begin{align}
\label{eq : ell1norm}
    \ell^1_\nu \bydef \left\{ U =  (U_n)_{n = 0}^\infty, ~ \|U\|_{\nu} < \infty \right\}, \quad \mbox{where }  \|U\|_{\nu} \bydef |U_0| + 2 \sum_{n \in \mathbb{N}} |U_n| \nu^n.
\end{align}
Such spaces are convenient to represent Chebyshev coefficients of analytic functions~\cite{Tre13}; and they have the useful property of being Banach algebra under the discrete convolution product corresponding to the point-wise product of functions. If analytic regularity is too much to ask for the problem at hand, weaker norms (with algebraic weights instead of exponential ones) corresponding to functions having a finite number of derivatives can also be used. 

For any $k\in\N_0$, we define $E_k^{(\nu)}\in\ell^1_\nu$ to be the sequence given by
\begin{equation} \label{def:Ek}
    \left(\Enu\right)_n = 
    \begin{cases}
    1 &\text{ if } k = n = 0\\
        \frac{1}{2\nu^k}  &\text{ if } k \geq 1 \text{ and } n=k,\\
        0  &\text{ otherwise}.
    \end{cases}
\end{equation}
The family $\left(\Enu\right)_{k \in \N_0}$ is the canonical Schauder basis of $\ell^1_\nu$, normalized so that $\Vert \Enu\Vert_\nu=1$ for all $k$.
The Banach space $\ell^1_\nu$ is in fact a Banach algebra, with the product inherited from the product of functions.
\begin{lemma}\label{lem : banach algebra}
    Let $\nu\geq 1$. For all $U, V \in \ell^1_\nu$, we define their discrete convolution product $U\ast V \bydef \cG_0^{-1} \left(\cG_0(U) \cG_0(V) \right)$. This discrete convolution product can be computed as follows:
    \begin{equation*}
        (U \ast V)_n \bydef \sum_{n' \in\Z} U_{\vert n'\vert}\,V_{\vert n-n'\vert},\qquad \mbox{ for all }n\in\N_0,
    \end{equation*}
    and satisfies
    \begin{equation*}
      \|U\ast V\|_{\nu} \leq \|U\|_{\nu} \|V\|_{\nu}.
    \end{equation*}
\end{lemma}
\begin{remark}
    Similarly, one can define a discrete convolution product associated to each Gegenbauer expansion, \textit{i.e.}, an operation $\ast_k$ such that $U\ast_k V = \cG_k^{-1} \left(\cG_k(U) \cG_k(V) \right)$, which, with a suitable choice of norm or of normalization $\Omega_k$, also yields a Banach algebra structure~\cite{Bre23,cadiot_MMT,Szw05}. 
    In this work, we use general Gegenbauer expansions in order to derive formula for linear differential operator in sequence space, but all the nonlinear operations will apply to Chebyshev series only, which is why such generalized convolution products are not needed in this paper.
\end{remark}
Also note that, for all $\nu\geq 1$, the $\ell^1_\nu$ norm of the coefficients controls the supremum norm of the function: for all $U\in\ell^1_\nu$,
\begin{equation} \label{eq:C0VSell1}
    \Vert \cG_0(U) \Vert_\infty \bydef \max_{x\in[-1,1]} \vert \cG_0(U)(x) \vert  \leq \Vert U\Vert _\nu,
\end{equation}
since $\vert G_n^{(0)}(x)\vert\leq 1$ for all $x\in[-1,1]$.

Finally, we introduce projection operators, which will naturally appear when we use finite dimensional objects to approximate infinite dimensional ones.

\begin{definition}\label{def:proj}
 For any $N \in \mathbb{N}$, we introduce the projection operators $\pi^{\leq N} : \ell^1_\nu \to \ell^1_\nu$ and $\pi^{> N} : \ell^1_\nu \to \ell^1_\nu$ defined as
 \begin{align*}
    \left(\pi^{\leq N}(U)\right)_n  =  \left\{
   \begin{array}{ll}
          U_n,  & \mbox{if }\, 0 \leq n \leq N, \\
              0, & \mbox{otherwise},
              \end{array}
\right.
     \quad \mbox{ and } \quad
     \left(\pi^{> N}(U)\right)_n  =  \left\{
   \begin{array}{ll}
          0,  & \mbox{if }\, 0 \leq n \leq N, \\
              U_n, & \mbox{otherwise},
              \end{array}
\right.
 \end{align*}
 for all $U\in\ell^1_\nu$.
\end{definition}
We note that we frequently abuse notations when using these projections. For instance, we often identify elements of $\pi^{\leq N}\ell^1_\nu$ with finite vectors in $\R^{N+1}$ and vice versa, as well as linear operators on $\pi^{\leq N}\ell^1_\nu$ with finite matrices. Similarly, we identify linear operators on $\pi^{\leq N}\ell^1_\nu$ with their natural injection as linear operators on $\ell^1_\nu$. That is, given a linear operator $\cA_0:\pi^{\leq N}\ell^1_\nu\to\pi^{\leq N}\ell^1_\nu$, we use the same symbol $\cA_0$ for the operator $\ell^1_\nu\to\ell^1_\nu$ which is defined, for all $u\in\ell^1_\nu$, by
\begin{align*}
    \cA_0 u = \cA_0 \pi^{\leq N} u \in \pi^{\leq N}\ell^1_\nu \subset \ell^1_\nu.
\end{align*}

\subsection{Differential and boundary operators in coefficient spaces}\label{subsec:diffop}



Lemma~\ref{lem : formulas change of basis and derivatives} 
shows that Gegenbauer polynomials form a natural family of polynomials for the study of differential equations. Indeed, similarly as for Fourier series, identity~\eqref{eq : derivative}  provides a simple representation of derivative operators in coefficient spaces. These operators are diagonals between the families $(G^{(k)}_n)_n$ and $(G^{(k+1)}_n)_n$ (up to a shift index). This property is of major importance for numerical investigations as well as for the computer-assisted analysis of this manuscript (although we note that slightly more complicated structures, e.g., tridiagonal operators in coefficient space, can also sometimes be handled~\cite{BreDesLes15,BreChu25, cadiot2025recentadvancesrigorousintegration}). Moreover, we saw in the previous section that we can define a discrete convolution in coefficient space for the representation of products, with corresponding Banach algebras $\ell^1_\nu$ in which (analytic) nonlinearities are particularly easy to control. We are therefore going to use $\ell^1_\nu$ in order to recast~\eqref{eq : stationary PDE}--\eqref{eq : BC stationary PDE} and study its solutions. 


For this purpose, we will look for a solution to~\eqref{eq : stationary PDE}--\eqref{eq : BC stationary PDE} of the form $\mathcal{G}_0(V)$, for an element $V\in  \ell^1_\nu$, where we recall that $G^{(0)}_n = T_n$ is nothing but the $n$-th Chebyshev polynomial of the first kind. For any $k\in\N$, the $k$-th derivate operator $\frac{d^k}{d{x^k}}$ has a somewhat complicated expression when written from the basis $(G^{(0)}_n)_n$ to itself. However, according to~\eqref{eq : derivative}, $\frac{d^k}{d{x^k}}$ has a much simpler expression from the basis $(G^{(0)}_n)_n$ to the basis $(G^{(k)}_n)_n$. That is, $\frac{d^k}{d{x^k}}$ has a natural Gegenbauer representation as an (unbounded) operator $\cD_k : \ell^1_\nu \to \ell^1_{\nu}$ given as
\begin{align}\label{def : derivative}
  \cD_k \bydef \mathcal{G}_k^{-1} \frac{d^k}{d{x^k}} \mathcal{G}_0 \quad \text{ where} \quad   (\cD_kU)_n = \begin{cases}
        2^{k}\,k!\, U_k &\text{ if } n=0,\\
        (n+k)\,2^{k-1}\,(k-1)!\, U_{n+k} &\text{ if } n \in \mathbb{N}.
    \end{cases}
\end{align}
In other words, given a sequence $U$ of Chebyshev coefficients corresponding to a function $u = \cG_0(U)$, $\cD_k U$ gives the coefficients of $\frac{d^k u}{d{x^k}}$, but in the basis $(G_n^{(k)})_{n}$. The operator $\cD_k$ already has a very simple expression, but in practice it is often more convenient to work with diagonal operators, so we prefer to consider a shift of  $\cD_k$. 
For this purpose, let us consider the linear operator $\Sigma : \ell^1_\nu \to \ell^1_\nu$ defined as 
\begin{align}\label{eq : shift operator}
    (\Sigma U)_n \bydef \begin{cases}
        0 &\text{ if } n = 0,\\
        U_{n-1} &\text{ if } n \in \mathbb{N}.
    \end{cases}
\end{align}
We note that $\Sigma^k \cD_k : \ell^1_\nu \to \ell^1_\nu$ is diagonal. Moreover, because of the shift $\Sigma^k$, the first $k$ rows of $\Sigma^k \cD_k$ are empty, and we will be able to use these rows to incorporate the boundary conditions associated to our differential operators, and hence to represent both the differential and boundary parts of the problem in sequence space using a single operator (see~\eqref{eq : deftL}).

We now introduce another important operator, which is the inverse of $\Sigma^k \cD_k$ on its range.
\begin{lemma}\label{lem : Dk dagger inverse of Dk}
    Let $k \in \mathbb{N}$ and let $\cD_k^\dagger : \ell^1_\nu \to \ell^1_\nu$ be defined as 
    \begin{align}\label{def : Dk dagger}
        (\cD_k^\dagger U)_n = \begin{cases}
            0 &\mbox{ if } \quad 0 \leq  n \leq k-1, \\
            \dfrac{U_k}{2^k \, k!} &\text{ if } \quad n=k,\\
            \dfrac{U_n}{n 2^{k-1} \, (k-1)!} &\text{ if } \quad n \geq k+1.
        \end{cases}
    \end{align}
    Then, we have
    \begin{equation}\label{eq : identity D Ddag}
       \left(\Sigma^k \cD_k\right) \cD_k^\dagger =  \cD_k^\dagger \left(\Sigma^k \cD_k \right) =  \pi^{>k-1} \qquad \text{ and } \qquad  \cD_k \cD^\dagger_k \Sigma^k = \cI,
    \end{equation}
    with $\pi^{>k-1}$ the projection introduced in Definition~\ref{def:proj}. 
\end{lemma}

\begin{proof}
    The proof is a direct consequence of the fact that $\Sigma^k \cD_k : \ell^1_\nu \to \ell^1_\nu$ is diagonal, combined with \eqref{def : derivative}.
\end{proof}

The above lemma will be particularly useful when inverting linear differential operator (cf. Section \ref{ssec : linear problem}).
The operators $\cD_k$ provide us with simple representation of derivatives acting on Gegenbauer series, but they map between different Gegenbauer bases. When considering a PDE which contains differential operators of different orders, it is therefore helpful to also introduce operators allowing to move from one basis to another, with the goal of ultimately recasting the whole problem using single and common basis. Therefore, for all $k\in\N_0$, we define $\mathcal{C}_k \bydef \mathcal{G}_{k+1}^{-1}\mathcal{G}_k : \ell^1_{\nu} \to \ell^1_{\nu}$ as the change of basis operator mapping a sequence $U$ in the basis $(G^{(k)}_n)_n$ to its representation in the basis $(G^{(k+1)}_n)_n$. According to~\eqref{eq : change of basis}, we have
\begin{align}\label{def : C0}
    (\mathcal{C}_0 U)_n \bydef \begin{cases}
        U_0 - U_2 &\text{ if } n = 0,\\
      \dfrac{1}{2}\left(U_n - U_{n+2}\right) &\text{ if } n \in \mathbb{N},
    \end{cases}
\end{align}
and
\begin{align}\label{def : change of basis before shift}
    (\mathcal{C}_k U)_n \bydef \begin{cases}
        U_0 - \dfrac{2k}{2+k}U_2 &\text{ if } n = 0,\\
      \dfrac{k}{n+k}U_n - \dfrac{k}{n+k+2}U_{n+2} &\text{ if } n \in \mathbb{N},
    \end{cases}
\end{align}
for all $k \in \mathbb{N}$, and we note that $\cC_k$ is indeed a well defined and bounded operator on $\ell^1_\nu$. Moreover, for $k\leq l$ we also introduce $\mathcal{C}_{k,l} : \ell^1_{\nu} \to \ell^1_\nu$ the change of basis from $(G^{(k)}_n)_n$ to $(G^{(l)}_n)_n$, which can be expressed as follows
\begin{equation}\label{def : change of basis order k}
    \mathcal{C}_{k,l} \bydef \cG_{l}^{-1} \cG_k = \mathcal{C}_{l-1}\mathcal{C}_{l-2} \dots \mathcal{C}_{k+1}\mathcal{C}_{k},
\end{equation}
the last formula being valid only if $k < l$. An important point to note is that  $\Sigma\cC_k$ is tridiagonal, and more generally that $\Sigma^{l-k}\cC_{k,l}$ has bandwidth $(l-k)$. This property will prove useful in several estimates to come. 

Next, we introduce a representation in the Chebyshev basis $(G^{(0)}_n)_n$ of antiderivative operators. Let ${\mathcal{S}} : \ell^1_\nu \to \ell^1_\nu$ given by
\begin{align}\label{def : antideriva}
    (\mathcal{S}U)_n \bydef \begin{cases}
    U_0 - \dfrac12 U_1 + \displaystyle\sum_{n=2}^\infty (-1)^{n+1} \dfrac{U_n}{n^2-1} &\text{ if } n = 0\\
        \dfrac{1}{2n}\left(U_{n-1}-U_{n+1}\right) &\text{ if } n \in \mathbb{N},
    \end{cases} 
\end{align}
and, for all $i \in \N_0$, 
\begin{align}\label{def : Si}
    \cS^{(i)} \bydef \pi^{>i-1}{\cS}^i :\ell^1_\nu\to\ell^1_\nu.
\end{align}
Thanks to Lemma~\ref{lem : formulas change of basis and derivatives}, one can readily see that
\begin{align*}
      {\mathcal{S}} = \mathcal{G}_0^{-1} \int_{-1}^x\mathcal{G}_0,
\end{align*}
i.e., $\mathcal{S}$ provides the coefficients representation of the antiderivative which vanishes at $-1$. In turn, $\cS^{(i)}$ corresponds to taking $i$ antiderivatives, but this time with the constants fixed so that the $i$ first Chebyshev coefficients are equal to $0$, because of the incorporation of the projection $\pi^{>i-1}$. Unsurprisingly, $\cS^{(i)}$ can also be expressed in terms of the operator $\cD_i^\dag$. Moreover, while $\cS^{(i)}\cS^{(j)}$ is not equal to $\cS^{(i+j)}$ in general (because the terms of order less that $i+j$ do not necessarily match), this becomes true when working only with high enough Chebyshev modes. The previous two comments are made precise in the statement below.
\begin{lemma}
\label{lem : cSi}
    For all $i\in\N$,
    \begin{align}\label{eq : identity for Si}
    \cS^{(i)} = \mathcal{D}^{\dagger}_i  \Sigma^i\mathcal{C}_{0,i}.
\end{align}
Moreover, for all $i, j \in \mathbb{N}$,
\begin{align}\label{eq : identity of S}
    \mathcal{S}^{(i)}\mathcal{S}^{(j)} \pi^{>2(i+j)-1} = \mathcal{S}^{(i+j)}\pi^{>2(i+j)-1}.
\end{align}
\end{lemma}
\begin{proof}
Let $W\in\ell^1_\nu$, $U=\cS^{(i)}W$, and consider the function representations $w=\cG_0(W)$ and $u=\cG_0(U)$. By construction, $u$ is the unique solution to
    \begin{align*}
    \begin{cases}
     \partial^{i}_x \,u = w\\
    \int_{-1}^1 u(x) G_n^{(0)}(x) \omega_0(x) \, dx = 0 \quad\text{ for all } n \in \{0, \dots, i-1\}. 
    \end{cases}
    \end{align*}
In sequence space, $U$ is therefore the unique solution to 
\begin{align*}
    \begin{cases}
     \cG_i \cD_i U = \cG_{0} W,\\
    \pi^{\leq i-1} U = 0, 
    \end{cases}
    \end{align*}
and we simply have to check that $\tilde{U}=\mathcal{D}^{\dagger}_i  \Sigma^i\mathcal{C}_{0,i} W$ satisfies these equations. The fact that $\pi^{\leq i-1} \tilde{U} = 0$ is immediate since the range of $\mathcal{D}^{\dagger}_i$ is included in $\pi^{>i-1}\ell^1_\nu$, and identity~\eqref{eq : identity D Ddag} and definition~\eqref{def : change of basis order k} directly yield
\begin{equation*}
    \cG_i \cD_i\tilde{U} = \cG_i\cC_{0,i}W = \cG_0W.
\end{equation*}

We now proceed to prove \eqref{eq : identity of S}. First, recall that $\cS^{(i)}$ amounts to taking $i$ antiderivatives, hence $\cS^{(i)}\cS^{(j)}$ does amount to taking $i+j$ antiderivatives, and may only differ from $\cS^{(i+j)}$ in the first $i+j$ modes. That is, we have
\begin{align}
\label{eq:identitytoprove}
    \pi^{>i+j-1}\cS^{(i)}\cS^{(j)} = \pi^{>i+j-1}\cS^{(i+j)}.
\end{align}
In order to rigorously establish~\eqref{eq:identitytoprove}, we use the fact that $\cS^{(i)}$ has bandwidth $i$, therefore
\begin{align*}
    \pi^{>i+j-1}\cS^{i} = \pi^{>i+j-1}\cS^{(i)} = \pi^{>i+j-1}\cS^{(i)} \pi^{>j-1} = \pi^{>i+j-1}\cS^{i} \pi^{>j-1},
\end{align*}
which yields
\begin{align*}
    \pi^{>i+j-1}\cS^{(i)}\cS^{(j)} &= \pi^{>i+j-1}\cS^{i}\pi^{>j-1}\cS^{j} \\
    &= \pi^{>i+j-1}\cS^{i}\cS^{j}\\
    &= \pi^{>i+j-1}\cS^{i+j} \\
    &= \pi^{>i+j-1}\cS^{(i+j)}.
\end{align*}
Using again several times the fact that $\cS^{(k)}$ has bandwidth $k$ for all $k\in\N$, we then get
\begin{align*}
    \mathcal{S}^{(i)}\mathcal{S}^{(j)} \pi^{>2(i+j)-1} &= \pi^{>i+j-1}\mathcal{S}^{(i)}\mathcal{S}^{(j)} \pi^{>2(i+j)-1}  \\
    &\stackrel{\eqref{eq:identitytoprove}}{=} \pi^{>i+j-1}\mathcal{S}^{(i+j)} \pi^{>2(i+j)-1} \\
    &= \mathcal{S}^{(i+j)} \pi^{>2(i+j)-1},
\end{align*}
which concludes the proof.
\end{proof}

Finally, let us describe how boundary conditions of the form~\eqref{eq : BC stationary PDE} can be rewritten in sequence space. Taking $U\in \ell^1_\nu$, $j \in \{1, \dots , 2m\}$, and using~\eqref{eq : derivative} and~\eqref{eq : evaluation at 1 or -1}, we get that $\partial_x^j u(1)$ is given by
\begin{align}\label{def : boundary condition at 1}
\cB_1(U,j) \bydef   2^{j}j! U_{j} +    2^{j}(j-1)! \sum_{n=1}^\infty  \dfrac{(n+j)\Gamma(2j+n)}{\Gamma(2j) n!} U_{n+j}.
\end{align}
Similarly, $\partial_x^j u(-1)$ is given by
\begin{align*}
\cB_{-1}(U,j) \bydef  2^{j}j! U_{j} +    2^{j}(j-1)! \sum_{n=1}^\infty (-1)^n \dfrac{(n+j)\Gamma(2j+n)}{\Gamma(2j) n!} U_{n+j}.
\end{align*}
For the case $j=0$, we obtain 
\begin{align}\label{def : boundary zero}
\cB_{1}(U,0) \bydef   U_{0} +   2  \sum_{n=1}^\infty U_{n} \quad \text{ and } \quad \cB_{-1}(U,0) \bydef   U_{0} +   2  \sum_{n=1}^\infty (-1)^n U_{n}. 
\end{align}
For simplification, we define $\alpha_{j,n}$ as 
\begin{align*}
    \alpha_{j,n} = \begin{cases}
        2^{j-1}j! &\text{ if } n=0,\\
        2^{j-1}(j-1)! \dfrac{(n+j)\Gamma(2j+n)}{\Gamma(2j) n!} &\text{ if } n \geq 1,
    \end{cases}
\end{align*}
for all $n \in \mathbb{N}_0$ and all $j \geq 1$, and we set $\alpha_{0,n} =1$ for all $n \in \mathbb{N}_0$. The expressions for $\cB_{\pm1}(U,j)$ then simplifies to
\begin{align*}
    \cB_{1}(U,j) = 2 \sum_{n=0}^\infty \alpha_{j,n}U_{n+j} \quad \mbox{ and } \quad  \cB_{-1}(U,j) = 2 \sum_{n=0}^\infty (-1)^n \alpha_{j,n}U_{n+j}, \quad  \forall j \geq 1.
\end{align*}
We also note that, since $\cB_{\pm1}(U,j) = \partial_x^ju(\pm 1)$, with $u=\cG_0 U$, we can write the term $\cB_{\pm1}(U,j)$ as $\cB_{\pm1}(\cG_0^{-1} \partial_x^j u,0)$, so that
\begin{align}
\label{eq:identityBj}
    \cB_{\pm1}(U,j) = \cB_{\pm1}(\cG_0^{-1} \cG_j \cD_j U,0).
\end{align} 

Finally, we define $\cB : \ell^1_{\nu} \to \ell^1_{\nu}$ as the linear operator given by
\begin{equation}\label{def : boundary condition}
    \begin{aligned}
    (\cB(U))_n =  \begin{cases}
       \displaystyle \sum_{j=0}^{2m-1} \beta_{j,n} \, \cB_{-1}(U,j) &\text{ if } n = 0, \dots, m-1,\\
         \displaystyle \sum_{j=0}^{2m-1} \gamma_{j,n-m} \, \cB_{1}(U,j) &\text{ if } n = m, \dots, 2m-1,\\
        0 &\text{ if } n \geq 2m,
    \end{cases}
\end{aligned}
\end{equation}
which encodes the boundary conditions~\eqref{eq : BC stationary PDE} of the problem.


\subsection{Linear Problem}\label{ssec : linear problem}

We now use the objects introduced up to now in Section~\ref{sec : gegenbauer} in order to write and solve linear boundary value problems in sequence space. More specifically, given $W\in\ell^1_\nu$ and $w=\cG_0(W)$, we are interested in solutions $v$ of the following problem: 
\begin{align}\label{eq : linear problem}
     \begin{cases}
     (-1)^{m+1}\partial_x^{2m} v = w \quad \text{ on }(-1,1)\\
         \displaystyle  \sum_{j=0}^{2m-1} \beta_{j,i} \, \partial^j_x\,v(-1) = 0 &\quad\text{ for all } i \in \{1, 2, \dots, m\},\\
        \displaystyle  \sum_{j=0}^{2m-1} \gamma_{j,i}\, \partial^j_x\,v(1) = 0 &\quad\text{ for all } i \in \{1, 2, \dots, m\}.
    \end{cases}
\end{align}
This linear problem will play a central role both in the analysis of the nonlinear steady state equation~\eqref{eq : stationary PDE}, and in the study of the stability of equilibria.
\begin{remark}
    While it is often more common to work with positive operators, i.e., with $(-1)^{m}\partial_x^{2m}$, we focus on $(-1)^{m+1}\partial_x^{2m}$ here because this is the operator appearing in the linearization, when studying stability.
\end{remark}
It is clear from a functional-analytic point of view that the above problem may or may not be well posed, depending on the boundary conditions. This can also easily be seen in sequence space, and we provide below an explicit criterion, that can be checked in practice, to ensure well-posedness. If this criterion is satisfied, we then provide a computable formula in sequence space for the inverse of the linear operator plus boundary conditions.
\begin{remark}
    In practice, we may have to restrict to a suitable subspace or to add extra symmetry constraints in order to recover a well posed problem. This is well known, for instance when using homogeneous Neumann boundary conditions. We will not discuss this example further here, but it can easily be dealt with within our sequence space framework, as will be illustrated in Section~\ref{ssec:KS}.
\end{remark}

\begin{proposition}
\label{prop:cLinv}
Let $\nu \geq 1$, $W\in\ell^1_\nu$, $w=\cG_0(W)$, and $\cB$ the operator representing the boundary conditions defined in~\eqref{def : boundary condition}. There exists a unique $V\in\ell^1_\nu$ such that $v=\cG_0(V)$ solves the linear problem~\eqref{eq : linear problem} if and only if the operator $\pi^{\leq 2m-1}\cB\pi^{\leq 2m-1} : \pi^{\leq 2m-1} \ell^1_{\nu} \to \pi^{\leq 2m-1} \ell^1_{\nu}$ is invertible. In that case, denoting $\cB^{\dagger}_{2m} : \pi^{\leq 2m-1} \ell^1_{\nu} \to \pi^{\leq 2m-1} \ell^1_{\nu}$  the inverse of $\pi^{\leq 2m-1}\cB\pi^{\leq 2m-1}$, the operator
\begin{align}\label{eq : defLinv}
    \cL^{-1} \bydef (-1)^{m+1} \left(\cI - \cB_{2m}^\dag \cB \right)\cD_{2m}^\dag\Sigma^{2m}\mathcal{C}_{0,2m} = (-1)^{m+1} \left(\cI - \cB_{2m}^\dag \cB \right)\cS^{(2m)}
\end{align}
is well defined and bounded on $\ell^1_\nu$, and $V = \cL^{-1}W$.
\end{proposition}
\begin{proof}
Let us first note that any element $V$ in the kernel of $\pi^{\leq 2m-1}\cB\pi^{\leq 2m-1}$ corresponds to a polynomial function $v=\cG_0(V)$ of degree strictly less than $2m$ which satisfies the boundary conditions of~\eqref{eq : linear problem}. Therefore, as soon as $\pi^{\leq 2m-1}\cB\pi^{\leq 2m-1}$ is not invertible, the linear system~\eqref{eq : linear problem} cannot have a unique solution.
We henceforth assume that $\pi^{\leq 2m-1}\cB\pi^{\leq 2m-1}$ is invertible, and study the linear system~\eqref{eq : linear problem} in sequence space. We first present some somewhat formal calculations leading to the definition of $\cL^{-1}$, and then prove that the resulting operator is actually well defined and solves~\eqref{eq : linear problem}.

Using the notation of Section \ref{subsec:diffop}, the equation $(-1)^{m+1}\partial_x^{2m} v = w$ rewrites
\begin{align*}
   (-1)^{m+1}\mathcal{D}_{2m}V = \mathcal{C}_{0,2m}W,
\end{align*}
and thanks to the injectivity of the shift operator $\Sigma$ defined in~\eqref{eq : shift operator}, we equivalently have 
\begin{align*}
   (-1)^{m+1}\Sigma^{2m}\mathcal{D}_{2m}V = \Sigma^{2m}\mathcal{C}_{0,2m}W.
\end{align*}
Moreover, since the boundary conditions of~\eqref{eq : linear problem} are equivalent to $\cB(V) = 0$, and only the $2m$ first rows of $\cB$ are nontrivial, we introduce the unbounded operator $\tcL :\ell^1_\nu\to\ell^1_\nu$ given by
\begin{align}\label{eq : deftL}
\tcL \bydef \cB + (-1)^{m+1}\Sigma^{2m}\mathcal{D}_{2m},
\end{align}
which allows us to rewrite the problem~\eqref{eq : linear problem} in sequence space as
\begin{align}\label{eq : linear problem sequence space}
    \tcL V = \Sigma^{2m}\mathcal{C}_{0,2m}W.
\end{align}


Here and in the sequel, it may be useful to think of $\tcL$ as an infinite-matrix associated to the canonical Schauder basis of $\ell^1_\nu$. In particular, using the projection operator introduced in Definition~\ref{def:proj}, we can write
\begin{align*}
    \tcL = \pi^{\leq 2m-1} \, \cB \, \pi^{\leq 2m-1} + \pi^{\leq 2m-1} \, \cB \, \pi^{>2m-1} + (-1)^{m+1}\Sigma^{2m} \, \cD_{2m},
\end{align*}
where in fact $\Sigma^{2m} \, \cD_{2m} = \pi^{>2m-1}\, \Sigma^{2m} \, \cD_{2m} \, \pi^{>2m-1}$. In other words, $\tcL$ is block upper triangular:
\renewcommand{\arraystretch}{2}
\begin{align*}
\tcL & = 
    \left(\begin{array}{c|c}
         \pi^{\leq 2m-1} \, \tcL \, \pi^{\leq 2m-1} & \pi^{\leq 2m-1} \, \tcL \, \pi^{> 2m-1} \\\hline
        \pi^{> 2m-1} \, \tcL\, \pi^{\leq 2m-1} & \pi^{> 2m-1} \, \tcL\, \pi^{> 2m-1}
    \end{array}\right) \\
    & = 
    \left(\begin{array}{c|c}
         \pi^{\leq 2m-1} \, \cB \, \pi^{\leq 2m-1} & \pi^{\leq 2m-1} \, \cB \, \pi^{>2m-1} \\\hline
        0 & \pi^{>2m-1}\, (-1)^{m+1}\Sigma^{2m} \, \cD_{2m} \, \pi^{>2m-1}
    \end{array}\right).
\end{align*}
Here and in the sequel, we recall that we slightly abuse notation and identify, e.g., the operator $\pi^{\leq 2m-1} \, \cB \pi^{\leq 2m-1} :\ell^1_\nu\to\ell^1_\nu$ and its restriction from $\pi^{\leq 2m-1}\ell^1_\nu$ to $\pi^{\leq 2m-1}\ell^1_\nu$, which can be described as an $2m\times 2m$ matrix. Since $\pi^{>2m-1}\, \Sigma^{2m} \, \cD_{2m} \, \pi^{>2m-1}$ is invertible according to Lemma~\ref{lem : Dk dagger inverse of Dk}, and we assumed $\pi^{\leq 2m-1} \, \cB \, \pi^{\leq 2m-1}$ to be invertible, one can formally invert $\tcL$ thanks to its block-triangular structure, and we get
\begin{equation*}
 \tcL^{-1} =   \left(\begin{array}{c|c}
         \cB^\dag_{2m} & -(-1)^{m+1}\cB^\dag_{2m}\pi^{\leq 2m-1}\cB \pi^{>2m-1}\cD^\dag_{2m}\pi^{>2m-1} \\\hline
        0 & \pi^{>2m-1}\, (-1)^{m+1}\cD^\dag_{2m} \, \pi^{>2m-1}
    \end{array}\right) ,
\end{equation*}
or, in a more condensed from:
\begin{align}\label{eq : formula inverse of L}
        \tcL^{-1} = \cB^{\dagger}_{2m} -  (-1)^{m+1}\cB^{\dagger}_{2m} \cB \mathcal{D}_{2m}^{\dagger} + (-1)^{m+1}\mathcal{D}_{2m}^{\dagger},
    \end{align}
which is a particular case of Schur's complement. 

Coming back to~\eqref{eq : linear problem sequence space}, and using that $\cB^{\dagger}_{2m} \Sigma^{2m}=0$, we get
\begin{align*}
    V = \tcL^{-1}\Sigma^{2m}\mathcal{C}_{0,2m} W
= (-1)^{m+1} \left(\cI - \cB_{2m}^\dag \cB \right)\cD_{2m}^\dag\Sigma^{2m}\mathcal{C}_{0,2m} W,
\end{align*}
which yields formula~\eqref{eq : defLinv}. 

Note that the only potentially unbounded operator involved in the definition of $\cL^{-1}$ is $\cB$. Therefore, recalling~\eqref{def : boundary condition}, it suffices to prove that $W\mapsto\cB_{\pm 1} \left(\cD_{2m}^\dag\Sigma^{2m}\mathcal{C}_{0,2m}W,j\right)$ is a well defined linear form on $\ell^1_\nu$ for all $j\in\{0,\ldots,2m-1\}$, in order to check that $\cL^{-1}$ is indeed a bounded operator on $\ell^1_\nu$. This is in fact an immediate consequence of Lemma~\ref{lem : estimation BC} below. 

We now denote $V=\cL^{-1}W$, and we prove that $v=\cG_0(V)$ indeed solves~\eqref{eq : linear problem}. First, observe that $\cB(\cI - \cB_{2m}^\dag \cB)=0$, and hence $\cB(V) = 0$, which means that $v$ satisfies the boundary conditions of~\eqref{eq : linear problem}. It remains to check that $(-1)^{m+1}\partial_x^{2m}v = w$. Note that $\cD_{2m}\pi^{\leq 2m-1} = 0$, hence $\cD_{2m}\cB_{2m}^\dag=0$, and using~\eqref{eq : identity D Ddag} we then get
\begin{align*}
    \cD_{2m}\cL^{-1} = (-1)^{m+1}\cD_{2m}\cD_{2m}^\dag\Sigma^{2m}\mathcal{C}_{0,2m} = (-1)^{m+1} \cC_{0,2m},
\end{align*}
which indeed yields
\begin{align*}
    (-1)^{m+1}\partial_x^{2m}v &= (-1)^{m+1}\cG_{2m} \cD_{2m}V \\
    & = (-1)^{m+1}\cG_{2m} \cD_{2m} \cL^{-1}W \\
    & = \cG_{2m} \cC_{0,2m} W \\
    & = \cG_0 W \\
    &= w.
\end{align*}


Obviously, if $\pi^{\leq 2m-1}\cB\pi^{\leq 2m-1}$ is invertible and $\tilde{V}\in\ell^1_\nu$ is such that $\cG_0(\tilde{V})$ also solves~\eqref{eq : linear problem}, then $\tilde V = V$. Indeed, under that assumption $\cB(V-\tilde{V}) = 0$ implies $\pi^{\leq 2m-1} (V-\tilde{V}) = 0$, and $\Sigma^{2m}\mathcal{D}_{2m} (V-\tilde{V}) = 0$ always implies $\pi^{> 2m-1} (V-\tilde{V}) = 0$.
\end{proof}
\begin{remark}\label{rem : inverse of B}
    In practice, when we want to apply Proposition~\eqref{prop:cLinv}, we use rigorous numerics to prove the invertibility of the (finite) matrix $\pi^{\leq 2m-1}\cB\pi^{\leq 2m-1}$ and to compute $\cB^\dagger_{2m}$. 
    Furthermore, note that the expression of $\mathcal{L}^{-1}$ is obviously independent of how the $2m$ first rows of $\mathcal{B}$ (the boundary conditions) are ordered. 
\end{remark}

As expected, Proposition~\ref{prop:cLinv} shows that $\cL^{-1}$ essentially amounts to taking $2m$ antiderivatives, with finite range modifications incorporated in $\cI - \cB_{2m}^\dag \cB$ to account for the boundary conditions. The main benefit of Proposition~\ref{prop:cLinv} is to give a general construction of the solution operator $\cL^{-1}$ of~\eqref{eq : linear problem} in sequence space, which can be used for all orders $m$ and any boundary conditions of the form~\eqref{def : boundary condition}. Moreover, formula~\eqref{eq : defLinv} allows to better describe the structure of $\cL^{-1}$. It can be split into two parts, namely $(-1)^{m+1} \cS^{(2m)}$, which is a banded operator of bandwidth $2m$, and $(-1)^{m} \cB_{2m}^\dag\cB\cS^{(2m)}$, whose range is contained in $\pi^{\leq 2m-1}\ell^1_\nu$, and therefore corresponds to $2m$ \emph{infinite rows}. These structural properties of $\cL^{-1}$ will prove helpful in several estimates to come. 

We end this section with two lemma giving quantitative compactness estimates for $\cS^{(i)}$, and thus for $\cL^{-1}$, which will prove instrumental in Section~\ref{ssec : compactness}. We recall that $\Enu$ are the normalized elements of the Schauder basis of $\ell^1_\nu$ introduced in~\eqref{def:Ek}, and refer to Section~\ref{subsec:diffop} for the definitions of the operators $\cB_{\pm 1}$.
%
%
%


\begin{lemma}\label{lem : compactness S}
Let $\nu \geq 1$, let $i \in \mathbb{N}$, let $k \geq 2i$, and let
    \begin{align}\label{def : gamma i k}
\gamma_{i,k}^{(\nu)} \bydef 
        \begin{cases}
            \nu^i \displaystyle\prod_{j=1}^{p} \frac{1}{k^2 - (2j-1)^2} &\text{ if } i = 2p \quad \text{ for some } p \in \mathbb{N},\\
            \nu^i k \,\displaystyle\prod_{j=0}^{p} \frac{1}{k^2 - (2j)^2} &\text{ if } i = 2p+1 \quad \text{ for some } p \in \mathbb{N}_0.
        \end{cases}
    \end{align}
    Then, $\|\mathcal{S}^{(i)}\Enu\|_{\nu} \leq \gamma_{i,k}^{(\nu)}$.
\end{lemma}

\begin{proof} 
Since $k \geq 2i$, we can use \eqref{eq : identity of S} and get
    \begin{align}\label{eq : lem S step 0}
        \|\mathcal{S}^{(i)}\Enu\|_\nu = \|\mathcal{S}^{(i-1)} \mathcal{S}^{(1)}\Enu\|_\nu.
    \end{align}
Now, using the definition~\eqref{def : antideriva} of $\cS$, we obtain
\begin{align}\label{eq : application to unit vec}
    \mathcal{S}^{(1)}\Enu = \frac{1}{2 k \nu} E^{(\nu)}_{k-1} - \frac{\nu}{2 k} E^{(\nu)}_{k+1},
\end{align}
so that
\begin{align}
\label{eq:SiEkrec}
    \|\mathcal{S}^{(i)}\Enu\|_\nu \leq \frac{\nu}{2k}\left( \|\mathcal{S}^{(i-1)} E^{(\nu)}_{k-1}\|_\nu + \|\mathcal{S}^{(i-1)} E^{(\nu)}_{k+1}\|_\nu\right).
\end{align}
Hence, in the sequel we argue by induction. If $i=1$, we get, thanks to~\eqref{eq : application to unit vec}, the following estimate
\begin{align*}
    \|\mathcal{S}^{(1)}\Enu\|_\nu = \frac{1}{2 k \nu} + \frac{\nu}{2 k} \leq \frac{\nu}{k} = \gamma^{(\nu)}_{1,k} \quad \forall k \geq 2.
\end{align*}
Now, let $i \in \mathbb{N}$ and suppose that 
\begin{align*}
    \|\mathcal{S}^{(i)} \Enu\|_\nu \leq \gamma^{(\nu)}_{i,k}
\end{align*}
for all $k \geq 2i$. Then, using~\eqref{eq:SiEkrec} and taking $k\geq 2(i+1)$, we get
\begin{align*}
    \|\mathcal{S}^{(i+1)}\Enu\|_{\nu} \leq \frac{\nu}{2k}\left( \gamma^{(\nu)}_{i,k-1} + \gamma^{(\nu)}_{i,k+1} \right). 
\end{align*}
Suppose first that $i = 2p$ for some $p \in \mathbb{N}$. Then, using \eqref{def : gamma i k}, we have that
\begin{align*}
        \|\mathcal{S}^{(i+1)}\Enu\|_{\nu} &\leq \frac{\nu^{i+1}}{2k} \prod_{j=1}^p \left( \frac{1}{(k-1)^2-(2j-1)^2} + \frac{1}{(k+1)^2-(2j-1)^2} \right)\\
        &= \frac{\nu^{i+1}}{2k} \prod_{j=1}^p \left( \frac{1}{(k- 2j)(k +2(j-1))} + \frac{1}{(k+2j)(k-2(j-1))} \right).
\end{align*}
Hence, reordering the terms, we obtain
\begin{align*}
    \|\mathcal{S}^{(i+1)}\Enu\|_{\nu} \leq \frac{\nu^{i+1}}{2} \, \frac{\gamma^{(\nu)}_{i-1,k}}{\nu^{i-1}} \left( \frac{1}{k-2p} +\frac{1}{k+2p} \right) = \nu^{i+1} \, \frac{\gamma^{(\nu)}_{i-1,k}}{\nu^{i-1}} \, \frac{k}{k^2-(2p)^2} =  \gamma^{(\nu)}_{i+1,k}.
\end{align*}
We argue similarly if $i$ is odd. Therefore, we conclude the proof by induction.
\end{proof}

\begin{lemma}\label{lem : estimation BC} 
   Let $\nu \geq 1$, let $j \in \mathbb{N}$ and let $i > j$. Then, for all $k \geq 2i$, we have
    \begin{align*}
        | \cB_{\pm 1}(\cS^{(i)} \Enu,j) | \leq \nu^{-k} \gamma_{i-j,k}^{(1)},
    \end{align*}
    where $\gamma_{i-j,k}^{(1)}$ is defined in \eqref{def : gamma i k}.
\end{lemma}
\begin{proof}
    According to~\eqref{eq:identityBj},
\begin{align*}
    \cB_{\pm 1}(\cS^{(i)} \Enu,j) = \cB_{\pm 1}(\cG_0^{-1}\cG_j \cD_j\cS^{(i)} \Enu,0).
\end{align*}
Since $k \geq 2i$, we can use~\eqref{eq : identity of S} to rewrite
\begin{align*}
    \mathcal{S}^{(i)}\Enu = \mathcal{S}^{(j)}\mathcal{S}^{(i-j)}\Enu,
\end{align*}
and therefore, thanks to~\eqref{eq : identity for Si} and~\eqref{eq : identity D Ddag},
\begin{align*}
    \cG_0^{-1}\cG_j \cD_j\cS^{(i)} \Enu & = \cG_0^{-1}\cG_j \cD_j \mathcal{S}^{(j)}\mathcal{S}^{(i-j)}\Enu \\
    &=\cG_0^{-1}\cG_j \cC_{0,j}\mathcal{S}^{(i-j)}\Enu \\
    & =\mathcal{S}^{(i-j)}\Enu. 
\end{align*}
In summary, we have just proven that, as one might have expected,
\begin{align*}
     \cB_{\pm 1}(\cS^{(i)} \Enu,j) = \mathcal{B}_{\pm 1}(\mathcal{S}^{(i-j)}\Enu,0).
\end{align*}
Finally, recalling \eqref{def : boundary zero} and using Lemma~\ref{lem : compactness S}, we have that 
\begin{equation*}
    |\mathcal{B}_1(\mathcal{S}^{(i-j)}\Enu,0)| \leq \|\mathcal{S}^{(i-j)}\Enu\|_{1} = \nu^{-k} \|\mathcal{S}^{(i-j)}E^{(1)}_k\|_{1} \leq \nu^{-k} \gamma_{i-j,k}^{(1)}. 
\end{equation*}
This concludes the proof of Lemma~\ref{lem : estimation BC}
\end{proof}
Recalling that $\cS^{(i)}$ corresponds to taking $i$ antiderivatives, Lemma~\ref{lem : compactness S} simply shows that we recover the expected $1/k^i$ decay, with explicit constants. It is then also not surprising that the operators $\cB_{\pm 1}(\cdot,j)$, which contain derivatives of order $j$, are compact when composed with $\cS^{(i)}$, with $i >j$. The main purpose of Lemma~\ref{lem : estimation BC} is to make such compactness quantitative. Explicit examples are derived in the Appendix \ref{App : explicit inverses} for the Dirichlet and the Neumann Laplacian operators.

\subsection{Definition of the zero finding problem}\label{ssec : zero finding}

We are now ready to define a zero finding problem, in coefficient space, whose zeros will correspond to solutions of~\eqref{eq : stationary PDE}--\eqref{eq : BC stationary PDE}. As is usually the case with computer-assisted proofs, there are in fact several zero-finding problems one could consider, which are mathematically equivalent but for which the analysis can differ significantly. Compared to more classical computer-assisted proofs using Fourier series, the differences are even more drastic in this work, mainly because of the impact of boundary conditions. We discuss several options below.

A natural approach would be to proceed as in Section \ref{ssec : linear problem} and to add the nonlinear term, which yields the zero-finding problem:
    \begin{align}\label{eq : F usual}
        \cF(V) = \tcL V + \Sigma^{2m} \cC_{0,2m} \cG_0^{-1}\cQ(\cG_0 V) + \Sigma^{2m} \cC_{0,2m}\Psi,
    \end{align}
with $\Psi =\cG_0^{-1}(\psi)$ Indeed, if $\cF(V)=0$ for some $V\in\ell^1_\nu$, then $v=\cG_0(V)$ solves~\eqref{eq : stationary PDE}--\eqref{eq : BC stationary PDE}. If we only cared about proving the existence of the steady state, but not about studying its stability, this would be a convenient formulation to consider. 

Assuming the linear problem~\eqref{eq : linear problem} is well-posed (this will be a standing assumption throughout, which can be checked in practice as shown in Proposition~\ref{prop:cLinv}) a mathematically equivalent problem would be to look for zeros of
    \begin{align}\label{eq : F alternatif}
        \cF(V) = V + \cL^{-1} \cG_0^{-1}\cQ(\cG_0 V) + \cL^{-1}\Psi,
    \end{align}
which is the sequence-space representation of the problem
\begin{equation}\label{eq : tildeF function}
    v + (-1)^{m+1}\left(\partial_x^{2m}\right)^{-1}\cQ\left( v\right) + (-1)^{m+1}\left(\partial_x^{2m}\right)^{-1}\psi = 0,
\end{equation}
where the inverse operator $(\partial_x^{2m})^{-1}$ incorporates the boundary conditions~\eqref{eq : BC stationary PDE}.
A key feature of the zero-finding problem~\eqref{eq : F alternatif} is that we expect its Fréchet derivative to be of the form ``identity + compact operator'', which is a standard setup for computer-assisted proofs, because $\cQ$ only contains derivatives of order strictly less than $2m$. However, getting quantitative and sharp compactness estimates might be tricky, because of the boundary conditions, which make the situation more subtle than when using Fourier series. Indeed, let us consider the sequence space representation of $(-1)^{m+1}\left(\partial_x^{2m}\right)^{-1} \partial_x^i$, i.e.,
\begin{equation*}
    \cL^{-1} \cG_0^{-1}\partial^i_x \cG_0 = \cL^{-1}(\cC_{0,i})^{-1}\cD_i = \tcL^{-1}\Sigma^{2m}\cC_{i,2m}\cD_i,
\end{equation*}
for $i\in\{0,\ldots,2m-1\}$. Although one might expect such operator to be compact on $\ell^1_\nu$, this may actually not be the case if $\nu=1$, and when $\nu>1$ the compactness estimates could be quite bad if $\nu$ is too close to $1$.
\begin{example} 
\label{ex : compactness}
Given $\psi$ odd, let us consider the following linear problem
\begin{equation}
\label{eq : ex1 compactness}
    \begin{cases}
        \partial_x^2 v = \psi \quad \text{on }(-1,1), \\
        \partial_x v(1) = 0, \text{ $v$ is odd.}
    \end{cases}
\end{equation}
Let $\mathcal{L}_1^{-1} : \ell^1_{\nu,o} \to \ell^1_{\nu,o}$ be the associated solution operator, which is explicitly given in \eqref{eq : inverse neumann}, and where $\ell^1_{\nu,o}$ is the restriction of $\ell^1_\nu$ to odd functions given in \eqref{def : odd and even restrictions}. Then, given $k \in \mathbb{N}$ sufficiently large, one can prove using \eqref{eq : inverse neumann} that $ \|\cL^{-1}_1 (\Enu)\|_\nu = \mathcal{O}(k^{-2}) + \mathcal{O}(k^{-2}\nu^{-k})$, but only that $\|\cL^{-1}_1 \cG_0^{-1}\partial_x \cG_0 (\Enu)\|_\nu =  \mathcal{O}(k^{-1}) + \mathcal{O}({\nu^{-k}})$. In particular, since $\|\Enu\|_\nu = 1$ for all $k \geq 1$, we do not get that $ \cL_1^{-1}\cG_0^{-1}\partial_x \cG_0 : \ell^1_{\nu,e} \to \ell^1_{\nu,o}$ is compact if $\nu = 1$, where $\ell^1_{\nu,e}$ is the restriction of $\ell^1_\nu$ to even functions \eqref{def : odd and even restrictions}. 

For higher order operators, the situation can be even worse: for instance, with $\cL^{-1}$ the solution operator associated to
\begin{equation}
\label{eq : ex2 compactness}
    \begin{cases}
         -\partial_x^4 v = \psi \quad \text{on }(-1,1), \\
        \partial_x v(\pm 1) = \partial_x^3 v(\pm 1 ) = 0,
    \end{cases}
\end{equation}
restricted to the subspace of odd functions to recover well-posedness, one gets that
\[
\|\cL^{-1} \cG_0^{-1}\partial^3_x \cG_0 (\Enu)\|_\nu =  \mathcal{O}(k^{-1}) + \mathcal{O}(k^3{\nu^{-k}}).
\]
In particular, even if this estimate guarantees that $\cL^{-1} \cG_0^{-1}\partial^3_x \cG_0$ is compact for any $\nu>1$, our computer-assisted proofs require quantitative estimates and would therefore become very computationally expensive, since $\|\cL^{-1} \cG_0^{-1}\partial^3_x \cG_0 (\Enu)\|_\nu$ only becomes small for very large values of $k$ (or for $\nu$ much larger than one, which is then detrimental to other estimates).
\end{example}

These difficult $\mathcal{O}({\nu^{-k}})$ or $\mathcal{O}(k^3{\nu^{-k}})$ terms only come from the boundary conditions, i.e., from the first rows of $\cL^{-1}$. Therefore, one could in fact recover better compactness estimates by using a suitably weighted $\ell^1$ space, e.g., with norm
\begin{equation*}
    \sum_{n\geq 0} \vert V_n\vert n^{2m} \nu^n,
\end{equation*}
but of course these faster growing weights make some other estimates worse.
We emphasize that we would face similar issues if $\cF$ was taken as in~\eqref{eq : F usual}, after introducing a suitable approximate inverse of $D\cF(V)$ (see Section~\ref{ssec : NK}) based on $\tcL^{-1}$. 

In order to circumvent these hurdles, we consider a slightly different zero-finding problem. Before actually moving to coefficient space, let us formally describe how we reformulate the problem in function space. Starting from~\eqref{eq : stationary PDE}, we make the change of unknown $v = (-1)^{m+1}(\partial_x^{2m})^{-1} u$, which yields the zero finding problem:
\begin{equation}\label{eq : F function}
    u + \cQ\left((-1)^{m+1}\left(\partial_x^{2m}\right)^{-1} u\right) + \psi = 0.
\end{equation}
If $u$ solves this equation, then $v = (-1)^{m+1}(\partial_x^{2m})^{-1} u$ does solve~\eqref{eq : stationary PDE}--\eqref{eq : BC stationary PDE}, and vice-versa.
Introducing
\begin{equation}\label{eq : K}
    \cK(U) \bydef \cG_0^{-1}\cQ(\cG_0 \cL^{-1} U),
\end{equation}
which represents the nonlinear terms in sequence space, we then consider the zero-finding problem
\begin{equation}\label{eq : F}
    \cF(U) \bydef U + \cK(U) + \Psi = 0,
\end{equation}
with $\Psi = \cG_0^{-1}(\psi)$ assumed to be in $\ell^1_\nu$, which is the coefficient-space formulation of~\eqref{eq : F function}. In other words, if $U\in\ell^1_\nu$ is a zero of $\cF$, then $v = \cG_0(\cL^{-1} U)$ solves~\eqref{eq : stationary PDE}--\eqref{eq : BC stationary PDE}. Reciprocally, if $\psi\in\cG_0\ell^1_\nu$ and $v$ solves~\eqref{eq : stationary PDE}--\eqref{eq : BC stationary PDE}, then $U = \cG_0^{-1}((-1)^{m+1}\partial_x^{2m}v)$ belongs to $\ell^1_\nu$ and is a zero of $\cF$.

The main difference with the $\cF$ considered previously in~\eqref{eq : F alternatif} is that $\cL^{-1}$ now acts on the other side, i.e., with~\eqref{eq : F} we are going to face operators of the form $\partial_x^i(\partial_x^{2m})^{-1}$,  instead of $(\partial_x^{2m})^{-1}\partial_x^i$ with~\eqref{eq : F alternatif}. This actually makes a crucial difference in sequence space, as the rows in $\cL^{-1}$ corresponding to the boundary conditions no longer get hit by the unbounded terms in $\cD_i$. As a consequence, the compactness estimates will be better behaved even if $\nu=1$.

\begin{example}
\label{ex : compactness2}
    Consider again the linear problem~\eqref{eq : ex1 compactness} from Example~\ref{ex : compactness}, and the associated  solution operator $\cL_1^{-1}$ given in \eqref{eq : inverse neumann}. With $\cL_1^{-1}$ acting on the right, we get $\|\cG_0^{-1}\partial_x \cG_0 \cL^{-1} (\Enu)\|_\nu = \mathcal{O}(k^{-1}) + \mathcal{O}(k^{-2} \,\nu^{-k})$. Therefore, $\cG_0^{-1}\partial_x \cG_0 \cL_1^{-1}:\ell^1_\nu\to\ell^1_\nu$ is compact even if $\nu=1$ (and taking $\nu$ close to $1$ is also not an issue).

    Similarly, for $\cL^{-1}$ associated to the problem~\eqref{eq : ex2 compactness}, $\cG_0^{-1}\partial^3_x \cG_0\cL^{-1}$ behaves much more favorably than $\cL^{-1}\cG_0^{-1}\partial^3_x \cG_0$, and we get $\| \cG_0^{-1}\partial^3_x \cG_0\cL^{-1} (\Enu)\|_\nu =  \mathcal{O}(k^{-1}) + \mathcal{O}(k^{-2} \, \nu^{-k})$.
 \end{example}

Because of these more favorable compactness estimates, the $\cF$ defined in~\eqref{eq : F} is the one we actually work with for the remainder of the paper. This choice also leads to more stable numerics when trying to numerically find an approximate zero of $\cF$. Once such an approximate zero has been obtained, our goal is to rigorously prove the existence of a nearby zero of $\cF$ in $\ell^1_\nu$, using a standard Newton-Kantorovich approach (see Section~\ref{ssec : NK}), and then to also study the stability of the corresponding steady states of~\eqref{eq : nonlinear part}. In order to establish the existence of zeros of $\cF$, we will need relatively sharp and fully computable compactness estimates on $D\cK(U)$, which essentially amounts to the type of estimates discussed in Example~\ref{ex : compactness2}. The derivation of such estimates is presented in Section~\ref{ssec : compactness}.

Before doing so, let us mention that the $\cF$ defined in~\eqref{eq : F} also turns out to be more convenient when studying the stability of the obtained steady state in Section~\ref{sec : stability}. Indeed, working in $\ell^1_1$ then becomes natural, because we know a priori that all the eigenvectors lie in $\ell^1_1$, but not necessarily in $\ell^1_\nu$ with $\nu > 1$.

\subsection{Compactness estimates for \texorpdfstring{$D\cK(\bU)$}{D\cK(\bU)}}
\label{ssec : compactness}

Looking back at the definition of $\cQ$, for any $U\in\ell^1_\nu$ we have
\begin{align*}
    \cK(U) &= \cG_0^{-1}\left((\cG_0\cL^{-1}U)^{j_0}\right) \ast \cG_0^{-1}\left((\partial_x \cG_0(\cL^{-1}U))^{j_1} \right) \ast \ldots \ast \cG_0^{-1}\left((\partial_x^{2m-1} \cG_0(\cL^{-1}U))^{j_{2m-1}} \right) \\
    &= \prod_{i=0}^{2m-1} \cG_0^{-1}\left((\partial_x^{i} \cG_0(\cL^{-1}U))^{j_{i}} \right),
\end{align*}
where the products in $\prod$ have to be understood as convolution products, as they apply to sequences of coefficients.
Therefore, fixing $\bU\in\ell^1_\nu$, the Fréchet derivative of $\cK$ at $\bU$ applied to $U$ writes
\begin{equation*}
    D\cK(\bU) U = \sum_{i=0}^{2m-1} \bcV_i \ast \cK_i(U),
\end{equation*}
where, for any $i \in \{0,\ldots,2m-1\}$,
\begin{equation}\label{eq: bcVi}
    \bcV_i \bydef j_i \, \cG_0^{-1}\left((\partial_x^{i} \cG_0(\cL^{-1}\bU))^{j_{i}-1}\right) \ast \prod_{\substack{l=0 \\ l\neq i}}^{2m-1} \cG_0^{-1}\left((\partial_x^{l} \cG_0(\cL^{-1}\bU))^{j_{l}} \right),
\end{equation}
and
\begin{equation}
\label{eq:defKi}
    \cK_i \bydef \cG_0^{-1} \partial_x^i \cG_0 \cL^{-1} = \cG_0^{-1} \cG_i \cD_i \cL^{-1} = (\cC_{0,i})^{-1}\cD_i \cL^{-1},
\end{equation}
where, if $i=0$, $\cD_i$ is simply the identity operator.

This section is devoted to the study of the operators $\cK_i$, and more precisely to the derivation of explicit compactness estimates. 
We emphasize that the operators $\cK_i$ incorporate all the possible boundary conditions of the form~\eqref{eq : BC stationary PDE}, via the operator $\cL^{-1}$ (at least, all boundary conditions for which $\tcL$ is invertible, see Section~\ref{ssec : linear problem}). It is because of this high level of generality that Proposition~\ref{prop : Ki is compact} below is somewhat bulky and notation-heavy, and that the given estimates are not necessarily very sharp. However, we reiterate for any given explicit choice of boundary conditions, deriving such estimates is relatively easy, and one can often obtain sharper constants than the general ones given here. We illustrate this on a concrete example in Section~\ref{ssec:KS}. 

\begin{proposition}\label{prop : Ki is compact}
Let $\cB$ be as in~\eqref{def : boundary condition}, and assume the operator $\pi^{\leq 2m-1}\cB\pi^{\leq 2m-1} : \pi^{\leq 2m-1} \ell^1_{\nu} \to \pi^{\leq 2m-1} \ell^1_{\nu}$ is invertible. Then, for all $i \in \{0, \dots, 2m-1\}$,
    $\cK_i : \ell^1_\nu \to \ell^1_\nu$ given by~\eqref{eq:defKi} is well defined and compact for any $\nu\geq 1$, and it admits the following decomposition:
    \begin{align}\label{eq : Ki}
       \cK_i =  \left(\cB_{\cK_i} + (-1)^{m+1} \Sigma^{2m-i}\mathcal{D}_{2m-i}\right)^{-1}\Sigma^{2m-i} \mathcal{C}_{0,2m-i},
    \end{align}
    with $\cB_{\cK_i} : \ell^1_{\nu} \to \ell^1_{\nu}$ a linear operator given by 
\begin{equation}\label{def : boundary condition for f}
    \begin{aligned}
    (\cB_{\cK_i}(U))_n =  \begin{cases}
       \displaystyle \sum_{j=0}^{2m-1} \left(\theta^{(-1)}_{j,n}\, \cB_{-1}(\mathcal{S}^{(i)}U,j) + \theta^{(1)}_{j,n}\, \cB_{1}(\mathcal{S}^{(i)}U,j)\right)  &\text{ if } \quad n = 0, \dots, 2m-i-1,\\
        0 &\text{ if } \quad n \geq 2m-i,
    \end{cases}
\end{aligned}
\end{equation}
 and where 
 \[
    \theta^{(\pm1)} = (\theta^{(\pm1)} _{j,n})_{(j,n) \in \{0, \dots, 2m-1\} \times \{0, \dots, 2m-i-1\}},
\]
are two real-valued sequences. 
In particular, $\cK_i$ can be rewritten as 
\begin{align}\label{eq : Ki2}
       \cK_i &=  (-1)^{m+1}(\cI-\cB_{\cK_i,2m-i}^{\dagger} \, \cB_{\cK_i})\mathcal{D}_{2m-i}^\dagger \Sigma^{2m-i} \mathcal{C}_{0,2m-i} \nonumber\\
       &= (-1)^{m+1}(\cI-\cB_{\cK_i,2m-i}^{\dagger} \, \cB_{\cK_i})\cS^{(2m-i)},
    \end{align}
    where $\cB_{\cK_i,2m-i}^{\dagger} : \pi^{\leq 2m-i-1} \ell^1_\nu \to \pi^{\leq 2m-i-1} \ell^1_\nu$ is the inverse of $\pi^{\leq 2m-i-1}\cB_{\cK_i}\pi^{\leq 2m-i-1}$.
Finally, for all $k > 4m$, we define $\eta^{(\nu)}_{i,k}$ as
\begin{align}\label{eq : eta ki}
    \eta^{(\nu)}_{i,k} \bydef \gamma^{(\nu)}_{2m-i,k} + \nu^{-k} \|\cB_{\cK_i,2m-i}^{\dagger}\|_\nu \,  \sum_{n=0}^{2m-i-1}\sum_{j=0}^{2m-1} \left(\left|\theta^{(-1)}_{j,n}\right| + \left|\theta^{(1)}_{j,n}\right|\right) \gamma_{2m-j,k}^{(1)}  ,
\end{align}
with $\gamma_{i,j}^{(\nu)}$ given as in \eqref{def : gamma i k}. Then, we have
\begin{align*}
    \|\cK_i \Enu\|_\nu \leq  \eta^{(\nu)}_{i,k},  
\end{align*}
where we recall that $\Enu \in \ell^1_\nu$ is the normalized $k$-th element of the canonical Schauder basis of $\ell^1_\nu$ (see~\eqref{def:Ek}).
\end{proposition}

\begin{proof}
Let $U\in\ell^1_\nu$, $V=\cL^{-1}U$ and $W = \cK_i U  $. We first explain how to get~\eqref{eq : Ki}. We could conduct all the analysis directly in sequence space, but some of the arguments are much easier to follow in function space, therefore we also introduce $u=\cG_0(U)$, $v=\cG_0(V)$ and $w = \cG_0(W)$. Noting that $\cK_i= \cG_0^{-1}\partial_x^i(-1)^{m+1} (\partial_x^{2m})^{-1}\cG_0$, where $(-1)^{m+1}(\partial_x^{2m})^{-1}$ denotes the solution operator of~\eqref{eq : linear problem}, we have that $w = \partial_x^i(-1)^{m+1} (\partial_x^{2m})^{-1} u = \partial_x^i v$. Therefore 
\begin{equation}\label{eq: dg=f}
    \partial_x^{2m-i}w = \partial_x^{2m} v= (-1)^{m+1}u,
\end{equation}
which in sequence space gives 
$$(-1)^{m+1}\cD_{2m-i}W = \cC_{0,2m-i}U,$$
and is also equivalent to 
\begin{equation}\label{eq: dg=f sequence}
    (-1)^{m+1}\Sigma^{2m-i}\cD_{2m-i}W = \Sigma^{2m-i}\cC_{0,2m-i}U.
\end{equation}
As in Section~\ref{ssec : linear problem}, the shift $\Sigma^{2m-i}$ has two purposes: it makes $\Sigma^{2m-i}\cD_{2m-i}$ diagonal, and it leaves space on the $2m-i$ first rows in order to incorporate the boundary conditions that are necessary to fully determine $w$ from~\eqref{eq: dg=f}.

In order to obtain these boundary conditions on $w$, or equivalently on $W$, we use the fact that $w = \partial_x^i v$. Therefore, $V = \mathcal{S}^{(i)} W + P$, where $P\in\pi^{\leq i-1}\ell^1_\nu$, i.e., $\cG_0(P)$ is a polynomial of order at most $i-1$. Then, we use the boundary conditions~\eqref{eq : linear problem} on $V$, i.e., the fact that 
\begin{equation}
\label{eq : BV0}
    0 = \cB(V) = \cB(\mathcal{S}^{(i)}W) + \cB(P),
\end{equation} 
where we will eliminate $P$ in order to recover $2m-i$ conditions on $W$.
Recalling that $\cB = \pi^{\leq 2m-1} \cB$, we write
\renewcommand{\arraystretch}{1.5}
\begin{align*}
    \pi^{\leq 2m-1}\cB = \left(\begin{array}{c|c}
         \pi^{\leq i-1} \, \cB \, \pi^{\leq i-1} & \pi^{\leq i-1} \, \cB \, \pi^{>i-1} \\\hline
        \pi^{\leq 2m-1}\pi^{> i-1} \, \cB \, \pi^{\leq i-1} & \pi^{\leq 2m-1}\pi^{> i-1} \, \cB \, \pi^{>i-1}
    \end{array}\right)
            = \left(\begin{array}{c|c}
         \cB_{11} & \cB_{12} \\\hline
        \cB_{21} & \cB_{22}
    \end{array}\right).
\end{align*}
Using that $P = \pi^{\leq i-1} P$ and that $\cS^{(i)} = \pi^{> i-1}\cS^{(i)} $ from \eqref{def : Si}, the ``first part'' $\pi^{\leq i}\cB(V) = 0$ of the boundary conditions yields
\renewcommand{\arraystretch}{1}
\begin{equation*}
    \cB_{11} \pi^{\leq i-1} P +  \cB_{12} 
   \mathcal{S}^{(i)} W = 0.
\end{equation*}
Moreover, we assumed $\pi^{\leq 2m-1}\cB\pi^{\leq 2m-1}$ to be invertible. Using Remark \ref{rem : inverse of B}, we can always rearrange of the $2m$ first rows of $\mathcal{B}$, without changing the expression of $\mathcal{L}^{-1}$, hence of $\mathcal{K}_i$. In particular, we consider, without loss of generality, a re-ordering for which $\cB_{11}$ is also invertible, which means
\begin{equation*}
    \pi^{\leq i-1} P = - \cB_{11}^{-1}\cB_{12} 
    \mathcal{S}^{(i)} W.
\end{equation*}
Plugging this expression in $\pi^{> i-1}\cB(V) = 0$, i.e., in
\begin{equation*}
    \cB_{21} \pi^{\leq i-1} P +  \cB_{22} 
    \mathcal{S}^{(i)} W = 0,
\end{equation*}
yields
\begin{equation*}
     \left(\cB_{22} - \cB_{21}\cB_{11}^{-1}\cB_{12}\right)
     \mathcal{S}^{(i)} W = 0,
\end{equation*}
which gives us $2m-i$ conditions on $\mathcal{S}^{(i)} W$ from~\eqref{eq : BV0}. We now show that these $2m-i$ conditions can be rewritten in the form $\cB_{\cK_i}(W)=0$.
To that end, note that 
\begin{align*}
    \left(\cB_{22} - \cB_{21}\cB_{11}^{-1}\cB_{12}\right)
     \mathcal{S}^{(i)} W  = (\pi^{>i-1} - \cB_{21}\cB_{11}^{-1})(\cB_{12} + \cB_{22})\mathcal{S}^{(i)} W, 
\end{align*}
since $\pi^{>i-1}\cB_{12} = 0$ and $\cB_{21}\cB_{11}^{-1}\cB_{22} = 0$ (because $\pi^{\leq i-1}\pi^{>i-1} = 0$). Moreover, since $\mathcal{S}^{(i)} = \pi^{>i-1} \mathcal{S}^{(i)}$ from \eqref{def : Si}, we obtain that $(\cB_{12} + \cB_{22})\mathcal{S}^{(i)}= \cB \mathcal{S}^{(i)}$. In other terms, our $2m-i$ conditions write
\begin{align*}
    (\pi^{>i-1} - \cB_{21}\cB_{11}^{-1})(\cB\mathcal{S}^{(i)} W) = 0,
\end{align*}
which means they are in fact all given by linear combinations of $(\cB\mathcal{S}^{(i)} W)_n$ for $n\in\{i,\ldots,2m-1\}$.
Therefore, $\cB_{\cK_i} \bydef \left(\cB_{22} - \cB_{21}\cB_{11}^{-1}\cB_{12}\right)
\mathcal{S}^{(i)}$ can indeed be written in the form~\eqref{def : boundary condition for f}, and we obtain that $W = \cK_i U$ is equivalent to having
\begin{equation*}
    (-1)^{m+1}\Sigma^{2m-i}\cD_{2m-i}W = \Sigma^{2m-i}\cC_{0,2m-i}U \quad\text{and}\quad \cB_{\cK_i}W = 0,
\end{equation*}
which can be rewritten as
\begin{equation*}
    \left(\cB_{\cK_i} + (-1)^{m+1}\Sigma^{2m-i}\cD_{2m-i}\right) W = \Sigma^{2m-i}\cC_{0,2m-i} U.
\end{equation*}
Moreover, using once more that $\pi^{\leq 2m-1}\cB\pi^{\leq 2m-1}$ is invertible, we have that the Schur complement
$\pi^{\leq 2m-i-1}\left(\cB_{22} - \cB_{21}\cB_{11}^{-1}\cB_{12}\right)\pi^{\leq 2m-i-1}$ is invertible, and therefore $\pi^{\leq 2m-i-1}\cB_{\cK_i}\pi^{\leq 2m-i-1}$ is  invertible.
The exact same reasoning and algebraic manipulations as those applied to $\tcL$ in Section~\ref{ssec : linear problem} then yield
\begin{align*}
    W &= \left(\cB_{\cK_i} + (-1)^{m+1}\Sigma^{2m-i}\cD_{2m-i}\right)^{-1}\Sigma^{2m-i}\cC_{0,2m-i} U \\
    &= \left(\cB_{\cK_i,2m-i}^{\dagger} - (-1)^{m+1} \cB_{\cK_i,2m-i}^{\dagger} \, \cB_{\cK_i} \, \cD^\dagger_{2m-i} +(-1)^{m+1} \cD^\dagger_{2m-i} \right) \, \Sigma^{2m-i} \, \mathcal{C}_{0,2m-i} U \\
    &= (-1)^{m+1}(\cI-\cB_{\cK_i,2m-i}^{\dagger} \, \cB_{\cK_i})\mathcal{D}_{2m-i}^\dagger \Sigma^{2m-i} \mathcal{C}_{0,2m-i} U,
\end{align*} 
which proves~\eqref{eq : Ki2}.

Now, it remains to estimate $\|\cK_i \Enu\|_\nu$, for $k > 4m$. We first use~\eqref{eq : Ki2} and~\eqref{eq : identity for Si}, and split into two terms:
\begin{align*}
    \|\cK_i \Enu\|_\nu \leq \|\cS^{(2m-i)} \Enu\|_\nu + \|\cB_{\cK_i,2m-i}^{\dagger} \|_\nu\, \|\cB_{\cK_i}\cS^{(2m-i)} \Enu\|_\nu.
\end{align*}
The first term is directly estimated thanks to Lemma~\ref{lem : compactness S}, and we focus on the last term. According to~\eqref{def : boundary condition for f}, we have
\begin{align*}
   \left\vert \left(\cB_{\cK_i} \cS^{(2m-i)}\Enu\right)_n \right\vert \leq   
     \sum_{j=0}^{2m-1} \left( | \theta^{(-1)}_{j,n}| \vert \displaystyle  \cB_{-1}(\mathcal{S}^{(i)}\cS^{(2m-i)} \Enu,j)\vert + | \theta^{(1)}_{j,n}| \vert \cB_{1}(\mathcal{S}^{(i)}\cS^{(2m-i)}\Enu,j)\vert \right),
\end{align*}
for all $n=0,\ldots,2m-i-1$. Since $k\geq 4m$ by assumption,  \eqref{eq : identity of S} yields
\begin{align*}
    \mathcal{S}^{(i)}\mathcal{S}^{(2m-i)} \Enu = \mathcal{S}^{(2m)}\Enu.
\end{align*}
The proof is then a direct consequence of Lemma \ref{lem : estimation BC}.
\end{proof}

We end this section by emphasizing that, according to formula~\eqref{eq : Ki2}, each $\cK_i$ has a structure similar to that of $\cL^{-1}=\cK_0$. That is, each $\cK_i$ is made of a banded operator of bandwidth $2m-i$, namely $(-1)^{m+1}\cS^{(2m-i)}$, complemented with $2m-i-1$ first infinite rows, given by $(-1)^{m}\cB_{\cK_i,2m-i}^{\dagger} \, \cB_{\cK_i}\cS^{(2m-i)}$. In particular, for all $N\in \N_0$,
\begin{equation}
\label{eq:Kifinite}
    \cK_i\pi^{\leq N} = \pi^{\leq N+2m-i}\cK_i\pi^{\leq N},
\end{equation}
and, for all $N\geq 2(2m-i)$,
\begin{equation}\label{eq:Kitail}
\begin{aligned}
    \cK_i\pi^{> N} &= \pi^{\leq 2m-i-1}\cK_i\pi^{> N} + \pi^{> N - (2m-i)}\cK_i\pi^{> N} \\
    & = \pi^{\leq 2m-i-1}\cK_i\pi^{> N} + (-1)^{m+1}\pi^{> N - (2m-i)}\mathcal{S}^{(2m-i)}\pi^{> N}.
\end{aligned}
\end{equation}
The first identity~\eqref{eq:Kifinite} shows that $\cK_i\pi^{\leq N}$ has finite range, therefore we will be able to compute its operator norm, while the second identity~\eqref{eq:Kitail} showcases the splitting between the first infinite rows and the banded part of $\cK_i$.

\section{A Newton-Kantorovich approach for studying steady states}\label{sec:NK}

In this section, we present the main theorem used in the sequel to prove the existence of a solution to~\eqref{eq : F} (or equivalently to~\eqref{eq : stationary PDE}--\eqref{eq : BC stationary PDE}). Before rigorously stating this theorem, let us roughly describe the main ideas of this result, reminiscent of the Newton-Kantorovich theorem.

\subsection{A general framework}
\label{ssec : NK}


In Section~\ref{ssec : zero finding}, we transformed \eqref{eq : stationary PDE}--\eqref{eq : BC stationary PDE} into the zero finding problem $\cF(U)=0$ given in~\eqref{eq : F}, on the Banach space $\ell^1_\nu$. Such a framework is natural when developing computer-assisted methodologies. Indeed, assuming that we have access to an approximate zero $\bar{U} \in \ell^1_\nu$ of $\cF$, we would like to prove that there exists an exact zero $\tilde{U}$ in an explicit neighborhood of $\bar{U}$. For this purpose, a classical idea is to prove that the operator $\cI-D\cF(\bar{U})^{-1}\cF$ from $\ell^1_\nu$ to itself admits a unique fixed-point $\tilde{U}$ in $\B_r(\bar{U})$, the closed ball of center $\bU$ and radius $r$ in $\ell^1_\nu$, for some $r>0$. However, in practice, it can be cumbersome to work with $D\cF(\bar{U})^{-1}$, and one can instead consider a Newton-like operator:
\[
 \cI-\cA \, \cF : \ell^1_\nu \to \ell^1_\nu,
\]
with $\cA$ a suitably chosen approximate inverse of $D\cF(\bar{U)}$. The following result provides us with an efficient way of finding an explicit radius $r$ such that the operator $\cI-\cA \cF$ is a contraction on $\B_r(\bar{U})$.

\begin{theorem}\label{th: radii polynomial}
Let $\nu \geq 1$, let $\cF : \ell^1_\nu \to \ell^1_\nu$ be given as in \eqref{eq : F} and let ${\cA} : \ell^1_\nu \to \ell^1_\nu$ be a bounded linear operator. Moreover, let $\bar{U}$ be a given element of $\ell^1_\nu$, let ${Y}, {Z}_1$ be non-negative constants and let ${Z}_2 : (0, \infty) \to [0,\infty)$ be a non-negative function such that
\begin{equation}\label{eq : bounds general radii polynomial}
  \begin{aligned}
    \|{\cA}{\cF}(\bar{U})\|_\nu & \le {Y}, \\
    \|\cI - {\cA}D{\cF}(\bar{U})\|_\nu &\le {Z}_1,\\
    \|{\cA}\left({D}{\cF}(\bar{U} + h) - D{\cF}(\bar{U})\right)\|_\nu &\le {Z}_2(r)r, ~~ \text{for all } h \in \B_r(0) \text{ and all } r>0.
\end{aligned}  
\end{equation}
If there exists $r>0$ such that
\begin{equation}\label{condition radii polynomial}
    \frac{1}{2}{Z}_2(r)r^2 - (1-{Z}_1)r + {Y} <0 \quad \text{ and } \quad {Z}_1 + {Z}_2(r)r < 1,
 \end{equation}
then, there exists a unique $\tilde{U} \in \B_r(\bar{U}) \subset \ell^1_\nu$ such that ${\cF}(\tilde{U})=0$. 
\end{theorem}

This theorem is easy to prove. Indeed, the conditions~\eqref{condition radii polynomial} ensure that $I-\cA \cF$ maps $\B_r(\bar{U})$ into itself and is contracting on that ball. Moreover, condition~\eqref{condition radii polynomial} also implies that $Z_1<1$, and therefore that $\cA$ is injective, hence fixed-points of $I-\cA \cF$ are in one-to-one correspondence with zeros of $\cF$. For variants of Theorem~\ref{th: radii polynomial} and detailed proofs, we refer to~\cite{ArKoTe05,DaLeMi07,NakPluWat19,Ort68,Yam98}.

Before deriving bounds satisfying~\eqref{eq : bounds general radii polynomial}, it remains to define the approximate solution $\bar{U}$ to the zero finding problem $\cF=0$ and the approximate inverse $\cA$ of $D\cF(\bar{U})$.

In the rest of the paper, we fix $N \in \mathbb{N}$ to be the truncation size of our numerical approximations. Then, we define 
    \[
    \cF^{\leq N} = \pi^{\leq N} \cF \,\pi^{\leq N}.
    \]
    Now, $\bar{U}$ is simply obtained as an approximate solution to the ``finite dimensional'' problem $\cF^{\leq N}=0$, obtained numerically. In particular, $\bar{U}$ satisfies $\bar{U} = \pi^{\leq N} \bar{U}$, meaning that $\bar{U}$ only has a finite number of non-zero coefficients ($\bar{U}$ can be seen as a vector) which are stored on the computer.

Let us now define the linear operator $\cA : \ell^1_\nu \to \ell^1_\nu$ which is an approximate inverse of $D\cF(\bar{U})$. We recall that there exist $\bar{\mathcal{V}}_0, \dots, \bar{\mathcal{V}}_{2m-1} \in \ell^1_\nu$, defined in~\eqref{eq: bcVi}, such that
\begin{align}\label{eq : DF compact perturbation identity}
    D\cF(\bar{U}) = \cI + \sum_{i=0}^{2m-1} \bar{\mathcal{V}}_i \cK_i,
\end{align}
where we identified the element $\bar{\mathcal{V}}_{i} \in \ell^1_\nu$ and the multiplication operator $\bar{\mathcal{V}}_{i} :\ell^1_\nu\to\ell^1_\nu$ defined by
$\bar{\mathcal{V}}_{i}U = \bar{\mathcal{V}}_{i}\ast U$. Since multiplication operators are bounded ($\ell^1_\nu$ is a Banach algebra), and each $\cK_i$ is compact, each operator $\bar{\mathcal{V}}_i \cK_i$ is compact on $\ell^1_\nu$. This implies that $D\cF(\bar{U})$ is a compact perturbation of the identity. Consequently, its inverse must also be a compact perturbation of the identity. Having such a property in mind, we take $\cA$ of the form
\begin{equation}\label{eq :decomposition of A}
    \cA = \cA_0 + \pi^{>N},
\end{equation}
where $\cA_0 = \pi^{\leq  N} \cA_0\pi^{\leq  N}$ is a finite dimensional operator, that can be represented by a $( N+1)\times ( N+ 1)$ matrix. In practice $\cA_0$ is constructed numerically to approximate $\pi^{\leq  N} D\cF(\bar{U})^{-1} \pi^{\leq  N}$, and we simply take $\cA_0 \approx \left(D\cF^{\leq N}(\bar{U})\right)^{-1}$. Such an $\cA$ is indeed a compact perturbation of the identity, which we expect to approximate $D\cF(\bar{U})^{-1}$ well if $N$ is taken large enough. We make this statement quantitative in the next section, where we derive computable bound $Y$, $Z_1$ and $Z_2$ satisfying~\eqref{eq : bounds general radii polynomial}.
\begin{remark}
    In practice, it could be computationally more efficient to use separate truncation levels $N$ for the approximate solution $\bU$ and for the finite part $\cA_0$ of $\cA$, but, for the sake of presentation, we use the same $N$ for both.
\end{remark}

\subsection{Computable estimates}\label{sec : constructive proofs}

In this entire section, we assume that $\cF$ is given by~\eqref{eq : F}, that $\cA$ is of the form specified in~\eqref{eq :decomposition of A}, and derive estimates $Y$, $Z_1$ and $Z_2$ satisfying the assumption~\eqref{eq : bounds general radii polynomial}.

\begin{lemma}\label{lem : Y bound general}
    Let $Y$ be a constant satisfying
    \begin{align*}
        \|\cA_0 \pi^{\leq N}\cF(\bar{U})\|_\nu + \|\pi^{>N}\cK(\bar{U}) - \pi^{>N}\Psi\|_\nu   \leq Y.
    \end{align*}
    Then, $\|\cA\cF(\bar{U})\|_\nu \leq Y$.
\end{lemma}

\begin{proof}
    By definition of $\cA$ in~\eqref{eq :decomposition of A}, we have 
    \begin{align*}
        \|\cA \cF(\bar{U})\|_\nu &= \|\cA_0 \pi^{\leq N}\cF(\bar{U})\|_\nu + \|\pi^{>N}\cF(\bar{U})\|_\nu.
    \end{align*}
    We conclude the proof using the definition of $\cF$ in~\eqref{eq : F}. 
\end{proof}
Note that, since $\bU$ belongs to $\pi^{\leq N}\ell^1_\nu$ (i.e., $\bu = \cG_0^{-1}(\bU)$ is a trigonometric polynomial of order at most $N$), $\cL^{-1}\bU$ belongs to $\pi^{\leq N+2m}\ell^1_\nu$ (i.e., $(\partial_x^{2m})^{-1}\bu$ is also a trigonometric polynomial, of order at most $N+2m$), and can be computed exactly. Here and in the sequel, when we say that something ``can be computed exactly'', we actually mean that it can be computed with finitely many basic arithmetic operations, and therefore rigorously enclosed by a computer using interval arithmetic. Since $\cQ$ only involves multiplication and derivatives, $\cK(\bU)$ itself can be computed exactly, as well as $\cF(\bU)$, hence we can find a computable bound $Y$ as in Lemma~\ref{lem : Y bound general}.
\begin{lemma}\label{lem : Z1 bound general}
    Define the following constants:
\begin{align*}
    Z_{1,0} &=  \|\pi^{\leq N} - \cA_0D\cF(\bar{U})\pi^{\leq N}\|_\nu, \quad
    Z_{1,1} = \left\|\sum_{i=0}^{2m-1}\pi^{>N}\bar{\mathcal{V}}_i \cK_i\pi^{\leq N}\right\|_\nu\\
    Z_{1,2} &= 2 \sum_{i=0}^{2m-1}\left(\eta^{(\nu)}_{i,N+1}\sum_{n > N-2m+i} |(\bar{\mathcal{V}}_i)_n| \nu^{n}   + \|\bar{\mathcal{V}}_i\|_\nu \gamma_{2m-i,N+1}^{(\nu)}\right), \quad
    Z_{1,3} = \sum_{i=0}^{2m-1}  \|\cA_0\bar{\mathcal{V}}_i\|_\nu \,\eta^{(\nu)}_{i,N+1},
\end{align*}
where we recall definition~\eqref{eq : eta ki} of the constants $\eta^{(\nu)}_{i,N+1}$. Then, the constant $Z_1$ such that
    \begin{align*}
     \max\{Z_{1,0} +Z_{1,1}, \, Z_{1,2} + Z_{1,3}\}   \leq Z_1,
    \end{align*}
    satisfies the estimate $\|\cI - \cA D\cF(\bar{U})\|_\nu \leq Z_1$.
\end{lemma}

\begin{proof}
 Using the fact that $\|\cdot\|_\nu$ is a weighted $\ell^1_1$ norm, we have 
   \begin{align*}
       \|\cI - \cA D\cF(\bar{U})\|_\nu 
       = \max\left\{\|(\cI - \cA D\cF(\bar{U}))\pi^{\leq N}\|_\nu, \, \|(\cI - \cA D\cF(\bar{U}))\pi^{>N}\|_\nu\right\}.
   \end{align*}
   Now, using the definition~\eqref{eq :decomposition of A} of $\cA$, we obtain
   \begin{align*}
        \|(\cI - \cA D\cF(\bar{U}))\pi^{\leq N}\|_\nu &\leq \|\pi^{\leq N} - \pi^{\leq N}\cA D \cF(\bar{U})\pi^{\leq N}\|_\nu + \| \pi^{>N} \cA D\cF(\bar{U})\pi^{\leq N}\|_\nu\\
        &= \|\pi^{\leq N} - \cA_0 D\cF(\bar{U})\pi^{\leq N}\|_\nu + \|\pi^{>N}D \cF(\bar{U})\pi^{\leq N}\|_\nu\\
        & =Z_{1,0} + \|\pi^{>N}D \cF(\bar{U})\pi^{\leq N}\|_\nu .
   \end{align*}

    Then, since $D\cF(\bar{U}) = \cI + D\cK(\bar{U})$ and $\pi^{>N}\pi^{\leq N} = 0$, we use~\eqref{eq : DF compact perturbation identity} to get 
    \begin{align*}
       \|\pi^{>N}D\cF(\bar{U})\pi^{\leq N}\|_\nu &= \|\pi^{>N}D\cK(\bar{U})\pi^{\leq N}\|_\nu = \left\|\sum_{i=0}^{2m-1}\pi^{>N}\bar{\mathcal{V}}_i \cK_i\pi^{\leq N}\right\|_\nu = Z_{1,1}. 
   \end{align*}
   Moreover, as $\cA = \cA_0 + \pi^{>N}$, we get
   \begin{align*}
       \|\pi^{>N} - \cA D\cF(\bar{U}) \pi^{>N}\|_\nu &\leq \|\pi^{>N} - \pi^{>N} \cA D\cF(\bar{U}) \pi^{>N}\|_\nu + \| \pi^{\leq N}\cA D\cF(\bar{U}) \pi^{>N}\|_\nu\\
       &= \|\pi^{>N} - \pi^{>N}\left(\cI - D\cK(\bar{U})\right)\pi^{>N}\|_\nu + \| \cA_0\pi^{\leq N}\left(\cI + D\cK(\bar{U})\right) \pi^{>N}\|_\nu \\
       &= \|\pi^{>N}D\cK(\bar{U})\pi^{>N}\|_\nu + \| \cA_0 D\cK(\bar{U}) \pi^{>N}\|_\nu.
   \end{align*}
   Using once more identity~\eqref{eq : DF compact perturbation identity}, together with Proposition~\ref{prop : Ki is compact} and the fact that $\eta^{(\nu)}_{i,n}$ decreases with $n$, we deduce that
   \begin{align*}
       \| \cA_0 D\cK(\bar{U}) \pi^{>N}\|_\nu \leq \left\|\sum_{i=0}^{2m-1} \cA_0  \bar{\mathcal{V}}_i\cK_i\pi^{>N}\right\|_\nu \leq \sum_{i=0}^{2m-1}  \|\cA_0\bar{\mathcal{V}}_i\|_\nu \,\eta^{(\nu)}_{i,N+1} = Z_{1,3}.
   \end{align*}
   Now, using~\eqref{eq:Kitail} together with the assumption $N\geq 4m$, we get, for each $i\in\{0,\ldots,2m-1\}$,
\begin{align*}
    \pi^{>N} \bar{\mathcal{V}}_i \mathcal{K}_i  \pi^{>N} = \pi^{>N} \bar{\mathcal{V}}_i \pi^{\leq 2m - i} \mathcal{K}_i  \pi^{>N} + \pi^{>N} \bar{\mathcal{V}}_i \pi^{> N - 2m + i} \mathcal{K}_i  \pi^{>N}.
\end{align*}
For all $U \in \ell^1_\nu$, we have 
\begin{align*}
    \|\pi^{>N} \bar{\mathcal{V}}_i \pi^{\leq 2m - i} U\|_{\nu} = 2\sum_{n > N} \nu^n \left|\sum_{k=-2m+i}^{2m-i} (\bar{\cV}_i)_{|n-k|}U_{|k|} \right| &\leq 2\sum_{k=-2m+i}^{2m-i} |U_{|k|}| \nu^{|k|}   \sum_{n > N} \nu^{n-|k|} |(\bar{\cV}_i)_{|n-k|}| \\
    &\leq  2 \|U\|_\nu   \sum_{n > N - 2m + i} \nu^{n} |(\bar{\cV}_i)_{n}|,
\end{align*}
and thus, according to Proposition~\ref{prop : Ki is compact}
\begin{align*}
  \| \pi^{>N} \bar{\mathcal{V}}_i \pi^{\leq 2m - i} \mathcal{K}_i  \pi^{>N}\|_\nu \leq  2\eta^{(\nu)}_{i,N+1}\sum_{n > N-2m+i} |(\bar{\mathcal{V}}_i)_n| \nu^{n}.
\end{align*}
Finally, combining \eqref{eq:Kitail} and  Lemma~\ref{lem : compactness S} to estimate the remaining term, we obtain 
\begin{align}\label{eq : estimate Z12}
    \| \pi^{>N} \bar{\mathcal{V}}_i \mathcal{K}_i  \pi^{>N}\|_\nu \leq 2\eta^{(\nu)}_{i,N+1}\sum_{n > N-2m+i} |(\bar{\mathcal{V}}_i)_n| \nu^{n}   + \|\bar{\mathcal{V}}_i\|_\nu \gamma_{2m-i,N+1}^{(\nu)} = Z_{1,2}.
\end{align}
   This concludes the proof of Lemma~\ref{lem : Z1 bound general}.
\end{proof}


We emphasize that each of the constants involved in Lemma~\ref{lem : Z1 bound general} are computable, as they involve only finite quantities. Indeed, recalling~\eqref{eq:Kifinite}, $\cK_i\pi^{\leq N}$ has a finite dimensional range. Furthermore, since $\bU = \pi^{\leq N} \bU$, we have that $\bar{\cV}_i$ has a finite number of nonzero coefficients (cf. \eqref{eq: bcVi}), which implies that $\bcV_i\cK_i \pi^{\leq N}$ is also a finite dimensional operator.

\begin{lemma}\label{lem : Z2 bound general}
Let $r>0$ and let $Z_2(r)$ be satisfying 
\begin{align*}
   Z_2(r)r \geq  \|\cA\|_\nu \sup_{h \in B_r(0)} \|D\cK(\bar{U}+h)-D\cK(\bar{U})\|_\nu.
\end{align*}
 Then, it holds 
\begin{align*}
\|{\cA}\left({D}{\cF}(\bar{U}+ h) - D{\cF}(\bar{U})\right)\|_\nu \le {Z}_2(r)r, \qquad \forall h \in \overline{\B_r(0)}.
\end{align*}
\end{lemma}

\begin{proof}
 The proof is a direct consequence of the definition of $\cF$ in~\eqref{eq : F}.
\end{proof}
Computable $Z_2$ estimates will be derived on a case-by-case basis in Section~\ref{sec : applications} for two problems. The obtained bounds could be adapted in a straightforward manner to any other example with a nonlinear term $\cQ$ of the form~\eqref{eq : nonlinear part}.

In Section~\ref{sec : applications}, we give two concrete instances of~\eqref{eq : stationary PDE}--\eqref{eq : BC stationary PDE} for which we use the bounds obtained here, together with Theorem~\ref{th: radii polynomial}, in order to obtain quantitative existence results.
\section{Stability}
\label{sec : stability}
In this section, we investigate the spectral stability of stationary solutions to \eqref{eq : parabolic PDE original}. Suppose we were able to prove the existence of a zero $\tilde{U}$ of $\cF$ in a vicinity of an approximate solution $\bar{U}$ thanks to Theorem \ref{th: radii polynomial}. As in the previous section, we always consider here the $\cF$ given in~\eqref{eq : F}. Denoting 
\begin{equation}\label{def : tilde v}
    \tilde{v} \bydef \mathcal{G}_0\left(\mathcal{L}^{-1}\tilde{U}\right),
\end{equation}
the corresponding stationary solution to~\eqref{eq : parabolic PDE original} in function space, our main objective is to rigorously control the spectrum of the linearized operator at $\tilde{v}$, and more precisely, to be able to count the number of unstable eigenvalues in~\eqref{eq : evp functions}.

As for the steady state problem, we first make the change of unknown $v = (-1)^{m+1}(\partial_x^{2m})^{-1}u $, which transforms~\eqref{eq : evp functions} into the generalized eigenvalue problem
\begin{equation*}
    u + D\cQ(\tilde{v}) (-1)^{m+1}(\partial_x^{2m})^{-1}u  = \lambda (-1)^{m+1}(\partial_x^{2m})^{-1}u,\qquad u \not\equiv 0.
\end{equation*}
In sequence space, this generalized eigenvalue problem writes
\begin{equation*}
    U + D\cK(\tilde{U})U = \lambda \cL^{-1}U \qquad U\neq 0,
\end{equation*}
i.e.,
\begin{equation}\label{eq : generalized evp sequences}
    D\cF(\tilde{U})U = \lambda \cL^{-1}U\qquad U\neq 0.
\end{equation}
Since the steady state $\tilde{v}$ is automatically smooth by elliptic regularity, any eigenvector $v$ of~\eqref{eq : evp functions} will also be smooth, and hence the corresponding eigenvectors $U = \cG_0((-1)^{m+1}\partial_x^{2m} v)$ will all at least belong to $\ell^1_1$. We therefore study the generalized eigenvalue problem~\eqref{eq : generalized evp sequences} for eigenvectors $U$ in $\ell^1_1$.

Even though we will come back to the generalized eigenproblem formulation~\eqref{eq : generalized evp sequences} later on, it also useful to recast it as a standard eigenvalue problem, which is made precise in the following statement.
\begin{lemma}\label{lem : equivalent eig problem}
    Repeat the assumptions of Theorem~\ref{th: radii polynomial}, assume that~\eqref{eq : bounds general radii polynomial}--\eqref{condition radii polynomial} hold for $\nu=1$, and let $\tilde{U}$ be the unique zero of $\cF$ in $\B_r(\bU)$. Then $D\cF(\tilde{U}):\ell^1_1 \to \ell^1_1$ is boundedly invertible. Moreover, we have that $\lambda \in \mathbb{C}$ is an eigenvalue of the generalized eigenvalue problem~\eqref{eq : generalized evp sequences} (posed on $\ell^1_1$), if and only if $\lambda \neq 0$ and $\frac{1}{\lambda}$ is an eigenvalue of $D\cF(\tilde{U})^{-1}\mathcal{L}^{-1}$.
\end{lemma}

\begin{proof}
Provided $D\cF(\tilde{U}):\ell^1_1 \to \ell^1_1$ is boundedly invertible, $0$ cannot be an eigenvalue of the generalized eigenvalue problem~\eqref{eq : generalized evp sequences}, which can then be rewritten as
\begin{equation*}
    \frac{1}{\lambda}U = D\cF(\tilde{U})^{-1}\mathcal{L}^{-1}U.
\end{equation*}
However, the invertibility of $D\cF(\tilde{U})$ is in fact automatic for any zero $\tilde{U}$ obtained via Theorem~\ref{th: radii polynomial}. Indeed, taking $r>0$ satisfying~\eqref{condition radii polynomial}, we have that $\Vert \tilde{U} - \bU\Vert \leq r$, and that 
\begin{align}\label{eq : Z1pZ2}
    \|\cI - \cA D\cF(\tilde{U})\|_1 \leq \|\cI - \cA D\cF(\bar{U})\|_1 + \|\cA (D\cF(\bar{U})-D\cF(\tilde{U}))\|_1 \leq Z_1 + Z_2(r) r < 1.
\end{align}
Therefore, $\cA D\cF(\tilde{U})$ is invertible, and $D\cF(\tilde{U})$ must be injective. Moreover, since $D\cF(\tilde{U}) = \cI + D\cK(\tilde{U})$, with $D\cK(\tilde{U})$ compact thanks to Proposition~\ref{prop : Ki is compact}, $D\cF(\tilde{U})$ is a Fredholm operator of index zero and is therefore invertible. 
\end{proof}

We present below two separate approaches to rigorously count the number of unstable eigenvalues of~\eqref{eq : generalized evp sequences}. The first one is described in Section~\ref{ssec: Gersh}, and relies on a Gershgorin disks argument applied to $D\cF(\tilde{U})^{-1}\mathcal{L}^{-1}$, combined with an a priori bound on unstable eigenvalues. 
The second approach is given in Section~\ref{ssec : second enclosure spectrum}. It is based on the idea that, if the eigenvalues under consideration are not too large, their corresponding eigenvectors must be mostly concentrated on low modes, hence we can compare such eigenvalues to those of a finite dimensional problem. Because of this initial assumption that the eigenvalues are not too large, we also requires an a priori bound on unstable eigenvalues in order to be able to count the total number of unstable eigenvalues. In both cases, it should be noted that the quantitative bounds needed are based on estimates that were already used for obtaining the steady state.

\subsection{Enclosure of the spectrum using Gershgorin disks}
\label{ssec: Gersh}

In this section, we study the standard eigenvalue problem reformulation of~\eqref{eq : generalized evp sequences} provided by Lemma~\ref{lem : equivalent eig problem}, i.e., we try to obtain quantitative information on the spectrum of $D\cF(\tilde{U})^{-1}\mathcal{L}^{-1}$. 

If we were dealing with a problem that can be studied with Fourier series, we could simply apply a Gershgorin argument to the Fourier series representation of $(-1)^{m+1} \partial_x^{2m} + D\cQ(\tilde{v})$ and be done. Here, we face the additional difficulty that we first needed to turn to a generalized eigenvalue problem to deal with the Gegenbauer setup and the boundary conditions. Even though we recovered a standard eigenvalue problem with $D\cF(\tilde{U})^{-1}\mathcal{L}^{-1}$, this operator is compact, and therefore its eigenvalues accumulate at zero. This means that a Gershgorin argument can only give us the sign (of the real part of) finitely many eigenvalues, unless we have a very fine control on the radius of the asymptotic Gershgorin disks, which we are not able obtain here. In order to overcome this obstacle, we combine a Gershgorin argument with rougher a priori estimates on the eigenvalues of~\eqref{eq : generalized evp sequences}, showing that if $\lambda$ has a positive real part, then it must remain bounded (in modulus), with an explicit bound. As a consequence, if an eigenvalue $\frac{1}{\lambda}$ of $D\cF(\tilde{U})^{-1}\mathcal{L}^{-1}$ is close enough to zero, it must have negative real part.

If we were to directly compute the Gershgorin disks of $D\cF(\tilde{U})^{-1}\mathcal{L}^{-1}$, there is no guarantee that we would get any useful information about its spectrum, as the disks would most likely overlap each others and intersect both halves of the complex plane. Therefore, we first introduce a pseudo-diagonalization of the operator $D\cF(\tilde{U})^{-1}\mathcal{L}^{-1}$, i.e., we want to build an operator $\cP$ such that $\cP^{-1}D\cF(\tilde{U})^{-1}\mathcal{L}^{-1} \cP$ is ``close'' to a diagonal (infinite) matrix.

Because $D\cF(\tilde{U})^{-1}\mathcal{L}^{-1}$ is compact, we consider $\cP$ of the form
\begin{equation*}
   \cP = \cP_0 + \pi^{>N},
\end{equation*}
where $\cP_0 = \pi^{\leq N}\cP_0 \pi^{\leq N}$, i.e., $\cP_0$ admits a square matrix representation of size $N+1$. In practice, the columns of $\cP_0$ are obtained numerically as  approximations of the $N+1$ first eigenvectors of $\pi^{\leq N}D\cF(\tilde{U})^{-1}\mathcal{L}^{-1}\pi^{\leq N}$. Moreover, we suppose that $\cP_0 : \pi^{\leq N}\ell^1_{\nu} \to \pi^{\leq N}\ell^1_{\nu}$ is invertible (which is a reasonable assumption since $D\cF(\tilde{U})^{-1}$ is invertible), and we denote, with a slight abuse of notation, $\cP_0^{-1} : \pi^{\leq N}\ell^1_{\nu} \to \pi^{\leq N}\ell^1_{\nu}$ its inverse. In particular, $\cP$ is invertible with
\begin{equation*}
     \cP^{-1} = \cP_0^{-1} + \pi^{>N}.
\end{equation*} 
Finally, we define $\mathcal{M}:\ell^1_1\to\ell^1_1$ as 
\begin{equation}\label{def : matrix diag M}
    \cM \bydef \cP^{-1}D\cF(\tilde{U})^{-1}\mathcal{L}^{-1} \cP.
\end{equation}
By construction, the operators $\cM$ and $D\cF(\tilde{U})^{-1}\mathcal{L}^{-1}$ have the same spectrum, and thanks to our choice of $\cP_0$, we expect $\cM$, or at least a finite truncation of $\cM$, to be diagonally dominant. In particular, the first Gershgorin disks of $\cM$ should be narrow enough to allow us to determine the sign of the corresponding eigenvalues by using Lemma~\ref{lem : gershgorin matrix} below, which is nothing but a generalization of the well-known Gershgorin theorem to infinite dimensional compact operators. In the sequel we identify $\cM$ with its infinite matrix $\left(\cM_{k,n}\right)_{k,n\in\N_0}$ in the Schauder basis $(E_n^{(1)})_{n\in\N_0}$ of $\ell^1_1$, and also consider the weights $(\xi_n)_{n\in\N_0}$ given by
\begin{equation}
    \label{eq : xi}
    \xi_n \bydef
    \begin{cases}
        1 \quad n=0,\\
        2 \quad n\in\N.
    \end{cases}
\end{equation}
Note that these weights are exactly those appearing in the $\ell^1_1$ norm~\eqref{eq : ell1norm}:
\begin{equation*}
    \Vert U \Vert_1 = \sum_{n\in\N_0} \vert U_n\vert \xi_n,\qquad \forall~U\in\ell^1_1,
\end{equation*}
and that the associated operator norm can then be expressed as follows
\begin{equation}
    \label{eq : ell1opnorm}
    \Vert \cM\Vert_1 = \sup_{n\in\N_0} \frac{1}{\xi_n}\sum_{k \in \mathbb{N}_0} |\mathcal{M}_{k,n}| \xi_k = \sup_{n\in\N_0} \frac{1}{\xi_n} \|\cM_{\cdot,n}\|_1,
\end{equation}
where $\cM_{\cdot,n}$ denotes the $n$-th column of $\cM$. 

\begin{lemma}\label{lem : gershgorin matrix}
For all $n \in \mathbb{N}_0$, we define
\begin{align*}
    r_n \bydef   \frac{1}{\xi_n}\sum_{k\in \mathbb{N}_0, ~k\neq n} 
 \left|\cM_{k,n}\right| \xi_k,  
\end{align*}
where we recall definition~\eqref{def : matrix diag M} of the operator $\cM$, and we denote $\B_{r_n}(\cM_{n,n}) \subset \mathbb{C}$ the closed disk of radius $r_n$ centered at $\cM_{n,n}$, which we call the $n$-th Gershgorin disk of $\cM$.
 Then, the spectrum of $D\cF(\tilde{U})^{-1}\mathcal{L}^{-1}$ consists of eigenvalues belonging to the set $\bigcup_{n \in \mathbb{N}_0} \B_{r_n}(\cM_{n,n}) \subset \C$. Moreover, any set of $k$ Gershgorin disks whose union is disjoint from all other Gershgorin disks contains exactly $k$ eigenvalues of $D\cF(\tilde{U})^{-1}\mathcal{L}^{-1}$, counted with algebraic multiplicity.
\end{lemma}

%
We note that the Gershgorin disks introduced in Lemma~\ref{lem : gershgorin matrix} are defined column-wise. As is the case for finite matrices, we could also have considered the Gershgorin disks associated to the rows of $\cM$, but we choose to work with columns as they interact more naturally with the $\ell^1_1$ operator norm~\eqref{eq : ell1opnorm}.
 This makes natural the above definition of the radii $r_n$ (and whose role will become apparent in Lemma~\ref{lem : Gersh M and Mbar}): essentially measuring the radii of the disks in our $\ell^1_1$ norm. 



%
\begin{proof}
First, we know from Lemma~\ref{lem : equivalent eig problem}, that $D\cF(\tilde{U})^{-1}: \ell^1_1 \to \ell^1_1$ is bounded, hence $D\cF(\tilde{U})^{-1}\mathcal{L}^{-1}$ is compact. Since $\cP:\ell^1_1\to\ell^1_1$ is bounded and has bounded inverse $\cM$ is also compact and has the same spectrum as $D\cF(\tilde{U})^{-1}\mathcal{L}^{-1}$, which is indeed made of eigenvalues only.

Let $\mu\in\C\setminus\{0\}$ an eigenvalue of $\cM$, and $U\in\ell^\infty = (\ell^1_1)^*$ a corresponding left-eigenvector (which is guaranteed to exist because $\cM$ is compact). We first show that the eigenvector $U$ must have additional decay, which then allows us to mimic the proof of the Gershgorin Theorem in the finite dimensional case. By construction, we have that \begin{equation*}
 \mathcal{M}^* U =   \left(\mathcal{L}^{-1}\right)^* \left(D\cF(\tilde{U})^{-1}\right)^* U = \mu U,
\end{equation*}
where the notation $*$ stands for the  the adjoint operator. Since $D\cF(\tilde{U})^{-1} : \ell^1_1 \to \ell^1_1$ is bounded, we have that $\left(D\cF(\tilde{U})^{-1}\right)^* : \ell^\infty \to \ell^\infty$ is bounded. Let $V = \left(D\cF(\tilde{U})^{-1}\right)^* U$, then $V \in \ell^\infty$ and 
\begin{equation*}
   \left(\mathcal{L}^{-1}\right)^* V = \mu U.
\end{equation*}
Let $\ell^\infty_1 \bydef \{U \in \ell^\infty, ~ \|U\|_{\ell^\infty_1} < \infty\}$, where 
$$\|U\|_{\ell^\infty_1} \bydef \sup_{n \in \mathbb{N}_0} |U_n|(1+n).$$ We want to prove that $U \in \ell^\infty_1$. In fact, this is a direct consequence of Proposition \ref{prop : Ki is compact}, since we have that $k \eta^{(1)}_{0,k} \to 0$ as $k \to \infty.$ Using that $U \in \ell^\infty_1$, we have that $|U_k| \to 0$ as $k \to \infty$, which implies that there exists $n \in \mathbb{N}_0$ such that $\vert U_{n}\vert = \max_{k\in\N_0} \vert U_k\vert$.

Now, mimicking the proof in the finite-dimensional case, we use that $(\cM^* U)_{n} = \mu U_n$, to get
\begin{equation*}
    \left( \cM_{n,n} - \mu \right) U_{n} = \sum_{k\neq n} \cM_{k,n}U_k, 
\end{equation*}
hence 
\begin{equation*}
    \vert \cM_{n,n} - \mu\vert \leq \sum_{k\neq n} \vert \cM_{k,n} \vert.
\end{equation*}
Therefore $\mu$ belongs to the closed disk of center $\cM_{n,n}$ and radius $\tilde{r}_n \bydef \sum_{k\neq n} \vert \cM_{k,n} \vert$. We explain at the end of the proof why we end up having slightly different radii $r_n$.

For the moment, we have already proven that the spectrum of $\cM$ is included in $\bigcup_{n \in \mathbb{N}_0} \B_{\tilde{r}_n}(\cM_{n,n})$. Let us then assume that there exists a subset $I_k\subset \N_0$ of cardinal $k$ such that 
\begin{equation*} 
    \left(\bigcup_{n \in I_k} \B_{\tilde{r}_n}(\cM_{n,n})\right) \bigcap \left(\bigcup_{n \in \mathbb{N}_0 \setminus I_k} \B_{\tilde{r}_n}(\cM_{n,n})\right) = \emptyset,
\end{equation*}
and consider a Jordan curve $\Gamma$ separating $\bigcup_{n \in I_k} \B_{\tilde{r}_n}(\cM_{n,n})$ from $\bigcup_{n \in \mathbb{N}_0 \setminus I_k} \B_{\tilde{r}_n}(\cM_{n,n})$.
We now prove that $\bigcup_{n \in I_k} \B_{\tilde{r}_n}(\cM_{n,n})$ contains exactly $k$ eigenvalues of $\cM$.
Denoting by $\diag(\cM)$ the diagonal operator made of the diagonal elements $\cM_{n,n}$, we consider, for all $s\in[0,1]$, 
\begin{equation*}
    \cM(s)\bydef \diag(\cM) + s (\cM - \diag(\cM)).
\end{equation*}
The map $s\mapsto \cM(s)$ is continuous from $[0,1]$ to $B(\ell^1_1)$ (the space of bounded linear operators on $\ell^1_1$), and we have that $\bigcup_{n \in I_k} \B_{\tilde{r}_n}(\cM_{n,n})$ contains exactly $k$ eigenvalues of $\cM(0)=\diag(\cM)$ (namely $\cM_{n,n}$ for $n\in I_k$). Moreover, for all $s\in[0,1]$, $\cM(s)$ is compact and for all $n\in\N_0$
\begin{align*}
    \sum_{k\neq n} \vert \cM_{k,n}(s) \vert = s\sum_{k\neq n} \vert \cM_{k,n} \vert \leq \tilde{r}_n.
\end{align*}
Therefore, the spectrum of $\cM(s)$ is included in $\bigcup_{n \in \mathbb{N}_0} \B_{\tilde{r}_n}(\cM_{n,n})$, and $\Gamma$ separates the spectrum of $\cM(s)$, for all $s\in[0,1]$. By upper continuity of separated parts of the spectrum~\cite[Chapter IV, Theorem 3.16]{kato2013perturbation} $\cM(s)$ has exactly $k$ eigenvalues in $\bigcup_{n \in I_k} \B_{\tilde{r}_n}(\cM_{n,n})$ for all $s\in[0,1]$, and in particular this holds for $\cM(1) = \cM$.

Finally, let us consider the diagonal operator $\xi:\ell^1_1\to\ell^1_1$ given by $\xi_{n,n,} = \xi_{n}$ for all $n\in\N_0$, and let $\tilde{\cM} = \xi \cM \xi^{-1}$. Obviously $\tilde{\cM}$ and $\cM$ have the same spectrum, and we obtain Lemma~\ref{lem : gershgorin matrix} by applying the above proof to $\tilde{\cM}$, for which we have
\begin{equation*}
    \sum_{k\neq n} \vert \tilde\cM_{k,n} \vert = \sum_{k\neq n} \xi_{k}\vert \cM_{k,n} \vert \xi_{n}^{-1} = r_n. \qedhere
\end{equation*}
\end{proof}
This proof of Lemma~\ref{lem : gershgorin matrix} is strongly inspired from~\cite{BrePayReiTan25}, which deals instead with operators having compact resolvent, and from the references therein. A similar Gershgorin argument for infinite compact operators was also proposed recently in~\cite{Cad25}.

For the remainder of this section, our goal is to obtain quantitative and sharp information on the Gershgorin disks of $\cM$, in order to make good use of Lemma~\ref{lem : gershgorin matrix}. The main difficulty lies in the fact that $D\cF(\tilde{U})^{-1}$, which appears in $\cM$, is not known exactly, as even $\tilde{U}$ itself is not known exactly. However, we have a reasonable proxy for $D\cF(\tilde{U})^{-1}$, namely $\cA$.  Therefore, we consider $\bar{\cM}$, an approximation of $\cM$, for which we can compute exactly an arbitrary number of Gershgorin disks using interval arithmetic, and then explicitly estimate the errors between the disks of $\bcM$ and those of $\cM$. In practice, we chose $\bar{\cM}$  as follows
\begin{align}\label{def : choice of bar M}
    (\bar{\cM})_{k,n} \bydef \begin{cases}
        \left(\pi^{\leq N}\cP^{-1}\cA \cL^{-1} \cP\right)_{k,n} &\text{ if } n \leq N\\
        \left(\cP^{-1}\cL^{-1} \cP\right)_{k,n} &\text{ if } n > N.
    \end{cases}
\end{align}
This specific choice is merely convenient but not mandatory, and one could potentially improve $\bar{\cM}$ by accurately computing the eigenvectors of $\cM.$ In what follows, we derive several estimates that are valid for a general bounded linear operator $\bar{\cM} : \ell^1_1 \to \ell^1_1$, but keep in mind the above for practical computations.
For any given choice of $\bar{\cM}$, we have the following lemma, which is a direct consequence of the triangle inequality, and of our weighted definition of the Gershgorin radii.
\begin{lemma}\label{lem : Gersh M and Mbar}
Given a bounded operator $\bcM:\ell^1_1\to\ell^1_1$, for all $n\in\N_0$ we define
\begin{align}\label{eq : def barn}
    \bar{r}_n \bydef     \frac{1}{\xi_n}\sum_{k\in \mathbb{N}_0, ~k\neq n} 
   \left|\bcM_{k,n}\right| \xi_k.
\end{align}
For any $n\in\N_0$, if $\frac{1}{\xi_n}\Vert \cM_{\cdot,n} - \bcM_{\cdot,n}\Vert_1 \leq \varepsilon_n$, then
\begin{align*}
    \B_{r_n}(\cM_{n,n}) \subset \B_{\bar{r}_n+\varepsilon_n}(\bcM_{n,n}).
\end{align*}
\end{lemma}
\begin{proof}
 For any $z\in \B_{r_n}(\cM_{n,n})$ and $n \in \N_0$, we estimate
 \begin{align*}
     \vert z - \bcM_{n,n}\vert &\leq r_n + \vert \cM_{n,n} - \bcM_{n,n}\vert  \\
     &= \frac{1}{\xi_n}\sum_{k\in \N_0,\ k\neq n}  \vert \cM_{k,n} \vert \xi_k +\frac{\xi_n}{\xi_n}\, \vert \cM_{n,n} - \bcM_{n,n}\vert  \\
     &\leq  \frac{1}{\xi_n}\sum_{k\in \N_0,\ k\neq n}  \vert \bcM_{k,n} \vert \xi_k  + \frac{1}{\xi_n}\,  \sum_{k\in \N_0} \vert \cM_{k,n} - \bcM_{k,n} \vert \xi_k  \\
     &= \bar{r}_n + \frac{1}{\xi_n} \,\Vert \cM_{\cdot,n} - \bcM_{\cdot,n}\Vert_1. \qedhere
 \end{align*}
\end{proof}
If $\bcM$ is taken as in~\eqref{def : choice of bar M}, $\bcM_{n,n}$ and $\bar{r}_n$ can be computed explicitly, and as soon as we get an explicit bound $\varepsilon_n$ we therefore have a computable enclosure for the Gershgorin disk $\B_{r_n}(\cM_{n,n})$ of $\cM$. An easy way to get such a $\varepsilon_n$ is simply to use that $\xi_{n}^{-1}\Vert \cM_{\cdot,n} - \bcM_{\cdot,n}\Vert_1 \leq \Vert \cM - \bcM\Vert_1$, as we can get a computable upper bound for $\Vert \cM - \bcM\Vert_1$ almost for free from the steady state proof. Indeed, the estimate~\eqref{eq : Z1pZ2} shows that $\Vert D\cF(\tilde{U})^{-1}\Vert_1 \leq \frac{\Vert \cA\Vert_1}{1-(Z_1 + Z_2(r)r)}$, and we then get
\begin{align*}
    \Vert  \cM - \bcM \Vert_1 &= \Vert \cP^{-1} \left(\cI - AD\cF(\tilde{U})\right) D\cF(\tilde{U})^{-1}\cL^{-1}\cP\Vert_1  \\
    &\leq \Vert \cP^{-1} \Vert_1 \left\Vert \cP\right\Vert_1 \, \frac{Z_1 + Z_2(r)r}{1-(Z_1 + Z_2(r)r)} \,\Vert \cA\Vert_1 \Vert \cL^{-1}\Vert_1.
\end{align*}
Admittedly, the above estimate could be somewhat sharpened, but nevertheless, how small $\cM - \bcM$ can be is limited by how small $\cI - AD\cF(\tilde{U})$ is (see~\eqref{eq : Z1pZ2}), and made worse by other factors such as the condition number of $\cP$. In order to successfully apply Theorem~\ref{th: radii polynomial} and obtain the existence of $\tilde{U}$ in $\B_r(\bU)$, we need to have $Z_1 + Z_2(r)r < 1$, but this term does not need to be much smaller than $1$ (in principle it can be made very small by taking $N$ very large, but this then severely increases the computational cost of the proof). Overall, this means the estimate just outlined would have rather poor performances in practice (i.e., $\varepsilon_n$ would be rather large). We therefore derive below a finer estimate for $\varepsilon_n$. We do so by using another (simplified) Newton-Kantorovich argument for solving for the columns of $\cM$, that  still reuses parts of the bounds computed to prove the steady state $\tilde{U}$.

\begin{proposition}\label{prop : finite number gershgorin}
Repeat the assumptions of Theorem~\ref{th: radii polynomial}, assume that~\eqref{eq : bounds general radii polynomial}--\eqref{condition radii polynomial} hold for $\nu=1$, and let $\tilde{U}$ be the unique zero of $\cF$ in $\B_r(\bU)$. Let $\cM$ be the operator defined in~\eqref{def : matrix diag M}, and let $\bcM:\ell^1_1\to\ell^1_1$. 
Finally, for all $n\in\N_0$, let 
    \begin{align}\label{eq : def epsilon n}
        \eps_n \bydef \frac{1}{\xi_n} \dfrac{\Vert \cP^{-1}\Vert_1  \left(\Vert \cA (D\cF(\bU) \cP  \bcM_{\cdot,n} - \frac{1}{\xi_n}\mathcal{L}^{-1}\cP E_n^{(1)})\Vert_1 + Z_2(r)r \, \Vert \cP  \bcM_{\cdot,n}\Vert_1 \right)}{1-(Z_1+Z_2(r)r)}.
    \end{align}
Then,
\begin{equation*}
    \B_{r_n}(\cM_{n,n}) \subset \B_{\bar{r}_n+\varepsilon_n}(\bcM_{n,n}).
\end{equation*}
\end{proposition}
%
\begin{proof}
We are going to show that $\eps_n$ defined in~\eqref{eq : def epsilon n} satisfies $\xi_n^{-1}\Vert \cM_{\cdot,n} - \bcM_{\cdot,n}\Vert_1 \leq \varepsilon_n$, which proves Proposition~\ref{prop : finite number gershgorin} thanks to Lemma~\ref{lem : Gersh M and Mbar}.  

To that end, we introduce a zero-finding problem satisfied by $\cM_{\cdot,n}$. We recall that $E_n^{(1)}$ is the $n$-th element of the canonical Schauder basis of $\ell^1_1$, and in particular that 
    \begin{equation*}
        \cM_{\cdot,n} = \frac{1}{\xi_n} \cM E_n^{(1)} = \frac{1}{\xi_n}\cP^{-1}D\cF(\tilde{U})^{-1}\cL^{-1}\cP E_n^{(1)}.
    \end{equation*}
In other words, 
\begin{equation*}
    D\cF(\tilde{U}) \cP  \cM_{\cdot,n} - \frac{1}{\xi_n}\mathcal{L}^{-1}\cP E_n^{(1)} = 0.
\end{equation*}
If $\bcM_{\cdot,n}$ is an accurate approximation of $\cM_{\cdot,n}$, then we expect
\begin{equation*}
    \delta_n =  D\cF(\tilde{U}) \cP  \bcM_{\cdot,n} - \frac{1}{\xi_n}\mathcal{L}^{-1}\cP E_n^{(1)},
\end{equation*}
to be small, and we are going to estimate $\Vert \cM_{\cdot,n} - \bcM_{\cdot,n}\Vert_1$ in terms of the residual $\delta_n$. In general this can be accomplished thanks to the Newton-Kantorovich Theorem~\ref{th: radii polynomial}, but here $\cM_{\cdot,n}$ is the zero of an affine map, therefore elementary calculations suffice. Indeed, by linearity
\begin{equation*}
    D\cF(\tilde{U}) \cP \left(\bcM_{\cdot,n}-\cM_{\cdot,n} \right) = \delta_n,
\end{equation*}
hence 
\begin{align*}
    \bcM_{\cdot,n}-\cM_{\cdot,n} &= \cP^{-1}D\cF(\tilde{U})^{-1} \left(D\cF(\tilde{U}) \cP  \bcM_{\cdot,n} - \frac{1}{\xi_n}\mathcal{L}^{-1}\cP E_n^{(1)}\right).
\end{align*}
We then use once more estimate~\eqref{eq : Z1pZ2}, which shows that $\cA$ itself is invertible, and that 
\[
    \Vert(\cA D\cF(\tilde{U}))^{-1}\Vert_1 \leq \frac{1}{1-(Z_1+Z_2(r)r)}.
\]
Therefore, we can get rid of the term $D\cF(\tilde{U})$ as follows 
\begin{align*}
    \Vert \bcM_{\cdot,n}-&\cM_{\cdot,n}\Vert_1 \\
    &=\left\Vert \cP^{-1}\left(\cA D\cF(\tilde{U})\right)^{-1} \cA\left(D\cF(\bU) \cP  \bcM_{\cdot,n} - \frac{1}{\xi_n}\mathcal{L}^{-1}\cP E_n^{(1)} + (D\cF(\tilde{U})-D\cF(\bU))\cP  \bcM_{\cdot,n}\right)\right\Vert_1 \\
    &\leq \dfrac{\Vert \cP^{-1}\Vert_1  \left(\Vert \cA (D\cF(\bU) \cP  \bcM_{\cdot,n} - \frac{1}{\xi_n}\mathcal{L}^{-1}\cP E_n^{(1)})\Vert_1 + \Vert \cA(D\cF(\tilde{U})-D\cF(\bU))\Vert_1 \, \Vert \cP  \bcM_{\cdot,n}\Vert_1 \right)}{1-(Z_1+Z_2(r)r)} \\
    &\leq \dfrac{\Vert \cP^{-1}\Vert_1  \left(\Vert \cA (D\cF(\bU) \cP  \bcM_{\cdot,n} - \frac{1}{\xi_n}\mathcal{L}^{-1}\cP E_n^{(1)})\Vert_1 + Z_2(r)r \, \Vert \cP  \bcM_{\cdot,n}\Vert_1 \right)}{1-(Z_1+Z_2(r)r)}. \qedhere
\end{align*}
\end{proof}
The key aspect of Proposition~\ref{prop : finite number gershgorin} is that, as soon as the column $\bcM_{\cdot,n}$ is finite (i.e., has a finite number of nonzero entries), then $\bar{r}_n$ is computable, and so is $\eps_n$ (provided we take $\cA$ of the form~\eqref{eq :decomposition of A}). Therefore, we do get an explicit enclosure of the $n$-th Gershgorin disk $\B_{r_n}(\cM_{n,n})$ of $\cM$. However, in practice $\eps_n$ and $\bar{r}_n$ can only be computed explicitly for finitely many $n$'s, and we need to do some extra work to control the remaining Gershgorin disks. To do so, we go back to our construction in \eqref{def : choice of bar M} and derive the following estimates.

\begin{proposition}\label{prop : tail gershgorin}
Repeat the assumptions of Theorem~\ref{th: radii polynomial}, assume that~\eqref{eq : bounds general radii polynomial}--\eqref{condition radii polynomial} hold for $\nu=1$, and let $\tilde{U}$ be the unique zero of $\cF$ in $\B_r(\bU)$. Let $\cM$ be the operator defined in~\eqref{def : matrix diag M}.
Then, for all $n>N$, recalling definitions~\eqref{def : gamma i k} and~\eqref{eq : eta ki} of $\gamma_{i,k}^{(1)}$ and $\eta^{(1)}_{i,k}$, and defining
    \begin{equation}\label{eq : def Gersh tail}
    \begin{aligned}
        \eps^\infty_{n} &\bydef \frac{1}{2}\dfrac{\Vert \cP^{-1}\Vert_1}{1-(Z_1+Z_2(r)r)} \left(Z_2(r)r \,\eta^{(1)}_{0,n} + \sum_{i=0}^{2m-1} \|\cA \bar{\mathcal{V}}_i\|_1 \,  \eta^{(1)}_{i,n-2m-1} \gamma_{2m,n}^{(1)} + \| \cA \bar{\mathcal{V}}_i \cK_i \pi^{\leq 2m-1}\|_1\, \eta^{(1)}_{0,n} \right)\\
        \bar{r}^\infty_{n} &\bydef \frac{1}{2}\Vert \cP^{-1} \Vert_1 \, \eta^{(1)}_{0,n},
        \end{aligned}
    \end{equation}
we get
\begin{equation*}
    \B_{r_n}(\cM_{n,n})  \subset \B_{\bar{r}^\infty_n+\varepsilon^\infty_n}\left(0\right).
\end{equation*}
\end{proposition}
\begin{proof}
Consider, for each $n>N$, $\bcM_{\cdot,n} \bydef \cP^{-1}\cL^{-1}E_n^{(1)}$ as in \eqref{def : choice of bar M}. According to Proposition~\ref{prop : finite number gershgorin}, 
\begin{equation*}
    \B_{r_n}(\cM_{n,n}) \subset \B_{\bar{r}_n+\varepsilon_n}(\bcM_{n,n}),
\end{equation*}
with $\bar{r}_n$ and $\varepsilon_n$ given by~\eqref{eq : def barn} and~\eqref{eq : def epsilon n}. 
Regarding $\bar{r}_n$, we simply estimate using Lemma \ref{prop : Ki is compact}:
\begin{align*}
   |\bar{\cM}_{n,n}| + \bar{r}_n = \Vert \bcM_{\cdot,n}\Vert_1 = \frac{1}{\xi_n}\Vert \cP^{-1}\cL^{-1}E_n^{(1)} \Vert_1 \leq \frac{1}{2} \Vert \cP^{-1} \Vert_1 \, \Vert \cL^{-1}E_n^{(1)} \Vert_1 \leq \frac{1}{2} \Vert \cP^{-1} \Vert_1 \,\eta^{(1)}_{0,n} = \bar{r}^\infty_n,
\end{align*}
where we used that $\xi_n = 2$ since  $n > N >  2$.
This implies that $\B_{\bar{r}_n+\varepsilon_n}(\bcM_{n,n}) \subset \B_{\bar{r}^\infty_n+\varepsilon_n}(0)$.
Therefore, it now suffices to prove that, for all $n>N$, $\varepsilon_n\leq \varepsilon^\infty_n$. Our specific choice of $\bcM_{\cdot,n}$ together with the fact that $D\cF(\tilde{U}) = \cI + D\cK(\tilde{U})$ and $\cP E_n^{(1)} = E_n^{(1)}$ for $n>N$ yields
\begin{align*}
    \eps_n &= \frac{1}{2}\dfrac{\Vert \cP^{-1}\Vert_1  \left(\Vert \cA (D\cF(\bU) \cL^{-1} E_n^{(1)} - \mathcal{L}^{-1} E_n^{(1)})\Vert_1 + Z_2(r)r \,\Vert \cL^{-1} E_n^{(1)}\Vert_1  \right)}{1-(Z_1+Z_2(r)r)} \\
    &= \frac{1}{2}\dfrac{\Vert \cP^{-1}\Vert_1  \left( \Vert \cA D\cK(\bU) \mathcal{L}^{-1} E_n^{(1)} \Vert_1 + Z_2(r)r \,\Vert\mathcal{L}^{-1} E_n^{(1)}\Vert_1  \right)}{1-(Z_1+Z_2(r)r)} \\
    & \leq \frac{1}{2}\dfrac{\Vert \cP^{-1}\Vert_1   \left( \Vert \cA D\cK(\bU) \mathcal{L}^{-1} E_n^{(1)} \Vert_1 + Z_2(r)r \, \eta^{(1)}_{0,n}  \right)}{1-(Z_1+Z_2(r)r)},
\end{align*}
where we have used the fact that $\cL^{-1} = \cK_0$ to bound from above $\|\cL^{-1} E^{(1)}_n\|_1$. Now, notice that 
\begin{align*}
    \Vert \cA D\cK(\bU) \mathcal{L}^{-1} E_n^{(1)} \Vert_1 \leq \sum_{i=0}^{2m-1} \|\cA \bar{\cV}_i \cK_i \mathcal{L}^{-1} E_n^{(1)}\|_1.
\end{align*}
Recalling \eqref{eq : defLinv}, we have that
\begin{align*}
    \|\cA \bar{\mathcal{V}}_i \cK_i\mathcal{L}^{-1} E_n^{(1)}\|_1 \leq  \| \cA \bar{\mathcal{V}}_i \cK_i \cS^{(2m)} E_n^{(1)}\|_1 +  \| \cA \bar{\mathcal{V}}_i \cK_i \cB^{\dagger}_{2m} \cB \cS^{(2m)} E_n^{(1)}\|_1. 
\end{align*}
Since $\cS^{(2m)} E_n^{(1)} = \pi^{>n-2m-1}\cS^{(2m)} E_n^{(1)}$, we apply Lemma~\ref{lem : compactness S} to get
\begin{align*}
    \| \cA \bar{\mathcal{V}}_i \cK_i \cS^{(2m)} E_n^{(1)}\|_1 &\leq  \| \cA \bar{\mathcal{V}}_i \cK_i\pi^{>n-2m-1}\|_1 \,\| \cS^{(2m)} E_n^{(1)}\|_1\\
    &\leq \|\cA \bar{\mathcal{V}}_i\|_1 \, \eta^{(1)}_{i,n-2m-1} \,\gamma^{(1)}_{2m,n}.
\end{align*}
On the other hand, since $\cB^{\dagger}_{2m} = \pi^{\leq 2m-1}\cB^{\dagger}_{2m}$, we use Proposition \ref{prop : Ki is compact} to get  
\begin{equation*}
    \| \cA \bar{\mathcal{V}}_i \cK_i \cB^{\dagger}_{2m} \cB \cS^{(2m)} E_n^{(1)}\|_1 \leq \| \cA \bar{\mathcal{V}}_i \cK_i \pi^{\leq 2m-1} \|_1   \|\cB^{\dagger}_{2m} \cB \cS^{(2m)} E_n^{(1)}\|_1
    \leq \| \cA \bar{\mathcal{V}}_i \cK_i \pi^{\leq 2m-1} \|_1 \eta^{(1)}_{0,n}.
\end{equation*}
This concludes the proof.
\end{proof}
With Proposition~\ref{prop : tail gershgorin}, we have another set of explicit and computable enclosures for the Gershgorin disks of $\cM$. These enclosures may be slightly less accurate than the one given in Proposition~\ref{prop : finite number gershgorin}, but they have the advantage of having more explicit dependency with respect to $n$, therefore we are going to be able to use them to derive asymptotic statements valid for all $n$ large enough.

Based only on Proposition~\ref{prop : finite number gershgorin} and Proposition~\ref{prop : tail gershgorin}, we are not able to count the number of unstable eigenvalues of $\cM$, as the balls $\B_{\bar{r}_n+\varepsilon_n}(\bcM_{n,n})$ might intersect both halves of the complex plane for $n$ large enough. However, what Proposition~\ref{prop : tail gershgorin} does provide is a quantitative lower bound on how fast the eigenvalues of $\cM$ go to $0$. 
On the other hand, going back to the eigenproblem~\eqref{eq : evp functions} in its initial form, we do expect the linearized operator to be sectorial, and in particular there should be an a priori bound on the modulus of potential unstable eigenvalues. In other words, recalling the link with the spectrum of $\cM$ given in Lemma~\ref{lem : equivalent eig problem}, arbitrarily small eigenvalues of $\cM$ should all be stable. Provided such statement can be made quantitative, we can then determine the spectral stability of $\tilde{U}$ using the following statement.


\begin{theorem}\label{th : conclusion stability}
    Repeat the assumptions of Theorem~\ref{th: radii polynomial}, assume that~\eqref{eq : bounds general radii polynomial}--\eqref{condition radii polynomial} hold for $\nu=1$, and let $\tilde{U}$ be the unique zero of $\cF$ in $\B_r(\bU)$. Assume that there exists $\lambda_{\max}>0$ such that any eigenvalue $\lambda$ in~\eqref{eq : evp functions} (with $\tilde{v} = \cG_0(\cL^{-1}\tilde{U})$) having nonnegative real part satisfies $\vert \lambda\vert \leq \lambda_{\max}$.    
    Assume further that:
    \begin{itemize}
        \item There are exactly $n_u\in\N_0$ Gershgorin disks of $\cM$ included in the unstable half space $\{z\in\C,\ \Re(z) >0\}$. 
        \item All the other Gershgorin disks of $\cM$ do not intersect the previous ones, and are contained in $\{z\in\C,\ \Re(z) <0\}\bigcup \B_{\lambda_{\max}^{-1}}(0)$.
    \end{itemize}
Then, the linearization of~\eqref{eq : parabolic PDE original} at the steady state $\tilde{v}$ has exactly $n_u$ unstable eigenvalues.
\end{theorem}
\begin{proof}
        According to the Gershgorin theorem, recalled in our context in Lemma~\ref{lem : gershgorin matrix}, the $n_u\in\N_0$ Gershgorin disks of $\cM$ included in the unstable half space $\{z\in\C,\ \Re(z) >0\}$ and not intersecting any of the other disks yield $n_u$ unstable eigenvalues of $\cM$. Moreover, by Lemma~\ref{lem : equivalent eig problem}, eigenvalues $\lambda$ of~\eqref{eq : evp functions} correspond to eigenvalues $\lambda^{-1}$ of $\cM$. Therefore, any eigenvalue of $\cM$ belonging to $\B_{\lambda_{\max}^{-1}}(0)$ must in fact be stable. Since all the remaining Gershgorin disks of $\cM$ are contained in $\{z\in\C,\ \Re(z) <0\}\bigcup \B_{\lambda_{\max}^{-1}}(0)$, all the other eigenvalues of $\cM$ are stable. We have proven that $\cM$ has exactly $n_u$ unstable eigenvalues, and, using once more Lemma~\ref{lem : equivalent eig problem}, there are exactly $n_u$ unstable eigenvalues in~\eqref{eq : evp functions}.  
\end{proof}
In practice, the constant $\lambda_{\max}$ has to be computed on a case by case basis (see Section~\ref{sec : applications} for two concrete examples). Moreover, in order to apply Theorem~\ref{th : conclusion stability}, we need explicit control on all the Gershgorin disks of $\cM$, but this is exactly what Proposition~\ref{prop : finite number gershgorin} and Proposition~\ref{prop : tail gershgorin} provide.
\begin{example}\label{ex : ccl stab}
    In practice, provided $N$ is taken large enough, what typically happens is that:
    \begin{enumerate}[label=(\roman*)]
        \item There are $n_u\in\{0,\ldots,N\}$ of the balls $\B_{\bar{r}_n+\varepsilon_n}(\bcM_{n,n})$ given by Proposition~\ref{prop : finite number gershgorin} included in the unstable half space $\{z\in\C,\ \Re(z) >0\}$,
        \item The remaining $N+1-n_u\in\N_0$ balls $\B_{\bar{r}_n+\varepsilon_n}(\bcM_{n,n})$ given by Proposition~\ref{prop : finite number gershgorin} are included in the stable half space $\{z\in\C,\ \Re(z) <0\}$,
        \item All the balls $\B_{\bar{r}^\infty_n+\varepsilon^\infty_n}(0)$ given by Proposition~\ref{prop : tail gershgorin} are included in $\B_{\lambda_{\max}^{-1}}(0)$.
    \end{enumerate}
The first two assumptions can be checked explicitly because the sets $\B_{\bar{r}_n+\varepsilon_n}(\bcM_{n,n})$ are all computable. Regarding the third one, note that the formula for $\bar{r}^\infty_n$ and $\varepsilon^\infty_n$ depend explicitly on $n$, and are in fact decreasing with $n$, therefore it suffices to check whether $\bar{r}^\infty_{N+1}+\varepsilon^\infty_{N+1} \leq \lambda_{\max}^{-1}$ (if not, one should use a larger $N$). These assumptions taken together allow us to apply Theorem~\ref{th : conclusion stability}, and to conclude that there are exactly $n_u$ unstable eigenvalues. In other words, as one might expect, all the stability information is contained in the $N$ first Gershgorin disks.
\end{example}

\subsection{Enclosure of the spectrum directly based on the generalized eigenproblem }\label{ssec : second enclosure spectrum}

In this section, we discuss an alternative method for enclosing the eigenvalues in~\eqref{eq : evp functions}. 
This approach still uses a finite dimensional approximation of $\cM$ to define approximate eigenvalues, but then differs from the one presented in Section~\ref{ssec: Gersh}, in that we directly enclose eigenvalues of the generalized eigenvalue problem \eqref{eq : generalized evp sequences}, rather than first enclosing the eigenvalues of $\cM$. This might seem like an inconsequential change at first glance, but it turns out to make a significant difference on some examples. In particular, it will allow us to deal with situations where the tail estimates from Proposition \ref{prop : tail gershgorin} are not tight enough in order to conclude with Theorem \ref{th : conclusion stability}, i.e., when the disks obtained in Proposition~\ref{prop : tail gershgorin} are not all contained in $\B_{\lambda_{\max}^{-1}}(0)$. 

Similarly as what was achieved in the previous section, we introduce a pseudo-diagonalization 
\begin{align*}
  \mathcal{M}_0 \bydef   \mathcal{P}_0^{-1} \mathcal{A}_0 \pi^{\leq N} \mathcal{L}^{-1} \pi^{\leq N}\mathcal{P}_0
\end{align*}
for some invertible (finite-dimensional) operator $\mathcal{P}_0 : \pi^{\leq N} \ell^1_1\to \pi^{\leq N} \ell^1_1$ with inverse $\mathcal{P}_0^{-1}$, where $\cA_0$ corresponds to the finite part of the operator $\cA$ used for establishing the existence of the steady state (see~\eqref{eq :decomposition of A}). In contrast to the $\mathcal{M}$ considered previously, $\cM_0$ is finite dimensional, and is defined via the explicit operator $\mathcal{A}_0$, and not $D\mathcal{F}(\tilde{U})^{-1}$. In particular, $\mathcal{M}_0$ is fully computable and does not require an additional approximation (such as $\bar{\mathcal{M}}$ to approximate $\cM$). Since $\cM_0$ is finite dimensional, we can only hope to use it to approximate eigenvalues of~\eqref{eq : generalized evp sequences} of moderate amplitude. As was the case in Section~\ref{ssec: Gersh}, this will then be combined with a priori estimates on the spectrum.

We define $S_0$ as the diagonal part of the matrix $\mathcal{M}_0$ and $R_0$ the rest of the entries, that is 
 \begin{align*}
     S_0 \bydef \diag(\mathcal{M}_0) \text{ and }  R_0 \bydef \mathcal{M}_0 -S_0.
 \end{align*}
In practice, we expect $R_0$ to be really small, because $\cP_0$ will be constructed with numerically obtained eigenvectors of $\mathcal{A}_0 \pi^{\leq N} \mathcal{L}^{-1} \pi^{\leq N}$.
Moreover, we denote by $(\mu_i)_{i\in \mathbb{N}_0}$  the diagonal entries of $S_0$, and for each $i \in \mathbb{N}_0$, we define 
\begin{equation}
\label{eq:deflambdai}
    \lambda_i \bydef \frac{1}{\mu_i},
\end{equation}
with the convention that $\lambda_i = \infty$ if $\mu_i = 0.$
Now, we provide below the central result of this section which allows to enclose any eigenvalues of \eqref{eq : generalized evp sequences} of moderate amplitude in a neighborhood of one of the explicit $\lambda_i'$s. Moreover, we provide computable estimates for this enclosure, which are compatible with rigorous numerics. 
\begin{theorem}\label{th : enclosure spectrum second homotopy}
Repeat the assumptions of Theorem~\ref{th: radii polynomial}, with $\cA$ of the form~\eqref{eq :decomposition of A}, assume that~\eqref{eq : bounds general radii polynomial}--\eqref{condition radii polynomial} hold for $\nu=1$, and let $\tilde{U}$ be the unique zero of $\cF$ in $\B_r(\bU)$.

Let $\lambda_{\max} >0$ be fixed, and let $\lambda\in\C$ be an eigenvalue of the generalized eigenproblem~\eqref{eq : generalized evp sequences} such that $\vert\lambda\vert \leq \lambda_{\max}$.
Let $\epsilon$ be defined as 
\begin{align}\label{eq : epsilon second approach}
\epsilon \bydef   \sum_{i=0}^{2m-1}\left( 2\eta^{(1)}_{i,N+1}\sum_{k > N-2m+i} |(\bar{\mathcal{V}}_i)_k|  + \|\bar{\mathcal{V}}_i\|_1 \gamma_{2m-i,N+1}^{(1)}\right) + \lambda_{\max} \gamma_{2m,N+1}^{(1)},
\end{align}
with $\gamma_{i,k}^{(1)}$ given in~\eqref{def : gamma i k}, $\bar{\mathcal{V}}_i$ in~\eqref{eq: bcVi}, and $\eta^{(1)}_{i,k}$ in~\eqref{eq : eta ki}.

Let $Z_{1,\cP_0}(r)$, $Z_{2,\cP_0}(r)$, $Z_{2,\leq N}(r)$, and $Z_{2,> N}(r)$ be bounds satisfying
\begin{align*}
Z_{1,\cP_0}(r) &\geq \Vert\mathcal{P}_0^{-1}\left(\pi^{\leq N} - \mathcal{A}_0\pi^{\leq N}D\cF(\bU)\pi^{\leq N}\right)\mathcal{P}_0\Vert_1, \\
Z_{2,\cP_0}(r)r &\geq \Vert\mathcal{P}_0^{-1}\mathcal{A}_0\pi^{\leq N}(D\cF(\bU) - D\cF(\tilde{U}))\pi^{\leq N}\mathcal{P}_0\Vert_1,\\
Z_{2,\leq N}(r)r &\geq \Vert\mathcal{A}_0\pi^{\leq N}(D\cF(\bU) - D\cF(\tilde{U}))\pi^{\leq N}\Vert_1,\\
Z_{2,> N}(r)r  &\geq \Vert \pi^{>N} (D\cK(\bU) - D\cK(\tilde{U})) \pi^{>N}\Vert_1.
\end{align*}
Let also $\beta_{0,1}$, $\beta_{0,2}$, $\beta_{1,1}$, $\beta_{1,2}$, $\beta_{1,3}$, $\beta_{1,4}$, $\beta_{2,1}$, and $\beta_{2,2}$ be such that
\begin{align*}
\beta_{0,1} &\geq  \|\pi^{>N} D\mathcal{K}(\bU) \pi^{\leq N}  \mathcal{P}_0\|_1 + \|\mathcal{P}_0\|_1 Z_{2,>N}(r)r,\\
\beta_{0,2} &\geq \|\mathcal{P}_0^{-1}\mathcal{A}_0 \pi^{\leq N}    D\mathcal{K}(\bU)  \pi^{>N} D\mathcal{K}(\bU) \pi^{\leq N}  \mathcal{P}_0\|_1 +  \|\mathcal{P}_0^{-1}\mathcal{A}_0   \pi^{\leq N}  D\mathcal{K}(\bU)  \pi^{>N} \|_1  \|\mathcal{P}_0\|_1 Z_{2,>N}(r)r,\\ 
\beta_{1,1} &\geq \eta^{(1)}_{0,N+1} \|\mathcal{P}_0^{-1}\mathcal{A}_0\|_1   ( \|\pi^{>N} D\mathcal{K}(\bU) \pi^{\leq N}  \mathcal{P}_0\|_1+ \|\mathcal{P}_0\|_1 Z_{2,>N}(r)r),\\
\beta_{1,2} &\geq \|\pi^{>N} \mathcal{L}^{-1} \pi^{\leq N}  \mathcal{P}_0\|_1,\\
\beta_{1,3} &\geq \|\mathcal{P}_0^{-1}\mathcal{A}_0  \pi^{\leq N}   D\mathcal{K}(\bU)  \pi^{>N} \mathcal{L}^{-1} \pi^{\leq N}  \mathcal{P}_0\|_1,\\
\beta_{1,4} &\geq \|\mathcal{P}_0^{-1}\mathcal{A}_0 \pi^{\leq N} D\mathcal{K}(\bU)  \pi^{>N}\|_1 + \|\mathcal{P}_0^{-1}\|_1 Z_{2,\leq N}(r)r,\\
\beta_{2,1} &\geq \|\mathcal{P}_0^{-1}\mathcal{A}_0  \pi^{\leq N}  \mathcal{L}^{-1}  \pi^{>N} \mathcal{L}^{-1} \pi^{\leq N}  \mathcal{P}_0\|_1,\\
\beta_{2,2} &\geq \eta^{(1)}_{0,N+1}\|\mathcal{P}_0^{-1}\mathcal{A}_0\|_1.
\end{align*}
Furthermore, let us assume that $\epsilon < 1$, and define $\beta_0$, $\beta_1$, and $\beta_2$ as 
\begin{align*}
    \beta_2 &\bydef \beta_{2,1} + \frac{\epsilon}{1-\epsilon} \beta_{1,2}\beta_{2,2},\\
    \beta_1 &\bydef \beta_{1,1} + \beta_{1,3} + \|\mathcal{P}_0^{-1}\|_1 Z_{2,\leq N}(r)r \beta_{1,2} + \frac{\epsilon}{1-\epsilon}(\beta_{1,2} \beta_{1,4} + \beta_{0,1}\beta_{2,2}),\\
    \beta_0 &\bydef \beta_{0,2} + \|\mathcal{P}_0^{-1}\|_1 Z_{2,\leq N}(r)r \beta_{0,1} + \frac{\epsilon}{1-\epsilon} \beta_{0,1}\beta_{1,4}.
\end{align*}
 Finally, let $\delta$ be defined as 
 \begin{align*}
     \delta \bydef \lambda_{\max} \|R_0\|_1 + \beta_2 \lambda_{\max}^2 + \beta_1 \lambda_{\max} + \beta_0 + Z_{1,\cP_0} + Z_{2,\cP_0}(r)r,
 \end{align*}
and assume that the truncation parameter $N$ satisfies $N\geq 4m$.
Then, there exists $j \in \{0, \dots, N\}$ such that 
\begin{align}\label{eq : enclosure eigenvalue lambda max}
    |\lambda_j - \lambda| \leq |\lambda_j| \delta.
\end{align}
In addition, let $\fB_j \bydef \B_{\delta|\lambda_j|}(\lambda_j)$  for all $j \in \{0, \dots, N\}$, let $k \in \mathbb{N}$ and let $I \subset \{0, \dots, N\} $ such that $|I| = k$.
If the set $\cup_{j \in I} \fB_j$ is included in the interior of $\B_{\lambda_{\max}}(0)$ and is disjoint from $\cup_{j \in  \{0, \dots, N\} \setminus I} \fB_j $, then $\cup_{j \in I} \fB_j$ contains exactly $k$ generalized eigenvalues of~\eqref{eq : generalized evp sequences}, counted with algebraic multiplicity. 
\end{theorem}

Before giving the proof of Theorem~\ref{th : enclosure spectrum second homotopy}, let us briefly discuss the many terms involved there. First, note that, provided $N$ is taken large enough, we do expect to get $\epsilon\ll 1$, in a quantitative way, thanks to the quantitative compactness estimates provided by Lemma~\ref{lem : compactness S} and Proposition~\ref{prop : Ki is compact}. Next, $Z_{1,\cP_0}$ is expected to be small because $\cA_0$ is defined as a numerical inverse of $\pi^{\leq N} D\cF(\bU)\pi^{\leq N}$, while $rZ_{2,\cP_0}$, $rZ_{2,\leq N}$ and $rZ_{2,> N}$ should be small because $\bU$ is very close to $\tilde{U}$, and hence $D\cF(\bU)-D\cF(\tilde{U})$ should be very small. Also note that, up to the conjugation by $\cP_0$, these terms are related to quantities that we already needed to estimate for the application of Theorem~\ref{th: radii polynomial}. Then, each of the $\beta_{i,j}$ terms are also expected to be small, because they involve some of the small terms already mentioned, and because of the compactness of $D\cK(\bu)$ and of $\cL^{-1}$ (see again Lemma~\ref{lem : compactness S} and Proposition~\ref{prop : Ki is compact} for quantitative estimates). Moreover, most of the estimates $\beta_{i,j}$ amount to the computation of operator norms of finite dimensional operators (we recall that $\pi^{>N}D\cK(\bU)\pi^{\leq N}$ has finite range according to~\eqref{eq:Kifinite}), and hence are computable. The only exception is the $\mathcal{P}_0^{-1}\mathcal{A}_0 \pi^{\leq N} D\mathcal{K}(\bU)  \pi^{>N}$ term (because of the infinite rows in $D\mathcal{K}(\bU)  \pi^{>N}$, see~\eqref{eq:Kitail}), but we still have a computable estimate for it, namely 
\begin{align*}
    \Vert \mathcal{P}_0^{-1}\mathcal{A}_0 \pi^{\leq N} D\mathcal{K}(\bU)  \pi^{>N}\Vert_1 \leq \Vert \cP_0^{-1}\Vert_1 Z_{1,3},
\end{align*}
where we again refer to Lemma~\ref{lem : Z1 bound general} for more details, including the definition of $Z_{1,3}$ and the proof of this estimate. In summary, we can get computable estimates for all the $\beta_{i,j}$ required in Theorem~\ref{th : enclosure spectrum second homotopy}, and hence for $\beta_0$, $\beta_1$ and $\beta_2$, which should all be small.

Finally, recalling that $R_0 = \cM_0 - \diag(\cM_0)$, we do expect $\delta$ to be reasonably small, and hence~\eqref{eq : enclosure eigenvalue lambda max} to provide a reasonably tight enclosure of the eigenvalues, at least in relative error. However, note that these enclosures become worse when $\lambda_{\max}$ increases, and may also become large in absolute distance if $\lambda_j$ itself has a large magnitude. 
\begin{proof}
Let $U\in\ell^1_1$, $U\neq 0$, be an eigenvector associated to $\lambda$, i.e., such that
\begin{align}\label{eq : eig original}
   U+ D\mathcal{K}(\tilde{U}) U = \lambda \mathcal{L}^{-1} U.
\end{align}
The first main idea of the proof is that, if $\lambda_{\max}$ is not too large (or if $N$ is large enough), we expect the eigenvector $U$ to be mostly concentrated in its finite dimensional part $\pi^{\leq N} U$. We start by making this statement quantitative, by explicitly expressing the tail $\pi^{> N} U$ in terms of $\pi^{\leq N} U$. To that end, applying $\pi^{>N}$ to~\eqref{eq : eig original}  and separating the contributions of $\pi^{\leq N} U$ and $\pi^{> N} U$, we get
\begin{align*}
   \pi^{>N} U+  \pi^{>N} D\mathcal{K}(\tilde{U})  \pi^{>N} U - \lambda  \pi^{>N} \mathcal{L}^{-1} \pi^{>N}U =  - \pi^{>N} D\mathcal{K}(\tilde{U})  \pi^{\leq N} U + \lambda  \pi^{>N} \mathcal{L}^{-1} \pi^{\leq N}U,
\end{align*}
which we rewrite
\begin{align}\label{eq : tail of eigen}
    B_\lambda \pi^{>N}U = C_\lambda\pi^{\leq N}U,
\end{align}
with
\begin{align}
    B_\lambda &\bydef \pi^{>N}+  \pi^{>N} D\mathcal{K}(\tilde{U})  \pi^{>N} - \lambda  \pi^{>N} \mathcal{L}^{-1} \pi^{>N} : \pi^{>N}\ell^1_1 \to \pi^{>N}\ell^1_1, \nonumber\\
    C_\lambda &\bydef \lambda  \pi^{>N} \mathcal{L}^{-1} \pi^{\leq N} - \pi^{>N} D\mathcal{K}(\tilde{U})  \pi^{\leq N} : \pi^{\leq N}\ell^1_1 \to \pi^{>N}\ell^1_1. \label{eq:defClambda}
\end{align}
We first show that $B_\lambda: \pi^{>N}\ell^1_1 \to \pi^{>N}\ell^1_1$ is invertible, by showing that it is close to $\pi^{>N}$. We start by estimating
\begin{align*}
    \|\pi^{>N} D\mathcal{K}(\tilde{U})  \pi^{>N}\|_1  &\leq  \|\pi^{>N} D\mathcal{K}(\bU)  \pi^{>N}\|_1 +  \|\pi^{>N} (D\mathcal{K}(\tilde{U}) - D\mathcal{K}(\bU))   \pi^{>N}\|_1 \\
    &\leq \|\pi^{>N} D\mathcal{K}(\bU)  \pi^{>N}\|_1 +  Z_{2,>N}(r)r.
\end{align*}
Recalling that $D\cK(\bU) = \sum_{i=0}^{2m-1} \bcV_i\cK_i$, we use \eqref{eq : estimate Z12} with $\nu=1$ to get
\begin{align*}
    \| \pi^{>N} \bar{\mathcal{V}}_i \mathcal{K}_i  \pi^{>N}\|_1 \leq 2\eta^{(1)}_{i,N+1}\sum_{k > N-2m+i} |(\bar{\mathcal{V}}_i)_k|  + \|\bar{\mathcal{V}}_i\|_1 \gamma_{2m-i,N+1}^{(1)},
\end{align*}
and we have obtained that 
\begin{align*}
    \|\pi^{>N} D\mathcal{K}(\tilde{U})  \pi^{>N}\|_1 \leq  Z_{2,>N}(r)r + \sum_{i=0}^{2m-1} \left(2\eta^{(1)}_{i,N+1}\sum_{k > N-2m+i} |(\bar{\mathcal{V}}_i)_k|  + \|\bar{\mathcal{V}}_i\|_1 \gamma_{2m-i,N+1}^{(1)}\right). 
\end{align*}

Moreover, using once again Lemma \ref{lem : compactness S} we have that 
\begin{align*}
    \|\lambda \pi^{>N} \mathcal{L}^{-1} \pi^{>N}\|_1 \leq \lambda_{\max} \gamma_{2m,N+1}^{(1)}.
\end{align*}
Overall, we have obtained that 
\begin{align}\label{eq : epsilon upper}
    \| B_\lambda - \pi^{>N} \|_1 \leq \epsilon,
\end{align}
where $\epsilon$ is given in \eqref{eq : epsilon second approach}.
Since we assumed $\epsilon <1$, we indeed get that $ B_\lambda: \pi^{>N}\ell^1_1\to \pi^{>N}\ell^1_1$ is boundedly invertible, and that
\begin{align}\label{eq : inv of Blambda}
    B^{-1}_\lambda = \sum_{k=0}^\infty \left( \pi^{>N} D\mathcal{K}(\tilde{U})  \pi^{>N} - \lambda  \pi^{>N} \mathcal{L}^{-1} \pi^{>N}\right)^k.
\end{align}
Going back to \eqref{eq : tail of eigen}, we get
\begin{align}\label{eq : tail of eigen with finite part}
    \pi^{>N} U = B_\lambda^{-1} C_\lambda \pi^{\leq N} U.
\end{align}
That is, we have identified the relationship between the tail of the eigenvector $\pi^{>N}U$ and the finite part $\pi^{\leq N}U$, and we expect $\Vert B_\lambda^{-1} C_\lambda\Vert_1$ to be small (this will be quantified later on).

The second part of the proof consists in obtaining an equation for $\pi^{\leq N}U$ only, then in approximately diagonalizing the main part of that equation, and finally in explicitly estimating all the remainder terms in order to obtain an enclosure on the eigenvalue $\lambda$. 
We first apply $\pi^{\leq N}$ to~\eqref{eq : eig original}, or equivalently to
\begin{equation*}
    (D\cF(\tilde{U})-\lambda\cL^{-1})U = 0,
\end{equation*}
and then separate the contributions of $\pi^{\leq N} U$ and $\pi^{> N} U$, and use~\eqref{eq : tail of eigen with finite part}, which yields
\begin{align}\label{eq : finite of eigen}
    \pi^{\leq N}(D\cF(\tilde{U})-\lambda\cL^{-1})\pi^{\leq N}U &= -\pi^{\leq N}(D\cF(\tilde{U})-\lambda\cL^{-1})\pi^{> N}U \nonumber\\
   &= G_\lambda \pi^{\leq N}U,
\end{align}
where 
\begin{align}\label{eq : expression Glambda}
    G_\lambda \bydef \pi^{\leq N}\left(\lambda   \mathcal{L}^{-1}  - D\mathcal{K}(\tilde{U}) \right)\pi^{> N}B_\lambda^{-1} C_\lambda  : \pi^{\leq N}\ell^1_1 \to \pi^{\leq N}\ell^1_1.
\end{align}
Then, applying $\cA_0$ to~\eqref{eq : finite of eigen} and reorganizing the terms to take into account that 
\[
\cA_0\approx \left(\pi^{\leq N}D\cF(\tilde U)\pi^{\leq N}\right)^{-1},
\]
we get
\begin{align}
\label{eq:finiteeigen}
    \left(\pi^{\leq N} - \lambda \mathcal{A}_0 \pi^{\leq N} \mathcal{L}^{-1} \pi^{\leq N}\right) \pi^{\leq N}U = \left(\mathcal{A}_0G_\lambda  + (\pi^{\leq N} - \mathcal{A}_0\pi^{\leq N}D\cF(\tilde{U})\pi^{\leq N})\right)\pi^{\leq N}U.
\end{align}
Next, recall that $\mathcal{A}_0 \pi^{\leq N} \mathcal{L}^{-1} \pi^{\leq N} = \mathcal{P}_0\mathcal{M}_0\mathcal{P}_0^{-1}$ and $\mathcal{M}_0 = S_0 + R_0$, with $S_0 = \diag(\mathcal{M}_0)$. Using these notations together with $V \bydef \mathcal{P}_0^{-1}U$, equation~\eqref{eq:finiteeigen} becomes
\begin{align}\label{eq : proof step 2}
    \left(\pi^{\leq N} - \lambda S_0 \right)\pi^{\leq N}V = \left(\lambda R_0  + \mathcal{P}_0^{-1}\left(\mathcal{A}_0G_\lambda  + (\pi^{\leq N} - \mathcal{A}_0 \pi^{\leq N}D\cF(\tilde{U})\pi^{\leq N})\right)\cP_0\right)\pi^{\leq N}V.
\end{align}

We expect the right-hand side of~\eqref{eq : proof step 2} to be small, and are going quantify this below. Since $S_0$ is diagonal, identity~\eqref{eq : proof step 2} therefore shows that $\lambda$ is close to one of the $\lambda_i$ defined in~\eqref{eq:deflambdai}, and how close depends on how small the right-hand side of~\eqref{eq : proof step 2} actually is. More precisely,
\begin{align}\label{eq : proof step 3}
   \| \left(\pi^{\leq N} - \lambda S_0 \right)\pi^{\leq N}V\|_1 = |(1-\mu_0\lambda)V_0| + 2 \sum_{n =1}^{N} |(1-\mu_n\lambda)V_n| \geq  \min_{i \leq N} |1 - \mu_i \lambda| \|\pi^{\leq N}V\|_1,
\end{align}
and therefore
\begin{align}
\label{eq:estminmulambda}
   \min_{i \leq N} |1 - \mu_i \lambda| \leq |\lambda| \|R_0\|_1 + \|\mathcal{P}_0^{-1}\mathcal{A}_0G_\lambda \mathcal{P}_0\|_1 + \|\mathcal{P}_0^{-1}(\pi^{\leq N} - \mathcal{A}_0\pi^{\leq N}D\cF(\tilde{U})\pi^{\leq N})\mathcal{P}_0\|_1.
\end{align}
It remains to estimate each term on the right-hand side of~\eqref{eq:estminmulambda}.
First, we have 
\begin{align*}
   \|\mathcal{P}_0^{-1}(\pi^{\leq N} - \mathcal{A}_0\pi^{\leq N}D\cF(\tilde{U})\pi^{\leq N})\mathcal{P}_0\|_1 & \leq 
    \|\mathcal{P}_0^{-1}(\pi^{\leq N} - \mathcal{A}_0\pi^{\leq N}D\cF(\bU)\pi^{\leq N})\mathcal{P}_0\|_1 \\
    &\quad + \|\mathcal{P}_0^{-1}\mathcal{A}_0\pi^{\leq N}(D\cF(\bU) - D\cF(\tilde{U}))\pi^{\leq N}\mathcal{P}_0\|_1 \\
    &\leq Z_{1,\cP_0} + Z_{2,\cP_0}(r)r.
\end{align*}
 Next, we focus on the term $\|\mathcal{P}_0^{-1}\mathcal{A}_0G_\lambda \mathcal{P}_0\|_1$. Recalling the definition of $G_\lambda$ in \eqref{eq : expression Glambda}, splitting $B_\lambda^{-1}$ from~\eqref{eq : inv of Blambda} as
\begin{align*}
    B_\lambda^{-1} = \pi^{>N} + \sum_{k=1}^\infty ( \pi^{>N} D\mathcal{K}(\tilde{U})  \pi^{>N} - \lambda  \pi^{>N} \mathcal{L}^{-1} \pi^{>N})^k,
\end{align*}
and using~\eqref{eq : epsilon upper}, we have
\begin{align}\label{eq : proof term to estimate}
\|\mathcal{P}_0^{-1}\mathcal{A}_0G_\lambda \mathcal{P}_0\|_1 &= \|\mathcal{P}_0^{-1}\mathcal{A}_0\pi^{\leq N}(\lambda   \mathcal{L}^{-1}  -  D\mathcal{K}(\tilde{U}) )\pi^{>N}B_\lambda^{-1} C_\lambda  \mathcal{P}_0\|_1 \\
    &\leq \|\mathcal{P}_0^{-1}\mathcal{A}_0\pi^{\leq N}(\lambda   \mathcal{L}^{-1}  -  D\mathcal{K}(\tilde{U})  )\pi^{>N} C_\lambda  \mathcal{P}_0\|_1 \nonumber\\
    &\quad  + \frac{\epsilon}{1-\epsilon} \|\mathcal{P}_0^{-1}\mathcal{A}_0\pi^{\leq N}(\lambda  \mathcal{L}^{-1}  -  D\mathcal{K}(\tilde{U})  )\pi^{>N}\|_1  \|
    C_\lambda \mathcal{P}_0\|_1, \nonumber
\end{align}
and we now estimate the two terms on the right-hand side separately. For the first term, we use the triangle inequality
\begin{align}\label{eq : step 3 proof Th}
\|\mathcal{P}_0^{-1}\mathcal{A}_0\pi^{\leq N}(\lambda   \mathcal{L}^{-1}  -  D\mathcal{K}(\tilde{U})  )\pi^{>N} C_\lambda  \mathcal{P}_0\|_1 &\leq |\lambda| \|\mathcal{P}_0^{-1}\mathcal{A}_0  \pi^{\leq N} \mathcal{L}^{-1}  \pi^{>N} C_\lambda  \mathcal{P}_0\|_1 \nonumber \\
&\quad + \|\mathcal{P}_0^{-1}\mathcal{A}_0  \pi^{\leq N}   D\mathcal{K}(\tilde{U})  \pi^{>N} C_\lambda  \mathcal{P}_0\|_1,
\end{align}
and again estimate the two resulting terms separately. Recalling the definition of $C_\lambda$ in~\eqref{eq:defClambda}, we have
\begin{align*}
& \|\mathcal{P}_0^{-1}\mathcal{A}_0 \pi^{\leq N}  \mathcal{L}^{-1}  \pi^{>N} C_\lambda  \mathcal{P}_0\|_1 \\
& \quad \leq |\lambda| \|\mathcal{P}_0^{-1}\mathcal{A}_0 \pi^{\leq N}  \mathcal{L}^{-1}  \pi^{>N} \mathcal{L}^{-1} \pi^{\leq N}  \mathcal{P}_0\|_1  + \|\mathcal{P}_0^{-1}\mathcal{A}_0 \pi^{\leq N}  \mathcal{L}^{-1}  \pi^{>N}\|_1 \| \pi^{>N} D\mathcal{K}(\tilde{U}) \pi^{\leq N}  \mathcal{P}_0\|_1\\
    & \quad \leq   \beta_{2,1}|\lambda| + \|\mathcal{P}_0^{-1}\mathcal{A}_0 \|_1 \eta^{(\nu)}_{0,N+1} \left(\| \pi^{>N}D\mathcal{K}(\bU) \pi^{\leq N}  \mathcal{P}_0\|_1 +\| \pi^{>N} (D\mathcal{K}(\bU) - D\mathcal{K}(\tilde{U})) \pi^{\leq N}  \mathcal{P}_0\|_1\right)\\
    & \quad \leq  \beta_{2,1}|\lambda| + \|\mathcal{P}_0^{-1}\mathcal{A}_0 \|_1 \eta^{(\nu)}_{0,N+1} \left(\| \pi^{>N}D\mathcal{K}(\bU) \pi^{\leq N}  \mathcal{P}_0\|_1 +\| \mathcal{P}_0\|_1 Z_{2,> N}(r)r\right)\\
    & \quad\leq  \beta_{2,1}  |\lambda| + \beta_{1,1},
\end{align*}
where we used Proposition~\ref{prop : Ki is compact} to control $\Vert \cL^{-1}\pi^{>N}\Vert_1$ by $\eta^{(1)}_{0,N+1}$. Regarding the second term of \eqref{eq : step 3 proof Th}, we first estimate
\begin{align}\label{eq : proof th step 4}
    \|\mathcal{P}_0^{-1}\mathcal{A}_0  \pi^{\leq N}   D\mathcal{K}(\tilde{U})  \pi^{>N} C_\lambda  \mathcal{P}_0\|_1 &\leq \|\mathcal{P}_0^{-1}\mathcal{A}_0 \pi^{\leq N}     D\mathcal{K}(\bU)  \pi^{>N} C_\lambda  \mathcal{P}_0\|_1  \\&\quad + \|\mathcal{P}_0^{-1}\|_1 Z_{2,\leq N}(r)r \|\pi^{>N} C_\lambda  \mathcal{P}_0\|_1.\nonumber
\end{align}
Using once more the definition of $C_\lambda$ in~\eqref{eq:defClambda}, rewriting each occurrence of $D\cK(\tilde{U})$ as $D\cK(\bU)+(D\cK(\tilde{U})-D\cK(\bU))$, and using the triangle inequality, we then get
\begin{align}
\label{eq:estClambda}
    \|\pi^{>N} C_\lambda  \mathcal{P}_0\|_1 \leq \beta_{1,2} |\lambda| + \beta_{0,1},
\end{align}
and
\begin{align*}
    \|\mathcal{P}_0^{-1}\mathcal{A}_0 \pi^{\leq N}     D\mathcal{K}(\bU)  \pi^{>N} C_\lambda  \mathcal{P}_0\|_1 \leq \beta_{1,3}|\lambda| + \beta_{0,2}.
\end{align*}
Overall, we have obtained that
\begin{multline}\label{eq : final step 1}
    \|\mathcal{P}_0^{-1}\mathcal{A}_0\pi^{\leq N}(\lambda   \mathcal{L}^{-1}  -  D\mathcal{K}(\tilde{U})  )\pi^{>N} C_\lambda  \mathcal{P}_0\|_1  \leq \\ 
     |\lambda|(\beta_{2,1}  |\lambda| + \beta_{1,1}) + \beta_{1,3}|\lambda| + \beta_{0,2} + (\beta_{1,2} |\lambda| + \beta_{0,1})( \|\mathcal{P}_0^{-1}\|_1   Z_{2,\leq N}).
\end{multline}
Focusing now on the last term of~\eqref{eq : proof term to estimate}, we estimate
\begin{align}\label{eq : final step 2}
    &\|\mathcal{P}_0^{-1}\mathcal{A}_0\pi^{\leq N}(\lambda  \mathcal{L}^{-1}  -  D\mathcal{K}(\tilde{U})  )\pi^{>N}\|_1 \nonumber\\
    &\qquad\leq|\lambda| \|\mathcal{P}_0^{-1}\mathcal{A}_0 \pi^{\leq N} \mathcal{L}^{-1} \pi^{>N}\|_1 + \|\mathcal{P}_0^{-1}\mathcal{A}_0 \pi^{\leq N} D\mathcal{K}(\tilde U)  \pi^{>N}\|_1  \nonumber\\
    &\qquad\leq  |\lambda|  \eta^{(1)}_{0,N+1}\|\mathcal{P}_0^{-1}\mathcal{A}_0\|_1   + \|\mathcal{P}_0^{-1}\mathcal{A}_0 \pi^{\leq N} D\mathcal{K}(\bU)  \pi^{>N}\|_1 + \|\mathcal{P}_0^{-1}\|_1 Z_{2,\leq N}(r)r \nonumber\\
    & \qquad\leq \beta_{2,2}|\lambda| + \beta_{1,4}.
\end{align}
Combining \eqref{eq : final step 1} and \eqref{eq : final step 2} with~\eqref{eq:estClambda}, we finally get
\begin{align*}
    \|\mathcal{P}_0^{-1}\mathcal{A}_0G_\lambda \mathcal{P}_0\|_1 \leq \beta_2 |\lambda|^2 + \beta_1 |\lambda| + \beta_0.
\end{align*}
Coming back to~\eqref{eq:estminmulambda}, we have shown that
\begin{align}
\label{eq : last inequality}
     \min_{i \leq N} |1 - \mu_i \lambda| \leq |\lambda| \|R_0\|_1 + \beta_2 |\lambda|^2 + \beta_1 |\lambda| + \beta_0 + Z_{1,\cP_0} + Z_{2,\cP_0}(r)r \leq \delta.
\end{align}
Finally, letting $j \leq N$ be such that $\min_{i \leq N} |1 - \mu_i \lambda| =  |1 - \mu_j \lambda|$, we get
\begin{align*}
    |1 - \mu_j \lambda| = \left\vert 1-\frac{\lambda}{\lambda_j}\right\vert \leq \delta,
\end{align*}
and~\eqref{eq : enclosure eigenvalue lambda max} is established.

In other words, we have shown that, provided $\vert \lambda\vert \leq \lambda_{\max}$, the eigenvalue $\lambda$ belongs to one of the disks $\fB_j = \B_{\delta \vert\lambda_j\vert}(\lambda_j)$, for some $j\in\{0,\ldots,N\}$.
We now prove the final part of the theorem, allowing to count the eigenvalues of~\eqref{eq : generalized evp sequences} in a disjoint union of disks $\fB_j$. 
To this end, we work again with the spectrum of $D\cF(\tilde{U})^{-1}\cL^{-1}$ (see Lemma~\ref{lem : equivalent eig problem}), and we therefore introduce the sets
\begin{align*}
    \fB_j^\dag \bydef \{ z\in\C, z^{-1} \in \fB_j\},\qquad \text{for all }j\in\{0,\ldots,N\},
\end{align*}
and
\begin{align*}
    \fB_{{\max}}^\dag \bydef \{ z\in\C, z^{-1} \in \B_{\lambda_{\max}}(0)\} = \{ z \in\C, \vert z\vert \geq (\lambda_{\max})^{-1} \}. 
\end{align*}

What we have just shown is that any eigenvalue of $D\cF(\tilde{U})^{-1}\cL^{-1}$ which is in $\fB_{\max}^\dag$ must lie in $\fB_j^\dag$ for some $j\in\{0,\ldots,N\}$. Having assumed that $I \subset \{0, \dots, N\} $, with $|I| = k$, is such that the set $\cup_{j \in I} \fB_j$ is included in the interior of $\B_{\lambda_{\max}}(0)$ and is disjoint from $\cup_{j \in  \{0, \dots, N\} \setminus I} \fB_j $, i.e., that $\cup_{j \in I} \fB_j^\dag$ is included in the interior of $\fB_{\max}^\dag$ and is disjoint from $\cup_{j \in  \{0, \dots, N\} \setminus I} \fB_j^\dag $, we want to prove that $\cup_{j \in I} \fB_j^{\dag}$ contains exactly $k$ eigenvalues of $D\cF(\tilde{U})^{-1}\cL^{-1}$. What is clear is that $S_0$, viewed as an operator on $\ell^1_1$, has exactly $k$ eigenvalues in $\cup_{j \in I} \fB_j^\dag$, because each $\mu_j = \lambda_j^{-1}$ belongs to $\fB_j^\dag$ for $j\in\{0,\ldots,N\}$, and the only other eigenvalue of $S_0$ is $0$, which lies outside of $\fB_{\max}^\dag$. Our aim will therefore be to construct an homotopy between $S_0$ and $D\cF(\tilde{U})^{-1}\cL^{-1}$, i.e., a continuous map $H:[0,1]\to B(\ell^1_1)$ such that $H(0) = S_0$ and $H(1) = D\cF(\tilde{U})^{-1}\cL^{-1}$, for which we can prove that any element of the spectrum of $H(t)$ which is in $\fB_{\max}^\dag$ must lie in one of the $\fB_j^\dag$, for any $t\in[0,1]$. As in the proof of Lemma~\ref{lem : gershgorin matrix}, Theorem 3.16 from~\cite[Chapter IV]{kato2013perturbation} then shows that $D\cF(\tilde{U})^{-1}\cL^{-1}$ has exactly $k$ eigenvalues in $\cup_{j \in I} \fB_j^\dag$.
This is essentially what we accomplish below, except that we use several successive homotopies to connect $S_0$ to $D\cF(\tilde{U})^{-1}\cL^{-1}$.

We first deform $S_0$ into $\cM_0$, by considering the homotopy $H_1(t) \bydef (1-t)S_0 + t\cM_0 = S_0 + tR_0$ for all $t\in[0,1]$. We claim that, for all $t \in [0,1]$, if $\mu$ is an eigenvalue of $H_1(t)$ which is also in $\fB_{\max}^\dag$, then $\mu$ lies in some $\fB_j^\dag$. Indeed, first note that
\begin{align*}
    z\in \fB_j^\dag \quad\Longleftrightarrow\quad \left\vert \frac{1}{z} - \lambda_j \right\vert \leq \delta \vert\lambda_j\vert \quad\Longleftrightarrow\quad  \left\vert \frac{\mu_j}{z} - 1 \right\vert \leq \delta.
\end{align*}
Then, since $\mu\in\fB_{\max}^\dag$, $\mu\neq 0$, and therefore we can select an eigenvector $U\in\pi^{\leq N}\ell^1_1$ associated to the eigenvalue $\mu$ of $H_1(t)$. We have
\begin{align*}
    (S_0-\mu\cI)U = -t R_0U,
\end{align*}
therefore, by the same argument used to obtain~\eqref{eq:estminmulambda}, there exists $j\in\{0,\ldots,N\}$ such that
\begin{align*}
        |\mu_j - \mu| \leq |t| \|R_0\|_1 \leq \|R_0\|_1,
    \end{align*}
which yields
\begin{align*}
    \left\vert \frac{\mu_j}{\mu} - 1 \right\vert \leq \|R_0\|_1 \lambda_{\max} \leq \delta.
\end{align*}
We have indeed proven that $\mu \in\fB_j^\dag$, thus $\cM_0$ also has exactly $k$ eigenvalues in $\cup_{j \in I} \fB_j^\dag$.


Then, recall  that $\mathcal{M}_0$ has the same spectrum as $\mathcal{A}_0 \pi^{\leq N}\mathcal{L}^{-1}\pi^{\leq N} = \mathcal{P}_0 \mathcal{M}_0 \mathcal{P}_0^{-1}$. We now deform $\mathcal{A}_0 \pi^{\leq N}\mathcal{L}^{-1}\pi^{\leq N}$ into $\cA\cL^{-1}$, using a second homotopy $H_2(t) \bydef (1-t)\mathcal{A}_0 \pi^{\leq N}\mathcal{L}^{-1}\pi^{\leq N}  + t \mathcal{A} \mathcal{L}^{-1}$, for all $t\in[0,1]$. Since $\mathcal{A}$ is bounded on $\ell^1_1$ and $\cL^{-1}$ is compact on $\ell^1_1$ by Proposition \ref{prop : Ki is compact}, $H_2(t) : \ell^1_1\to \ell^1_1$ is compact for all $t \in [0,1]$. This implies that the spectrum of $H_2(t)$ is composed of eigenvalues only. Let $t \in [0,1]$ and let $(\mu, U) \in \fB_{\max}^\dag \times \ell^1_1$ be an eigenpair of $H_2(t)$, which means that
\begin{align*}
    (1-t)\mathcal{A}_0\pi^{\leq N} \mathcal{L}^{-1}\pi^{\leq N}U  + t \mathcal{A} \mathcal{L}^{-1} U = \mu U.
\end{align*}
Multiplying the above by $\pi^{>N}$ and using the structure of $\cA$ (see~\eqref{eq :decomposition of A}), we are left with
\begin{align*}
    t \pi^{>N} \mathcal{L}^{-1} U = \mu \pi^{>N} U.
\end{align*}
Recalling that  $\mu\in\fB_{\max}^\dag$, we introduce $\lambda=\mu^{-1}$ which satisfies $\vert\lambda\vert \leq \lambda_{\max}$, an get that
\begin{align*}
   \pi^{>N} U - t \lambda \pi^{>N} \mathcal{L}^{-1} \pi^{>N} U  = t \lambda \pi^{>N} \mathcal{L}^{-1} \pi^{\leq N} U.
\end{align*}
Following a similar analysis as presented above for the proof of \eqref{eq : finite of eigen}, we get that 
\begin{align*}
   \pi^{>N} U = \tilde B_{\lambda}^{-1} \tilde C_\lambda \pi^{\leq N} U,
\end{align*}
where this time
\begin{align*}
    \tilde B_{\lambda} = \pi^{>N}  - t \lambda \pi^{>N} \mathcal{L}^{-1} \pi^{>N} \qquad\text{and} \qquad \tilde C_{\lambda} = t \lambda \pi^{>N} \mathcal{L}^{-1} \pi^{\leq N}.
\end{align*}
In particular, note that $\tilde{B}_{\lambda} : \pi^{>N} \ell^1_1\to \pi^{>N} \ell^1_1$ is invertible since 
\begin{align*}
    \|\pi^{>N} - \tilde{B}_\lambda\|_1  \leq \epsilon,
\end{align*}
where $\epsilon$ is given by \eqref{eq : epsilon second approach}, and we can reproduce the remainder of the arguments leading to~\eqref{eq : last inequality} to prove that there exists $j\in\{0,\ldots,N\}$ such that $\lambda\in\fB_j$. This shows that $\mu\in\fB_j^\dag$, hence $\cA\cL^{-1}$ also has exactly $k$ eigenvalues in $\cup_{j \in I} \fB_j^\dag$. 

Finally, we deform $\cA\cL^{-1}$ into $D\cF(\tilde U)\cL^{-1}$, using a well-chosen (nonlinear) homotopy that allows us to re-use the estimates derived above. We first note that $(1-t)\cA^{-1} + t D\cF(\tilde{U}) : \ell^1_1\to \ell^1_1$ is boundedly invertible for all $t \in [0,1]$, as
\begin{align*}
    \|\cI - \mathcal{A} ((1-t)\cA^{-1} + t D\cF(\tilde{U}))\|_1 = t\|\cI - \mathcal{A}D\cF(\tilde{U})\|_1 \leq t Z_1 + t Z_2(r)r < 1.
\end{align*}
This enables us to define the homotopy
\begin{align*}
    H_3(t) \bydef \left( (1-t)\cA^{-1} + t D\cF(\tilde{U})\right)^{-1}\mathcal{L}^{-1},\qquad \text{for all }t\in[0,1].
\end{align*}
Once more, $H_3(t)$ is compact on $\ell^1_1$ for all $t \in [0,1]$, hence the spectrum of $H_3(t)$ is composed of eigenvalues only. Let $t \in [0,1]$ and let $(\mu, U) \in \fB_{\max}^\dag \times \ell^1_1$ be an eigenpair of $H_3(t)$. Then
\begin{align*}
     \left((1-t)\cA^{-1} + t D\cF(\tilde{U})\right) U = \lambda \mathcal{L}^{-1} U,
\end{align*}
where $\lambda = \frac{1}{\mu}$. Multiplying the above by $\mathcal{A}$, we get
\begin{align}\label{eq : H3 homotopy equation}
      U + t \left(\mathcal{A} D\cF(\tilde{U}) -\cI\right)U = \lambda \mathcal{A} \mathcal{L}^{-1} U.
\end{align}
Now, we closely follow the analysis used to prove \eqref{eq : last inequality}. We mimic the steps of the proof and use similar notations for convenience. Applying $\pi^{>N}$ to~\eqref{eq : H3 homotopy equation}, separating the contributions of $\pi^{\leq N} U$ and $\pi^{> N} U$, and using \eqref{eq :decomposition of A}, we get
\begin{align*}
   \pi^{>N} U+  t\pi^{>N} D\cK(\tilde{U})  \pi^{>N}U - \lambda \pi^{>N} \mathcal{L}^{-1} \pi^{>N} U  = \lambda \pi^{>N} \mathcal{L}^{-1} \pi^{\leq N} U - t\pi^{>N} D\cK(\tilde{U}) \pi^{\leq N} U,
\end{align*}
which we rewrite
\begin{align}\label{eq : tail of eigen 2}
    B_\lambda(t) \pi^{>N}U = C_\lambda(t) \,\pi^{\leq N}U,
\end{align}
with
\begin{equation}\label{eq : new B and C lambda}
    \begin{aligned}
    B_\lambda(t) &\bydef \pi^{>N}+  t\pi^{>N} D\mathcal{K}(\tilde{U})  \pi^{>N} - \lambda  \pi^{>N} \mathcal{L}^{-1} \pi^{>N} : \pi^{>N}\ell^1_1 \to \pi^{>N}\ell^1_1, \\
    C_\lambda(t) &\bydef \lambda  \pi^{>N} \mathcal{L}^{-1} \pi^{\leq N} - t\pi^{>N} D\mathcal{K}(\tilde{U})  \pi^{\leq N} : \pi^{\leq N}\ell^1_1 \to \pi^{>N}\ell^1_1.
\end{aligned}
\end{equation}
Since $t \in [0,1]$, we can use the above estimates and obtain $\| B_\lambda(t) - \pi^{>N} \|_1 \leq \epsilon$. In particular, this implies that $B_\lambda(t) : \pi^{>N} \ell^1_1 \to \pi^{>N} \ell^1_1$ has a bounded inverse.

Now, applying $\pi^{\leq N}$ to \eqref{eq : H3 homotopy equation}, we get
\begin{align*}
     \pi^{\leq N} U - \lambda \mathcal{A}_0 \mathcal{L}^{-1} \pi^{\leq N} U   =  - t \left(\mathcal{A}_0 D\cF(\tilde{U}) -\pi^{\leq N}\right)\pi^{\leq N}U + \mathcal{A}_0\left(  \lambda\mathcal{L}^{-1}  - t  D\cK(\tilde{U}) \right)\pi^{>N}U.
\end{align*}
Denoting $G_\lambda(t) = \left(  \lambda\mathcal{L}^{-1}  - t  D\cK(\tilde{U}) \right)\pi^{>N}B_{\lambda}(t)^{-1} C_\lambda(t)$, we get
\begin{align}\label{eq : finite part H3}
     \pi^{\leq N} U - \lambda \mathcal{A}_0 \mathcal{L}^{-1} \pi^{\leq N} U   = \left(\mathcal{A}_0 G_\lambda(t) + t \left( \pi^{\leq N} - \mathcal{A}_0 D\cF(\tilde{U})\right) \right)\pi^{\leq N}U,
\end{align}
exactly as in \eqref{eq:finiteeigen}. The rest of the estimates follows exactly as for the proof of \eqref{eq : last inequality}, using in addition that $t \in [0,1]$ (meaning that the estimates derived for the proof of \eqref{eq : last inequality} will uniformly bound the ones needed for \eqref{eq : finite part H3}).
More specifically, we obtain that  $\lambda \in \mathcal{B}_j$ for some $j \in \{0, \dots, N\}$, which finishes the proof.
\end{proof}
%

As was the case for Proposition~\ref{prop : finite number gershgorin} and Proposition~\ref{prop : tail gershgorin}, if we have a computable a priori estimate on unstable eigenvalues of~\eqref{eq : evp functions} we can try to combine it with Theorem~\ref{th : enclosure spectrum second homotopy} in order to count the number of unstable eigenvalues.

\begin{theorem}
\label{th : conclusion stability 2}
Repeat the assumptions of Theorem~\ref{th: radii polynomial}, with $\cA$ of the form~\eqref{eq :decomposition of A}, assume that~\eqref{eq : bounds general radii polynomial}--\eqref{condition radii polynomial} hold for $\nu=1$, and let $\tilde{U}$ be the unique zero of $\cF$ in $\B_r(\bU)$. Assume that there exists $\lambda_{\max} >0$ 
    such that any eigenvalue $\lambda$ in~\eqref{eq : evp functions} (with $\tilde{v} = \cG_0(\cL^{-1}\tilde{U})$) having nonnegative real part satisfies $\vert \lambda\vert \leq \lambda_{\max}$, and that $\epsilon<1$, with $\epsilon$ as in~\eqref{eq : epsilon second approach}. Consider the disks $\fB_j$ introduced in Theorem~\ref{th : enclosure spectrum second homotopy} for $j\in\{0,\ldots,N\}$, and assume further that:
    \begin{itemize}
        \item There are exactly $n_u\in\N_0$ disks $\fB_j$ included in $\{z\in\C,\ \Re(z) >0\}\bigcap \B_{\lambda_{\max}}(0)$.
        \item All the other disks $\fB_j$ are contained in the stable half space $\{z\in\C,\ \Re(z) <0\}$.
    \end{itemize}
Then, the linearization of~\eqref{eq : parabolic PDE original} at the steady state $\tilde{v}$ has exactly $n_u$ unstable eigenvalues.
\end{theorem}
\begin{remark}
    In order for some of the disks $\fB_j$ to be disjoints from others, we need to have $\delta<1$. This condition might fail if our a priori bound $\lambda_{\max}$ on unstable eigenvalues is too rough, as $\delta$ grows quadratically with $\lambda_{\max}$. However, Theorem~\ref{th : enclosure spectrum second homotopy} may also be used iteratively to improve upon that bound $\lambda_{\max}$. That is, let $\lambda_{\max} >0$ be
    such that any eigenvalue $\lambda$ in~\eqref{eq : evp functions} (with $\tilde{v} = \cG_0(\cL^{-1}\tilde{U})$) having nonnegative real part satisfies $\vert \lambda\vert \leq \lambda_{\max}$, and assume that $\epsilon<1$, with $\epsilon$ as in~\eqref{eq : epsilon second approach}. Then, provided the disks $\fB_j$ are all included in $\B_{\lambda_{\max}}(0)$ for $j\in\{0,\ldots,N\}$, Theorem~\ref{th : enclosure spectrum second homotopy} proves that any eigenvalue of $\lambda$ in~\eqref{eq : evp functions} having nonnegative real part lies $\bigcup_{j=0}^N \fB_j$. That is, letting 
    \begin{equation*}
        \lambda_{\max}^{\mathrm{new}} \bydef \inf \left\{ \rho >0,\ \bigcup_{j=0}^N \fB_j \subset \B_{\rho}(0) \right\},  
    \end{equation*}
    we get a sharper a priori estimate on unstable eigenvalues, that can be used instead of $\lambda_{\max}$.
\end{remark}

\subsection{Summary}\label{ssec : summary stability}

We introduced two computer-assisted approaches that allow the study of linear stability (and more precisely, to count the number of unstable eigenvalues) of steady states of equations of the form~\eqref{eq : parabolic PDE original}, summarized in Theorem~\ref{th : conclusion stability} and Theorem~\ref{th : conclusion stability 2}. Theorem~\ref{th : conclusion stability} relies on Gershgorin disks and typically allows for very sharp enclosures of the first eigenvalues $\{\lambda_i\}_{i=0}^N$. However, because Theorem~\ref{th : conclusion stability} mostly works with the compact operator $D\cF(\tilde{U})^{-1}\cL^{-1}$, whose eigenvalues accumulate at $0$, it requires a relatively fine control on the remaining eigenvalues $\{\lambda_i\}_{i>N}$ in order to be able to check that they are indeed stable, which can sometimes be challenging to obtain in practice. On the other hand, while Theorem~\ref{th : conclusion stability 2} gives comparatively rougher bounds on the first eigenvalues, in practice it can  make it easier to rigorously prove that the remaining eigenvalues are all stable. These differences will be illustrated in Section~\ref{sec : applications}.

Finally, we note that once the number of unstable eigenvalue has been determined, with either Theorem~\ref{th : conclusion stability} or Theorem~\ref{th : conclusion stability 2}, one can always get very tight enclosures  of these eigenvalues (if need be) by using a Newton-Kantorovich argument directly on the eigenvalue problem (cf. \cite{plum_orr_summerfeld} for instance).  This is the best way of getting a tight enclosure of any single eigenvalue, but it cannot tell us anything about the other eigenvalues, which is why we need Theorem~\ref{th : conclusion stability} or Theorem~\ref{th : conclusion stability 2} to determine how many unstable eigenvalues there are.

\section{Applications}
\label{sec : applications}

We now consider two concrete problems of the form~\eqref{eq : parabolic PDE original} on which we apply the methodology introduced in this paper. That is, after having found numerically what we hope to be a reasonably accurate approximation of a steady state, we first prove the existence of an actual steady state nearby using Theorem~\ref{th: radii polynomial}, and then rigorously determine its spectral stability using Theorem~\ref{th : conclusion stability} or Theorem~\ref{th : conclusion stability 2}.

\subsection{A simple parabolic PDE}
\label{ssec : toy problem}

We start with a very simple reaction diffusion equation
\begin{align}\label{eq : simple parabolic pb}
\begin{cases}
    \partial_t v  = \partial_x^2 v + \alpha (v^2 +1 ),\quad x\in(-1,1), \\
    v(\pm 1) = 0,
\end{cases}
\end{align}
where $\alpha>0$ is a parameter. 

This problem could obviously be studied without computer-assisted proofs, and in particular the stationary problem can already be well understood using ODE techniques, shooting methods, convexity, and the existence of a first integral. Nonetheless, we use this as a first example for the method introduced in this paper, as it will allow us to detail all the steps without too many technicalities. Here is a typical statement produced by our approach.
\begin{theorem}
    \label{th : toy problem1}
Let $\bv=\bv(x)$ be the function represented in Figure~\ref{fig : toy model}, and whose precise description in terms of Chebyshev coefficients can be downloaded at~\cite{julia_cadiot}\footnote{More precisely, $\bv\bydef \cG_0\cL^{-1}\bU$, where $\bU$ is a polynomial whose coefficients (in the Chebyshev basis) are given at~\cite{julia_cadiot}.}. There exists a smooth steady state $\tilde{v}=\tilde{v}(x)$ of~\eqref{eq : simple parabolic pb} with $\alpha=1$ such that
\begin{equation*}
    \max\left\{ \| \tilde{v} - \bv \|_\infty, ~ \|\partial_x^2 \tilde{v} - \partial_x^2 \bv \|_\infty \right\}\leq 10^{-14}.
\end{equation*}
This steady state $\tilde{v}$ is asymptotically stable.
\end{theorem}
\begin{figure}[H]
\centering
 \begin{minipage}[H]{0.9\linewidth}
\centering\epsfig{figure=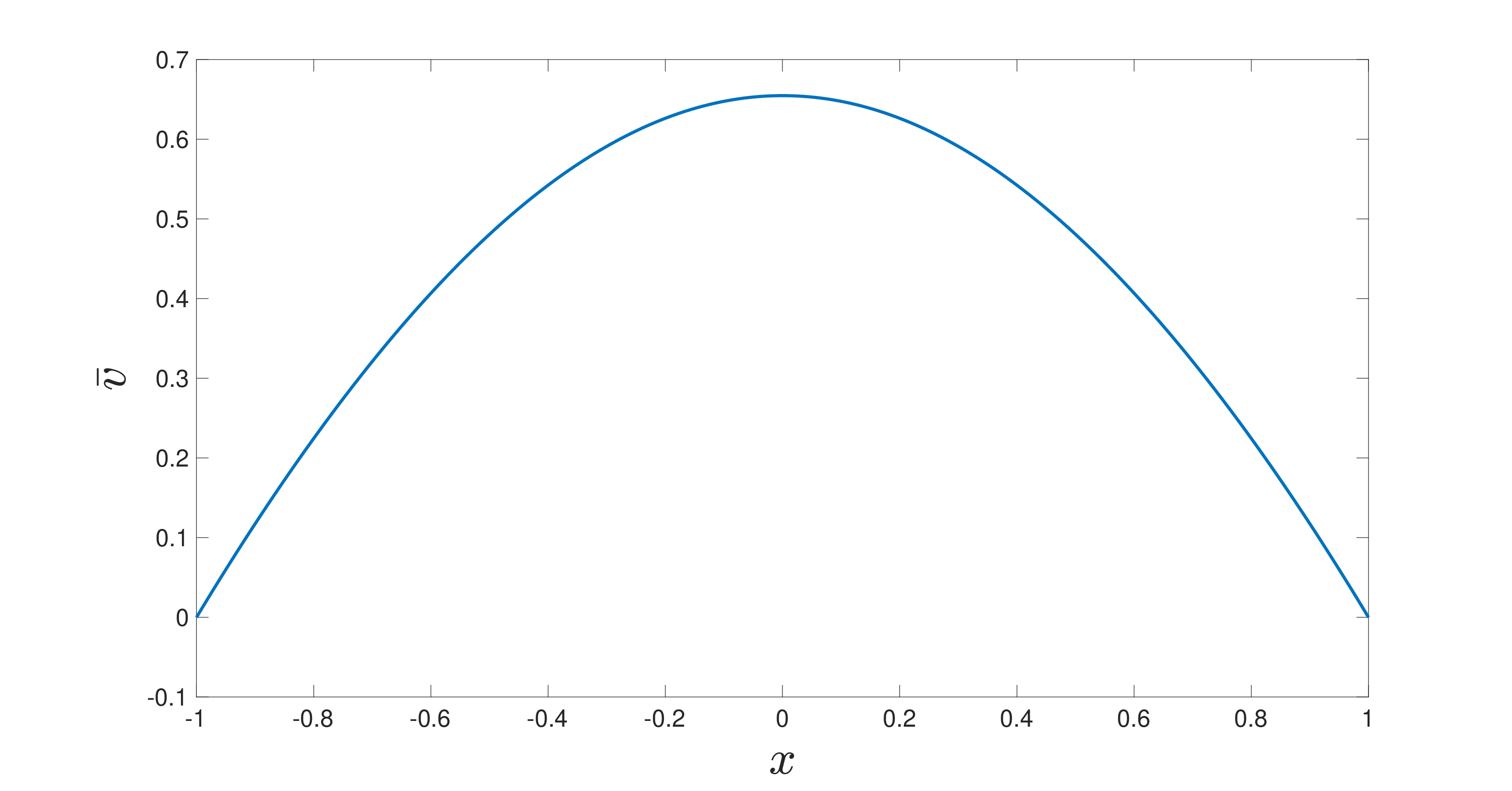,width=
  \linewidth}
  \caption{Numerical approximation $\bar{v}$ of the steady state $\tilde{v}$ of~\eqref{eq : simple parabolic pb} obtained in Theorem~\ref{th : toy problem1}.}
  \label{fig : toy model}
 \end{minipage} 
 \end{figure}

Despite the simplicity of this example, we emphasize that the existing computer-assisted proof literature based on Fourier series could not be efficiently used here, because the homogeneous Dirichlet boundary conditions are incompatible with the reaction term, which does not have odd symmetry. Finite element-based computer-assisted proofs~\cite{NakPluWat19} could obviously be used to prove the existence of a steady state near $\bv$, but with way poorer error bounds compared to what we obtain using our Chebyshev-based approach. In this specific example, a computer-assisted proof based on a Legendre expansion as in~\cite{NakKin09} would probably yield error bounds similar to those obtained in Theorem~\ref{th : toy problem1}. However, the work~\cite{NakKin09} only covers the Laplacian (i.e., $m=1$) with Dirichlet boundary conditions, does not deal with derivatives in the nonlinear term, and more crucially only focuses on obtaining a steady state, but not on studying its stability, which is the main novelty of our work. Once a steady state has been obtained for this example, one could also use other techniques such as the homotopy method~\cite[Chapter 10]{NakPluWat19} in order to rigorously count its number of unstable eigenvalues, but only because the linearized operator is self-adjoint, which is not the case in general for equations of the form~\eqref{eq : parabolic PDE original}.

In the remainder of Section~\ref{ssec : toy problem}, we explain how the setup and estimates introduced in this paper can be used to prove Theorem~\ref{th : toy problem1}. We first obtain the existence of $\tilde{v}$ in Theorem~\ref{th : existence simple PDE}, and then its stability in Theorem~\ref{th : stability simple problem}. Taken together, these two results immediately prove Theorem~\ref{th : toy problem1}.

\subsubsection{Construction of the operators of interest}

We first focus our attention on the existence proof of stationary solutions to \eqref{eq : simple parabolic pb}. Equivalently, we look for $v : [-1,1] \to \R$ such that $v$ solves
\begin{align}\label{eq : simple PDE stationary}
\begin{cases}
      \partial_x^2 v + \alpha (v^2 +1 ) = 0, \quad x \in (-1,1),\\
    v(\pm 1) = 0.
\end{cases}
\end{align}
Similarly as what was achieved in Section \ref{ssec : linear problem}, we first focus on the linear problem, determined by $\mathcal{L}^{-1}$.
In the case of \eqref{eq : simple PDE stationary}, $\mathcal{L}^{-1}$  corresponds to the inverse of the Dirichlet Laplacian, that is, to the operator $\mathcal{L}_0^{-1}$, explicitly given in Appendix \ref{App : explicit inverses} in \eqref{eq : inverse dirichlet}. For consistency with Section \ref{ssec : linear problem}, we keep the notation $\cL^{-1}$ here. 

Then, following Section \ref{ssec : zero finding}, we consider the following zero finding problem
\begin{align}\label{eq : F toy problem}
    \cF(U) = U + \mathcal{K}(U) + \alpha E^{(1)}_0 = U + \alpha (\mathcal{L}^{-1}U)*(\mathcal{L}^{-1}U) + \alpha E^{(1)}_0.
\end{align}
As described in Section \ref{ssec : NK}, we fix a numerical truncation size $N$ (in this case $N = 30$) and numerically construct an approximate solution $\bar{U}$ such that $\bar{U} = \pi^{\leq N} \bar{U}$. In practice, the approximate solution $\bU$ for $\alpha = 1$ used in Theorem~\ref{th : toy problem1} is obtained thanks to a numerical continuation method starting from $\alpha = 0$, where $0$ is a trivial solution.

Given $\bar{U}$, we then consider the Fr\'echet derivative 
\begin{align*}
    D\cF(\bar{U}) = I + \bar{\cV}_0 \,\mathcal{L}^{-1} \quad \text{ where } \quad \bar{\cV}_0 \bydef 2\alpha \mathcal{L}^{-1}\bar{U},
\end{align*}
and, as in Section~\ref{ssec : NK}, $\bar{\cV}_0$ must be interpreted as the multiplication operator $U\mapsto \bar{\cV}_0\ast U$.
We then construct an approximate inverse $\cA$ for $D\cF(\bU)$, as described in Section \ref{ssec : NK}. Having constructed $\bar{U}$ and $\cA$, we are ready to prove the existence of a zero of $\cF$ in a neighborhood of $\bar{U}$, using Theorem \ref{th: radii polynomial}.
\subsubsection{Computer-assisted proof of existence}
Using the analysis developed in Section \ref{sec:NK}, we already have explicit formulas for the bounds $Y$ and $Z_1$ of Theorem \ref{th: radii polynomial}. Specifically, the required estimates for applying the formula of Lemma~\ref{lem : Z1 bound general} are provided in Lemma~\ref{lem : properties dirichlet} (see Remark \ref{rem : dirichlet and eta}). 

In order to apply Theorem \ref{th: radii polynomial}, it remains to obtain explicit estimates for the bound  $Z_2$, which is achieved in the next result. 
\begin{lemma}\label{lem : Z2 for toy model}
    Let $\nu\geq 1$ and $Z_2$ be a constant satisfying 
    \begin{align}
        Z_2 \geq \frac{(1+\nu^2)^2|\alpha|}{8}\|\cA\|_\nu,
    \end{align}
    then we have that $\|\cA(D\cF(\bar{U}+h)-  D\cF(\bU))\|_\nu \leq Z_2 r$ for all $h \in \B_r(0)\subset \ell^1_\nu$.
\end{lemma}
\begin{proof}
    Let $h \in \B_r(0)\subset \ell^1_\nu$. Then, for all $W \in \ell^1_\nu$, we have that 
    \begin{align*}
        \|\cA(D\cF(\bar{U}+h)-  D\cF(\bU))W\|_\nu &= \|\cA(D\cK(\bar{U}+h)-  D\cK(\bU))W\|_\nu \\
        &= 2|\alpha|\|\cA(\cL^{-1}(\bar{U}+h)-  \cL^{-1}\bU)*\cL^{-1}W\|_\nu\\
        & \leq 2|\alpha|\|\cA\|_\nu\|\cL^{-1}h\|_\nu\|\cL^{-1}W\|_\nu.
    \end{align*}
   We conclude the proof using Lemma \ref{lem : properties dirichlet}.
\end{proof}

Moreover, since $\cA$ is of the form~\eqref{eq :decomposition of A}, the norm of $\cA$ reduces to a finite computation:
\begin{equation*}
    \Vert A\Vert_\nu = \max\left(\max_{0\leq n\leq N} \frac{1}{\xi_n \nu^n}\sum_{k=0}^N |A_{k,n}| \xi_k \nu_k,\, 1\right),
\end{equation*}
hence the bound $Z_2$ given in Lemma~\ref{lem : Z2 for toy model} is computable.
The estimates $Y$, $Z_1$ and $Z_2$
are implemented in \textit{Julia} using the packages \textit{RadiiPolynomial.jl} \cite{julia_olivier} and \textit{IntervalArithmetic.jl} \cite{julia_interval}. In particular, the rigorous computation of these bounds is achieved in \cite{julia_cadiot}, where we selected $\nu=1$.
\begin{remark}
    In principle, one could have used a larger value of $\nu$, and indeed the curious reader can check by using the code provided at \cite{julia_cadiot} that the proof of existence still succeeds for, e.g., $\nu = 3.5$ (increasing the number of coefficients to $N=100$). This then shows that the steady state $\tilde{v}$ is analytic at least on the Bernstein ellipse (see~\cite{Tre13}) of size $\nu$. 

    The reason we chose to focus on the proof with $\nu=1$ in the paper is that we will need some of the estimates obtained for that specific value of $\nu$ when studying the spectral stability of the steady state.
\end{remark}

Having computable $Y$, $Z_1$ and $Z_2$ bounds at our disposal, we can now find an explicit $r$ satisfying the assumption~\eqref{condition radii polynomial} of Theorem \ref{th: radii polynomial}, which then finishes the existence part of the proof of Theorem~\ref{th : toy problem1}.
\begin{theorem}\label{th : existence simple PDE}
    Let $\bU$ as defined in~\cite{julia_cadiot}, $r = 9.55 \times 10^{-15}$ and $\alpha = 1$. Then, there exists a unique $\tilde{U} \in \B_r(\bar{U}) \subset \ell^1_1$ such that $\cF(\tilde{U}) = 0$. In particular, $ \tilde{v} \bydef \mathcal{G}_0\left( \mathcal{L}^{-1} \tilde{U}\right)$ solves \eqref{eq : simple PDE stationary} and, letting $\bv \bydef \mathcal{G}_0\left( \mathcal{L}^{-1} \bU\right)$, we have $\max\{\|\tilde{v} - \bar{v}\|_{\infty}, ~ \|\partial_x^2\tilde{v} - \partial_x^2\bar{v}\|_{\infty} \} \leq {r}$.
\end{theorem}
\begin{proof}
    Recalling that we chose $N=30$, using Lemmas \ref{lem : Y bound general}, \ref{lem : Z1 bound general} and \ref{lem : Z2 for toy model}, and evaluating the bounds using the code available at~\cite{julia_cadiot}, we obtain 
    \begin{align*}
        Y = 9.27\times 10^{-15} , Z_1 = 4.65 \times 10^{-3} \text{ and } Z_2 = 1.12,
    \end{align*}
    and check that~\eqref{condition radii polynomial} holds. Theorem \ref{th: radii polynomial} then yields the existence of a unique zero $\tilde{U}$ of $\cF$ in $\B_r(\bar{U}) \subset \ell^1_1$.
    Then, using~\eqref{eq:C0VSell1} and the fact that $\Vert\cL^{-1}\Vert_1 \leq \frac{1}{2}$ (cf. Lemma \ref{lem : properties dirichlet}), we get
    \begin{align*}
        \|\bar{v}- \tilde{v}\|_{\infty} \leq \|\cL^{-1}(\bU - \tilde{U})\|_1 \leq \frac{1}{2} \|\bU - \tilde{U}\|_1 \leq \frac{r}{2}.
    \end{align*}
      Moreover, since $\partial_{x}^2\mathcal{G}_0(\cL^{-1}W) = \mathcal{G}_0(W)$ for all $W \in \ell^1_1$, we have that 
     \begin{align*}
         \|\partial_x^2\tilde{v} - \partial_x^2\bar{v}\|_{\infty} = \|\cG_0(\tilde{U} - \bU)\|_\infty \leq \Vert \tilde{U}-\bU \Vert_1 \leq r, 
     \end{align*}
     which concludes the proof.
\end{proof}

\subsubsection{Proof of stability}

In this section, we focus our attention on the stability of the stationary solution $\tilde{v}$ and prove the second part of Theorem~\ref{th : toy problem1}.

Let us denote $g(v) = \partial_x^2v + \alpha(v^2+1)$ the right-hand side of~\eqref{eq : simple parabolic pb}. Our goal is to prove, using the strategy presented in Section~\ref{ssec: Gersh}, that all the eigenvalues of $Dg(\tilde{v})$ (with the boundary conditions of~\eqref{eq : simple parabolic pb}) are stable. To that end, as explained in Section~\ref{ssec: Gersh}, we first need an a priori estimate regarding the location of the spectrum of $Dg(\tilde{v})$, henceforth denoted $\sigma\left(D g(\tilde{v})\right)$, which we derive below. Note that $D g(\tilde{v})$ has compact resolvent and is self-adjoint, hence we only have to deal with real eigenvalues.
\begin{lemma}\label{lem : enclosure simple parabolic pb}
   Let $\alpha = 1$, let $\tilde{v}$ be the steady state of~\eqref{eq : simple parabolic pb} obtained in Theorem \ref{th : toy problem1}, let $\bU$ and $r$ be as in Theorem~\ref{th : existence simple PDE}, and let $\lambda \in \sigma\left(Dg(\tilde{v})\right)$. Then, $\lambda$ is real and satisfies 
    \begin{align*}
        \lambda \leq \, \left(\|\bar{U}\|_1 +r\right) \leq 1.43.
    \end{align*}
\end{lemma}

\begin{proof}
    Let $v \in H^2(-1,1)$ be an eigenvector  satisfying the boundary conditions of~\eqref{eq : simple parabolic pb}, associated with $\lambda \in \sigma\left(D g(\tilde{v})\right)$, that is
\begin{align} \label{eq : simple evp}
\begin{cases}
      \partial_x^2 v + 2\tilde{v}v = \lambda v, \quad x \in (-1,1),\\
    v(\pm 1) = 0.
\end{cases}
\end{align}
    Then, multiplying the above equation by $v$ and integrating by parts we get
    \begin{align*}
        \lambda \, \|v\|^2_{L^2(-1,1)} = -\|\partial_x v\|^2_{L^2(-1,1)} + 2  \int_{-1}^1 \tilde{v} \, v^2 \, dx \leq 2  \, \|\tilde{v}\|_{L^\infty(-1,1)} \, \|v\|^2_{L^2(-1,1)}.
    \end{align*}
    We obtain the desired estimate using Theorem \ref{th : existence simple PDE} and that $\|\tilde{v}\|_\infty \leq \|\mathcal{L}^{-1}\tilde{U}\|_1 \leq \frac{1}{2}\left(\|\bar{U}\|_1 + r\right)$ (where we used Lemma \ref{lem : properties dirichlet}). Finally, the value $1.43$ is obtained using rigorous computations in \cite{julia_cadiot}.
\end{proof}
The above result provides the necessary a priori information for applying Theorem \ref{th : conclusion stability}, namely an explicit constant $\lambda_{\max}=1.43$. This allows us to obtain a much finer control of the spectrum, and to prove the spectral stability of $\tilde{v}$.
\begin{theorem}
\label{th : stability simple problem}
    Let $\alpha = 1$ and $\tilde{v}$ be the steady state of~\eqref{eq : simple parabolic pb} obtained in Theorem \ref{th : toy problem1}. For all $\lambda \in \sigma\left(Dg(\tilde{v})\right)$, $\lambda \leq -1.33$.
    In particular, $\tilde{v}$ is spectrally stable, and therefore asymptotically stable.
\end{theorem}

\begin{proof}
We proceed as explained in Section~\ref{ssec: Gersh}, and consider the linear map $\cM$ defined in~\eqref{def : matrix diag M}, with the map $\cF$ introduced in~\eqref{eq : F toy problem} for obtaining the existence of the steady state. 
Using the estimates of Proposition \ref{prop : finite number gershgorin} with $N = 30$, we control a finite number of Gershgorin disks of $\cM$. The implementation of these estimates available at \cite{julia_cadiot} yields
   \begin{align*}
       \Re\left(\cup_{n \leq N} \B_{r_n}(\cM_{n,n})\right) \subset [-0.75, 1.08\times 10^{-4}],
   \end{align*}
   where we only care of the real part of the disks because we already know that all the eigenvalues are real.
    Moreover, the implementation of the estimates of Proposition \ref{prop : tail gershgorin} yields the existence of a constant $\bar{r}^{\infty}_{\max}$ such that
   \begin{align*}
       \bar{r}^\infty_n + \eps^\infty_n \leq \bar{r}^{\infty}_{\max} \leq  4.6\times 10^{-3},
   \end{align*}
   for all $n > N$.  Taken together, these results show that the spectrum of $\cM$ is included in $[-0.75,4.6\times 10^{-3}]$.

   However, thanks to Lemma~\ref{lem : enclosure simple parabolic pb} we already know that there are no eigenvalues of $\cM$ in the interval $[0, \frac{1}{1.43}]$. Indeed, this Lemma together with Lemma~\ref{lem : equivalent eig problem} show that any nonnegative eigenvalue of $\cM$ is bigger than $1/1.43$. Hence, as $4.6\times 10^{-3}< \frac{1}{1.43}$, all the eigenvalues of $\cM$ must be contained in $[-0.75,0)$. Coming back to the original eigenvalue problem~\eqref{eq : simple evp} (using once more Lemma~\ref{lem : equivalent eig problem}), we have that all the eigenvalues $\lambda$ of $Dg(\tilde{v})$ are included in $(-\infty, \blue{-}\frac{1}{0.75}]$.
   Therefore $\tilde{v}$ is indeed spectrally stable, and hence also asymptotically stable (see, e.g., \cite[Section 5.1]{henry_semilinear}).
\end{proof}
\begin{remark}
    The second part of the above proof is nothing but a specific and particularly simple instance of Theorem~\ref{th : conclusion stability}.
    In that case, the Gershgorin disks for $n\leq N$ are not all included in the stable subspace $\{\Re(z) < 0\}$ because we used a relatively small $N$. However, as soon as $N$ is taken a bit larger, we are back in the setting of Example~\ref{ex : ccl stab}.
\end{remark}

\subsection{Kuramoto--Sivashinsky PDE}
\label{ssec:KS}
We consider now the Kuramoto--Sivashinsky PDE with Neumann boundary conditions, i.e.,
\begin{align}\label{eq : KS original}
    \begin{cases}
        \partial_t v =  -\partial_x^4 v - \alpha\, \partial_x^2 v - \alpha \, v \, \partial_x v ~~ &(t,x) \in (0,T) \times (-1,1),\\
        \partial_x v(t,-1) = \partial_x v(t,1) = \partial_x^3 v(t,-1) = \partial_x^3 v(t,1) =0 ~~ &t \in (0,T),\\
        v(t,-x) = -v(t,x) ~~ &(t,x) \in (0,T)\times[-1,1],
    \end{cases}
\end{align}
where $\alpha\geq 0$ is a parameter.
This equation is known to generate nontrivial spatiotemporal dynamics, including chaos~\cite{WilZgl24}, and therefore often serves as a test case for numerical methods and for computer-assisted proofs. In particular, relatively early computer-assisted proofs for steady states of the Kuramoto--Sivashinsky can be found in~\cite{ZglMis01}. However both this work and subsequent computer-assisted proofs of steady states for this equation~\cite{Zgl02bis,AriKoc10} use homogeneous Dirichlet or periodic boundary conditions, because these allow for a smooth sine series expansion (which is also compatible with the symmetries of the equation itself). When considering Neumann boundary conditions, these Fourier techniques can no longer be applied, but we showcase here that our more flexible setup based on Chebyshev expansions still works, both for obtaining steady states and then for studying their stability. 

We mention that the restriction to odd functions made here is often used when studying the Kuramoto--Sivanshisky equation, and in particular in all the aforementioned computer-assisted works, as it removes the zero eigenvalue that would appear because the equation preserves mass (i.e. $\int_{-1}^1 v(t,x) \mathrm{d}x$ is constant). In practice, this symmetry constraint simply means that we will work on a reduced set of indices (namely the odd natural numbers) for our series expansions. 

Here is a typical statement produced by our approach.
\begin{theorem}
    \label{th : KS}
Let $\bv_1$ and $\bv_2$ be the functions represented in Figure~\ref{fig:kuramoto}, and whose precise descriptions in terms of Chebyshev coefficients can be downloaded at~\cite{julia_cadiot}. There exist two smooth steady states $\tilde{v}_1$ and $\tilde{v}_2$ of~\eqref{eq : KS original} with $\alpha = 1$ and $\alpha = 100$ respectively such that
\begin{equation*}
    \max\left\{ \| \tilde{v}_1 - \bv_1 \|_\infty, ~ \|\partial_x^4 \tilde{v}_1 - \partial_x^4 \bv_1 \|_\infty \right\} \leq 1.28 \times 10^{-12},
    \end{equation*}
    and
    \begin{equation*}
 \max\left\{ \| \tilde{v}_2 - \bv_2 \|_\infty, ~ \|\partial_x^4 \tilde{v}_2 - \partial_x^4 \bv_2 \|_\infty \right\} \leq 3.92 \times 10^{-9}.
\end{equation*}
The steady state $\tilde{v}_1$ has Morse index $1$, i.e., the linearized operator around $\tilde{v}_1$ has exactly one unstable eigenvalue. The steady state $\tilde{v}_2$ has Morse index $2$.
\end{theorem}
\begin{figure}[H]
\centering
\begin{minipage}[t]{0.5\linewidth}
  \centering
  \includegraphics[width=\linewidth]{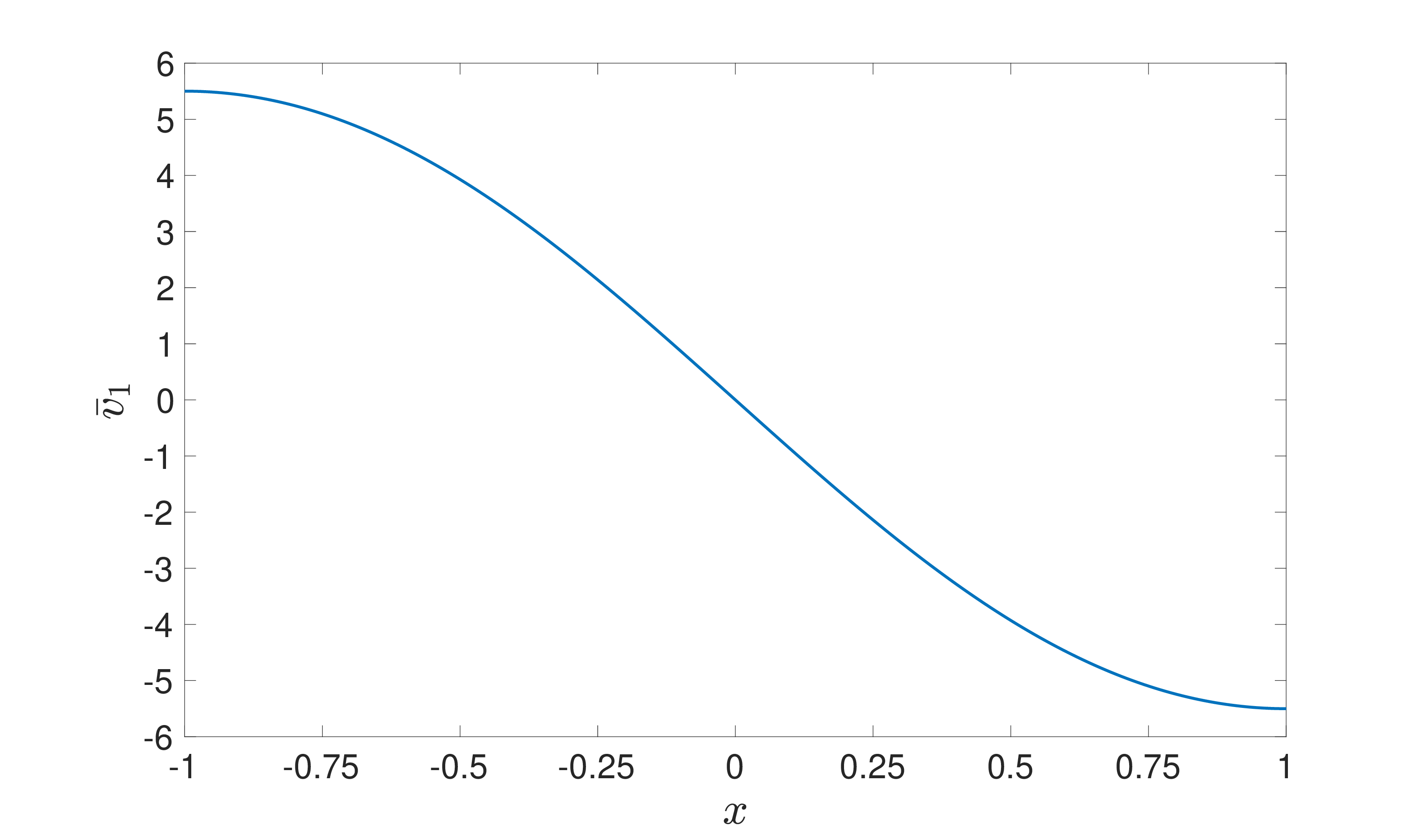}
\end{minipage}\hfill
\begin{minipage}[t]{0.5\linewidth}
  \centering
  \includegraphics[width=\linewidth]{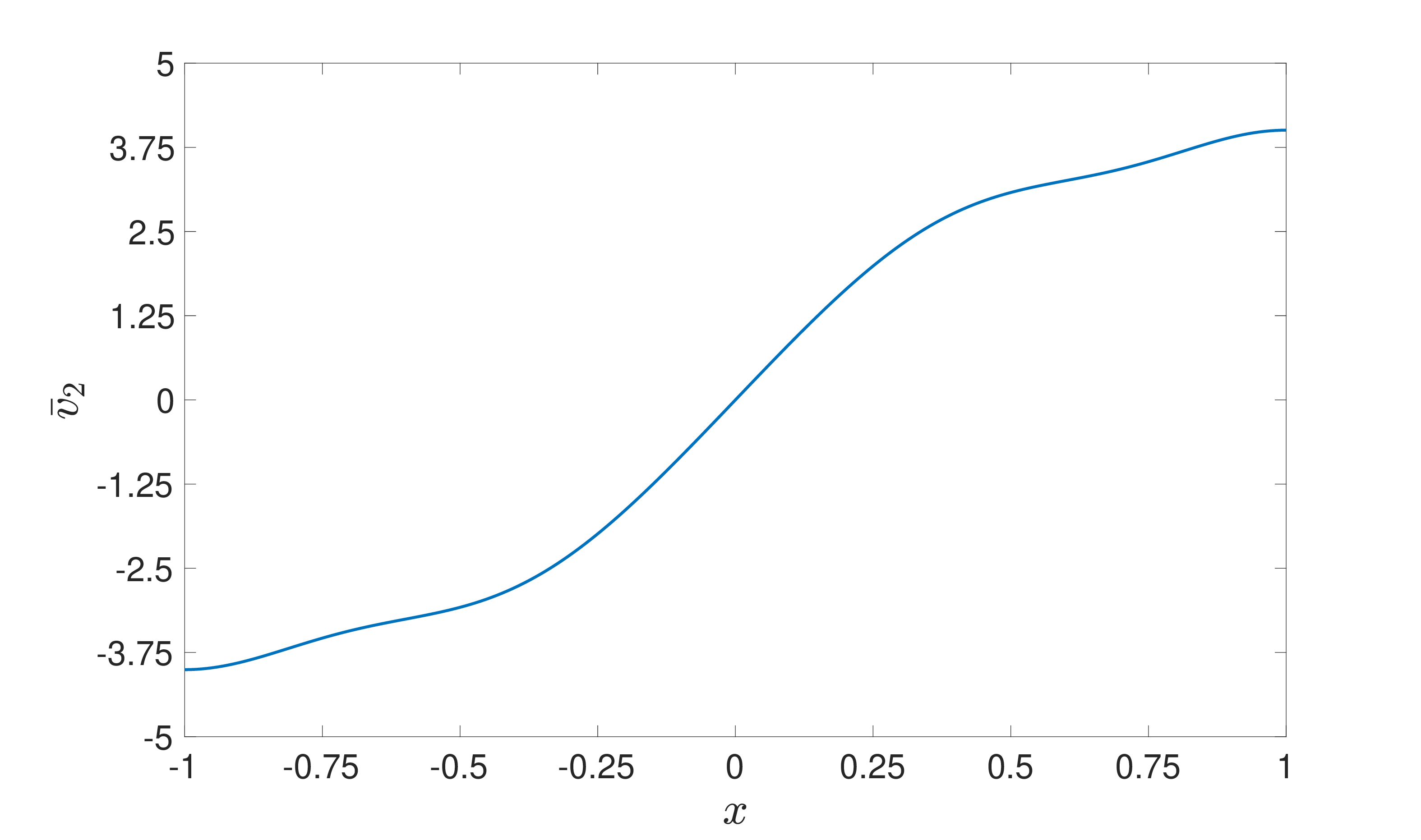}
\end{minipage}

\caption{Numerical approximations $\bar{v}_1$ (left) and $\bar{v}_2$ (right) of the steady states $\tilde{v}_1$ and $\tilde{v}_2$ of the Kuramoto--Sivashinsky equation obtained in Theorem~\ref{th : KS}.}
\label{fig:kuramoto}
\end{figure}

Note that stable nontrivial steady states also exist for the Kuramoto--Sivashinsky equation~\cite{Zgl02bis}, but we purposely selected unstable steady states for this example in order to illustrate that our method can also handle such cases.
\subsubsection{Additional notation and reformulation in sequence space}
The steady state problem writes
\begin{align}\label{eq : KS stationary}
    \begin{cases}
         -\partial_x^4 v - \alpha \partial_x^2 v - \alpha v \partial_x v = 0 ~~ &x \in  (-1,1),\\
        \partial_x v(1) = \partial_x^3 v(1) = 0 &\text{ and } v \text{ is odd}.
    \end{cases}
\end{align}
The odd symmetry allows us to work with sequences having trivial components. Indeed, let us define $\ell^1_{\nu,o}$ and $\ell^1_{\nu,e}$, respectively containing sequences with odd and even functions in their Gegenbauer expansion, as follows
\begin{equation}\label{def : odd and even restrictions}
\begin{aligned}
     \ell^1_{\nu,o} &\bydef \{ U = (U_n)_{n \in \mathbb{N}_0} \in \ell^1_\nu, ~ U_n = 0 \text{ for all even $n$}\} \\   \ell^1_{\nu,e} &\bydef \{ U = (U_n)_{n \in \mathbb{N}_0} \in \ell^1_\nu, ~ U_n = 0 \text{ for all odd $n$}\}.
\end{aligned}
\end{equation}
Note that in practice, we will work with reduced sets, that is, we numerically implement sequences indexed on odd indices to represent sequences in $\ell^1_{\nu,o}$, and respectively even indices for $\ell^1_{\nu,e}$. For our analysis, it is more convenient to consider \eqref{def : odd and even restrictions}, since we can readily use the work developed in the previous sections.

Now, we provide further details regarding the establishment of the zero finding problem in sequence space for finding solutions to \eqref{eq : KS stationary}. First, we construct the boundary conditions operator. More specifically, we define $\cB : \ell^1_{\nu,o} \to \ell^1_{\nu,o}$ as 
\begin{align}
    (\cB(U))_k = \begin{cases}
       \displaystyle 2\sum_{n\in \N} n^2 U_n &\text{ if } k=1,\\
     \displaystyle   \frac{32}{120}\sum_{n\in \N \setminus \{1\}}^\infty n(n-2)\ldots (n+2) U_n &\text{ if } k=3,\\
        0 &\text{ otherwise,}
    \end{cases}
\end{align}
which is nothing but the instance of the operator introduced in~\eqref{def : boundary condition} corresponding to the boundary conditions of~\eqref{eq : KS stationary}. Using this specific $\cB$ in the construction of Section~\ref{ssec : linear problem}, we obtain a sequence space representation $\cL^{-1}:\ell^1_{\nu,o}:\to\ell^1_{\nu,o}$ of $-(\partial_x^4)^{-1}$ restricted to the odd subspace. Then, the general zero finding problem introduced in~\eqref{eq : F} becomes
\begin{align*}
    \cF(U) &=  U - \alpha \cG_0^{-1}\partial_x^2 \cG_0\mathcal{L}^{-1} U -  \alpha (\cG_0^{-1}\partial_x \cG_0\mathcal{L}^{-1} U) \ast ( \mathcal{L}^{-1} U) \\
    &= U - \alpha \cK_2 U - \alpha (\cK_1 U) \ast (\cK_0 U).
\end{align*}
Note that $\cK_1$ restricts to an (unbounded) operator from $\ell^1_{\nu,o}$ to $\ell^1_{\nu,e}$, and that the discrete convolution between an element of $\ell^1_{\nu,e}$ and an element of $\ell^1_{\nu,o}$ yields and element of $\ell^1_{\nu,o}$ (the product of an even function and an odd function is an odd function), hence $\cF$ does preserve the odd symmetry (which is of course natural, since the orginal Kuramoto--Sivashinsky equation does so).


In what follows, we provide an explicit representation for  $\cK_0$, $\cK_1$ and $\cK_2$.

\subsubsection{Formulas for the operators $\cK_i$}

In order to study the operators $\cK_i$, it will be convenient to factorize $\cL^{-1}$. We therefore first consider $\cL_1^{-1}:\ell^1_{\nu,o}\to\ell^1_{\nu,o}$, which, proceeding as in Section~\ref{ssec : linear problem}, we define to be the sequence space representation of the solution operator of
\begin{align*}
         \begin{cases}
     \partial_x^{2} v = \psi \quad\quad \text{ on }(-1,1),\\
         \partial_x v(1) = 0 \quad\text{ and } v \text{ is odd},
    \end{cases}
\end{align*}
restricted to odd inputs $\psi$. Note that $\mathcal{L}^{-1}_1$ is explicitly given in \eqref{eq : inverse neumann}.
In fact, we have $-\cL^{-1} = \cL_1^{-2}$, simply because
\begin{align*}
         \begin{cases}
     \partial_x^{4} v = \psi \quad\quad\qquad\qquad \text{ on }(-1,1),\\
         \partial_x v(1) = \partial_x^3 v(1) =  0 \quad\text{ and } v \text{ is odd},
    \end{cases}
\end{align*}
is equivalent to
\begin{align*}
         \begin{cases}
     \partial_x^{2} v = w \quad\quad \text{ on }(-1,1),\\
         \partial_x v(1) = 0 \quad\text{ and } w \text{ is odd},
    \end{cases}
    \qquad 
    \text{and}
    \qquad
    \begin{cases}
     \partial_x^{2} w = \psi \quad\quad \text{ on }(-1,1),\\
         \partial_x w(1) = 0 \quad\text{ and } w \text{ is odd}.
    \end{cases}
\end{align*}
Consequently, the explicit expression for $\mathcal{L}_1^{-1}$ allows us to express  all the $\cK_i$:
\begin{equation*}
    \cK_0 = \cL^{-1} = -\cL_1^{-2},\qquad \cK_1 = -\cG_0^{-1}\partial_x \cG_0 \cL_1^{-2}, \quad \text{and}\quad \cK_2 = -\cL_1^{-1}.
\end{equation*}
It remains to obtain a more explicit formula for $\cK_1$, in the spirit of Proposition~\ref{prop : Ki is compact}.

Let $U\in\ell^1_{\nu,o}$, and $W = \cK_1 U \in \ell^1_{\nu,e}$. Proceeding as in the proof of Proposition~\ref{prop : Ki is compact}, we also introduce $V = \cL^{-1}U$, as well as $u = \cG_0(U)$, $v = \cG_0(V)$, and $w = \cG_0(W)$. We have $\partial_x^3 w = -u$, i.e., 
\begin{align*}
    \left(\cB_{\cK_1} - \Sigma^{3}\cD_3\right) W = \Sigma^3 \cC_{0,3}U,
\end{align*}
with a boundary operator $\cB_{\cK_1}$ to be determined. To that end, note that $w = \partial_x v$, where $v$ is odd and satisfies $\partial_x v(1) = \partial_x^3 v(1) = 0$. Hence, $w$ is even and satisfies $w(1) = \partial_x^2 w(1) = 0$. That is, we can take $\cB_{\cK_1} : \ell^1_{\nu,o} \to \ell^1_{\nu,e}$ given as 
\begin{align*}
    (\cB_{\cK_1}U)_k = \begin{cases}
      U_0 +  2\sum_{n \in \mathbb{N}_{0}} U_n &\text{ if } k=0,\\
        \frac{1}{3}\sum_{n \geq 2} n^2 (n-1)(n+1) U_n &\text{ if } k=2,\\
        0 &\text{ otherwise.}
    \end{cases}
\end{align*}
We then get the following fully computable expression
\begin{align}\label{eq : K1 for KS}
    \cK_1 = -(\cI- \cB_{\cK_1,3}^{\dagger} \cB_{\cK_1})\mathcal{D}_{3}^\dagger \Sigma^{3} \mathcal{C}_{0,3},
\end{align}
where $\cB_{\cK_1,3}^{\dagger} $ relies on the computation of the inverse of a $2$ by $2$ matrix, which is rigorously achieved in \cite{julia_cadiot} thanks to interval arithmetic computations.
In practice, recall that we consider the restriction  $\cK_1 : \ell^1_{\nu,o} \to \ell^1_{\nu,e}$, mapping odd to even functions. 

\subsubsection{Constructive existence proof of a steady-state}
Now that the zero finding map $\cF$ can be made explicit, we are able to implement its numerical representation. We will now follow the analysis derived in Section \ref{sec:NK} with $\nu=1$, that is we choose to work on the usual $\ell^1_1$ space. We fix $N = 200$ and use a Newton method in order to compute a numerical approximation $\bar{U} = \pi^{\leq N}\bU \in \ell^1_{\nu,o}$, for $\alpha = 1$ and $\alpha =100$.  In addition, following Section \ref{ssec : NK}, we construct an approximate inverse $\cA$ for 
\begin{align*}
    D\cF(\bU) = \cI - \alpha \cK_2 +  \cV_1 \cK_1 + \cV_0 \cK_0 \quad  \text{where} \quad  \cV_0 \bydef -\alpha \cK_1 \bU \text{ and } \cV_1 \bydef -\alpha \cL^{-1} \bU.
\end{align*}
In order to be able to apply Theorem \ref{th: radii polynomial}, it remains to compute an upper bound $Z_2$ and provide details for the computation of $Z_1$. The bound $Y$ and the main analysis for $Z_1$ is already provided in general in Section \ref{sec:NK}. In particular, the bound $Z_1$ is explicitly computable thanks to the properties on the operators $\mathcal{K}_i$ derived in Lemma \ref{lem : properties operators KS} (see also Remark \ref{rem : operators KS and eta}). In order to verify the hypotheses of Theorem \ref{th: radii polynomial}, it remains to compute an upper bound $Z_2$. This is achieved in the next lemma, which, when combined with Lemma \ref{lem : properties operators KS} and \eqref{eq : ell1opnormnu}, provides an explicit formula for $Z_2.$
\begin{lemma}
Let $Z_2$ be a bound satisfying
\begin{align*}
    Z_2 \geq 2 |\alpha| \|\cA\|_1 \|\mathcal{L}^{-1}\|_1\|\mathcal{K}_1\|_1, 
\end{align*}
then we have that $\|\cA(D\cF(\bU +h) - \cF(\bU))\|_{\ell^1} \leq Z_2$ for all $h \in \B_r(0) \subset \ell^1_1$.
\end{lemma}
\begin{proof}
     Let $h \in \B_r(0)\subset \ell^1_1$. Then, for all $W \in \ell^1_1$, we have that 
    \begin{align*}
        \|\cA(D\cF(\bar{U}+h)-  &D\cF(\bU))W\|_1 = \|\cA(D\cK(\bar{U}+h)-  D\cK(\bU))W\|_1 \\
        &= |\alpha|\|\cA(\cL^{-1}(\bar{U}+h)-  \cL^{-1}\bU)*\cK_1W + \cA(\cK_1(\bar{U}+h)-  \cK_1\bU)*\cL^{-1}W\|_1\\
        &\leq |\alpha|\|\cA\|_1 \left(\|\cL^{-1}h\|_1\|\cK_1W\|_1 + \|\cL^{-1}W\|_1\|\cK_1h\|_1\right). \qedhere
    \end{align*}
\end{proof}
Using the above, we constructively prove the existence of  steady-state solutions to \eqref{eq : KS original} in a neighborhood of our approximate solutions $\bar{v}_1  = \cG_0(\cL^{-1}\bU_1)$ and $\bar{v}_2  = \cG_0(\cL^{-1}\bU_2)$, depicted in Figure \ref{fig:kuramoto}.
\begin{theorem}\label{th : existence KS}
Let $j \in \{1,2\}$, $\bU_j$ as defined in~\cite{julia_cadiot}, $r_1 = 1.28 \times 10^{-12}$, $r_2 = 3.92 \times 10^{-9}$, $\alpha_1 = 1$ and $\alpha_2 = 100$. There exists a unique $\tilde{U}_j \in \B_{r_j}(\bar{U}_j) \subset \ell^1_1$ such that $\cF(\tilde{U}_j) = 0$. In particular, $ \tilde{v}_j \bydef \mathcal{G}_0\left( \mathcal{L}^{-1} \tilde{U}_j\right)$ solves \eqref{eq : KS stationary} with $\alpha = \alpha_j$ and, letting $\bv_j \bydef \mathcal{G}_0\left( \mathcal{L}^{-1} \bU_j\right)$, we have $\max\{\|\tilde{v}_j - \bar{v}_j\|_{\infty}, ~ \|\partial_x^4 \tilde{v}_j - \partial_x^4 \bar{v}_j\|_{\infty} \} \leq {r}_j$.
\end{theorem}

\begin{proof}
     Using Lemma \ref{lem : Y bound general}, \ref{lem : Z1 bound general} and \ref{lem : Z2 for toy model}, and evaluating the bounds using the code available at~\cite{julia_cadiot}, we obtain  
    \begin{align*}
        Y = 5.06 \times 10^{-13}, ~ Z_1 = 1.1\times 10^{-3} \text{ and } Z_2 = 2.55 \times 10^{-1},
    \end{align*}
    in the case $j=1.$ 
    The rest of the existence proof for $\tilde{v}_1$ follows similarly as the one of Theorem \ref{th : existence simple PDE}.  In particular, using the proof of Lemma \ref{lem : properties operators KS}, we have $\Vert \cL_1^{-1}\Vert  \leq 1$. The existence proof for $\tilde{v}_2$ follows similarly with bounds
    \begin{align*}
        Y = 3.03 \times 10^{-9}, ~ Z_1 = 2.28\times 10^{-1} \text{ and } Z_2 = 7.98 \times 10^{1}.
    \end{align*}
    This concludes the proof of Theorem~\ref{th : existence KS}.
\end{proof}

\subsubsection{Proof of instability and number of unstable directions}
We now study the stability of the steady states $\tilde{v}_1$ and $\tilde{v}_2$ of the Kuramoto--Sivashinsky equation~\eqref{eq : KS original} obtained in Theorem \ref{th : existence KS}, with the aim of proving the second part of Theorem~\ref{th : KS}. To that end, let $g$ be defined as 
\[
    g(v) \bydef -\partial_x^4 v - \alpha \partial_x^2 v - \alpha v\partial_x v.
\]
Our goal in this section is to control the spectrum of $Dg(\tilde{v})$ (with the boundary conditions of~\eqref{eq : KS original}), following the strategies described in Section~\ref{sec : stability}. For this purpose, we first state a result providing a rough a priori estimate on the location of spectrum of $Dg(\tilde{v})$.

\begin{lemma}\label{lem.lambdastarTM2}
Let $\alpha>0$. Let $\mu$ be a nonnegative real number and $\lambda \in \sigma\left(Dg(\tilde{v})\right)$ with $\lambda \in \C_{\geq -\mu} :=\{z\in \C \,:\,\Re(z)\geq -\mu \}$.
Then, for any $\kappa_1,\kappa_2 >0$ satisfying the constraint 
\begin{equation}
\label{eq:constraints_toto}
    1-\alpha\kappa_1-\frac{2 \alpha\kappa_2}{\pi}\Vert \tilde{v}\Vert_{L^\infty(-1,1)} \geq 0,
\end{equation}
$\lambda$ satisfies
\begin{align}
\label{eq:realpart_KS}
    -\mu \leq \Re(\lambda) \leq \alpha\left( \|\pa_x\tilde{v}\|_{L^\infty(-1,1)} + \frac{1}{2\pi\kappa_2}\Vert \tilde{v}\Vert_{L^\infty(-1,1)} + \frac{1}{4\kappa_1} \right).
\end{align}
Moreover, if the inequality~\eqref{eq:constraints_toto} is strict, then
\begin{align}
\label{eq:imagpart_KS}
|\Im(\lambda)| \leq \alpha\left(1  + \frac{2}{\pi}\|\tilde{v}\|_{L^\infty(-1,1)} \right)
    \left(\dfrac{\mu +\alpha\left( \|\pa_x\tilde{v}\|_{L^\infty(-1,1)} + \frac{1}{2\pi\kappa_2}\Vert \tilde{v}\Vert_{L^\infty(-1,1)} + \frac{1}{4\kappa_1} \right)}{1-\alpha\kappa_1-\frac{2 \alpha\kappa_2}{\pi}\Vert \tilde{v}\Vert_{L^\infty(-1,1)}}\right)^{1/2}.
\end{align}
\end{lemma}
We refer to Appendix~\ref{App : rough enclosure spectrum KS} for a proof of Lemma~\ref{lem.lambdastarTM2}.
\begin{remark}\label{rem : values kappas}
The estimate~\eqref{eq:realpart_KS} for the real part can be optimized by hand under the constraint~\eqref{eq:constraints_toto}, and the optimal choice
\begin{align*}
    \kappa_1 = \kappa_2 = \left(2\alpha \left(\frac{1}{2} + \frac{1}{\pi}\Vert \tilde{v}\Vert_{L^\infty(-1,1)}\right)\right)^{-1},
\end{align*}
yields
\begin{align*}
    \Re(\lambda) \leq \alpha\Vert\pa_x\tilde{v}\|_{L^\infty(-1,1)} + \alpha^2\left(\frac{1}{2} + \frac{1}{\pi}\Vert \tilde{v}\Vert_{L^\infty(-1,1)}\right).
\end{align*}
For the imaginary part, the minimum of~\eqref{eq:imagpart_KS} under~\eqref{eq:constraints_toto} is reached for
\begin{equation*}
    \kappa_1 = \kappa_2 = \left(\frac{\alpha\left(2+\frac{4}{\pi}\Vert\tilde{v}\Vert_{L^\infty(-1,1)}\right) + \sqrt{\alpha^2\left(2+\frac{4}{\pi}\Vert\tilde{v}\Vert_{L^\infty(-1,1)}\right)^2 + 16(\mu+\alpha\Vert\partial_x\tilde{v}\Vert_{L^\infty(-1,1)})}}{2}\right)^{-1},\end{equation*}
and we spare the reader the corresponding explicit estimate for $\Im(\lambda)$.
\end{remark}

%
%
For $\alpha=1$ and the steady state $\tilde{v}_1$, the approach presented in Section~\ref{ssec: Gersh} for studying stability works well, and yields the following result.
\begin{theorem}\label{th : stability v1}
Let $\tilde{v}_1$ be the steady state of~\eqref{eq : KS original} obtained in Theorem \ref{th : existence KS} for $\alpha = 1$. Moreover, let $\lambda_0 = 3.55557324817263$, $r_0 = 7.24 \times 10^{-12}$ and $\Lambda_s = -463<0$. Then, there exists $\tilde{\lambda}_0 \in \B_{r_0}(\lambda_0)\cap \sigma(Dg(\tilde{v}_1))$, while for all $\lambda \in \sigma\left(Dg(\tilde{u})\right) \setminus\{\tilde{\lambda}_0\}$ it holds $\Re(\lambda) \leq \Lambda_s$. In particular, $\tilde{v}_1$ is unstable and $Dg(\tilde{v}_1)$ possesses exactly one unstable direction.
\end{theorem}
\begin{proof}
      We implement in \cite{julia_cadiot} the estimations of Proposition \ref{prop : finite number gershgorin}, with $N=200$. This allows us to control a finite number of Gershgorin disks. In particular, using the notations of Lemma \ref{lem : gershgorin matrix}, we prove that there exists $\bar{\mu}_0 = 0.2812486004933086$, $\bar{r}_0 \leq 5.72\times 10^{-13}$ and $\bar{r}_{\max}^{\leq N} \leq 2.17 \times 10^{-3}$ such that 
   \begin{align*}
    \B_{r_0}(\cM_{0,0}) \subset \B_{\bar{r}_0}(\bar{\mu}_0)\quad\text{and} \quad  \bigcup_{n \in \{1, \dots, N\}} \B_{r_n}(\cM_{n,n})  \subset  \B_{\bar{r}_{\max}^{\leq N}}(0).
   \end{align*}
    Furthermore, we implement the estimates of Proposition \ref{prop : tail gershgorin}, which yield a positive constant $\bar{r}_{\max}^{\infty}$ such that
   \begin{align*}
       \bar{r}^\infty_n + \eps^\infty_n \leq \bar{r}_{\max}^{\infty}  \leq  6.53 \times 10^{-4} 
   \end{align*}
   for all $n > N$. Moreover, using Lemma \ref{lem.lambdastarTM2} combined with Remark \ref{rem : values kappas} with $\mu = 463.1$, we obtain that any eigenvalue of $Dg(\tilde{v})$ with  real part larger than $\mu$ has its modulus bounded from above by $\lambda_{\max} = 76.51$. Using Theorem \ref{th : conclusion stability} and the fact that $\bar{r}^\infty_n + \eps^\infty_n \leq \bar{r}_{\max}^{\infty} \leq  \lambda_{\max}^{-1}$, we obtain that all the disks $\B_{r_n}(\cM_{n,n})$ with $n >N$ enclose negative eigenvalues of $Dg(\bar{v}_1)$, of modulus larger that $\frac{1}{\bar{r}^\infty_n + \eps^\infty_n} \geq \frac{1}{\bar{r}_{\max}^{\infty}} \geq  1532 > \Lambda_s$. Now, since   $\B_{r_0}(\cM_{0,0}) \subset \B_{\bar{r}_0}(\bar{\mu}_0)$ is isolated from the rest of the Gershgorin disks, it implies from Lemma \ref{lem : gershgorin matrix} that $\mathcal{M}$ (cf. \eqref{def : matrix diag M}) possesses a unique eigenvalue in $\B_{\bar{r}_0}(\bar{\mu}_0)$. Equivalently, we obtain that $Dg(\tilde{v}_1)$ possesses a unique eigenvalue in $\B_{\bar{r}_0}(\bar{\mu}_0)^{\dagger} = \left\{z^{-1}, ~ z \in \B_{\bar{r}_0}(\bar{\mu}_0)\right\} \subset \B_{r_0}(\lambda_0)$, which provides the desired enclosure for $\tilde{\lambda}_0.$
   
   Then, since $\bar{r}_{\max}^{\leq N} \leq \lambda_{\max}^{-1}$, Lemma \ref{lem.lambdastarTM2} shows that $\bigcup_{n \in \{1, \dots, N\}} \B_{r_n}(\cM_{n,n})$ only contains eigenvalues of $\mathcal{M}$ with a negative real part. In particular, we have that $\left(\bar{r}_{\max}^{\leq N}\right)^{-1} \leq \Lambda_s \leq \mu$, which proves that  $\Re(\lambda) \leq -\Lambda_s$ for all $\lambda \in \sigma\left(Dg(\tilde{u})\right) \setminus\{\tilde{\lambda}_0\}$.
\end{proof}
In the previous result, we were able to obtain a tight enclosure on the unstable eigenvalue of the linearization, as well as a control on the amplitude of the spectral gap on the left part of the complex plane.  This result stems from a successful application of the method presented in Section~\ref{ssec: Gersh}.
For $\alpha = 100$ and the steady state $\tilde{v}_2$, applying the same approach turns out to be numerically demanding. Indeed, for $\alpha=1$, we found in the proof of Theorem \ref{th : stability v1} that $\lambda_{\max} = 76.51$. This allows to successfully apply  Theorem \ref{th : conclusion stability} with $N=200$. For the case $\alpha=100$ and the steady state $\tilde{v}_2$, it turns out that $\lambda_{\max}$ (computed thanks to Remark \ref{rem : values kappas}) is equal to $8.24 \times 10^4.$ With $N = 3000$, the application of Theorem \ref{th : conclusion stability} is not successful with our current estimates. This numerical inefficiency is resolved by using the approach presented in Section~\ref{ssec : second enclosure spectrum} instead. Note that we do not obtain as tight enclosures as in Theorem \ref{th : stability v1}, but this could easily be resolved using an extra Newton-Kantorovich approach, as explained in Section \ref{ssec : summary stability}.

We illustrate our previous discussion with the following result, providing an enclosure of the unstable eigenvalues for the steady state $\tilde{v}_2$ with a low numerical cost.
\begin{theorem}
Let $\tilde{v}_2$ be the steady state of~\eqref{eq : KS original} obtained in Theorem \ref{th : existence KS} for $\alpha = 100$.
    Moreover, let $\lambda_0 = 1173.40 + i 1426.49$, $r_0 = 3.65$ and $\Lambda_s = -375  < 0$. Then, there exist $\tilde{\lambda}_0 \in  \B_{r_0}(\lambda_0)\cap \sigma\left(Dg(\tilde{v}_2)\right)$. Moreover, its complex conjugate $\tilde{\lambda}_0^*$ also belongs to $\B_{r_0}(\lambda_0^*)\cap\sigma\left(Dg(\tilde{v}_2)\right)$, while for all $\lambda \in \sigma\left(Dg(\tilde{v}_2)\right) \setminus\{\tilde{\lambda}_0, \tilde{\lambda}_0^*\}$ we have $\Re(\lambda) \leq \Lambda_s$. In particular, $\tilde{v}_2$ is unstable and $Dg(\tilde{v}_2)$ possesses exactly two unstable directions.
\end{theorem}

\begin{proof}
First, using Lemma \ref{lem.lambdastarTM2}  and Remark \ref{rem : values kappas} with $\mu = 377$, we obtain that eigenvalues with positive real part must have an amplitude smaller than $\lambda_{\max} = 8.24 \times 10^4$. Then, we implement in \cite{julia_cadiot} the estimates of Theorem \ref{th : enclosure spectrum second homotopy} with $N=200$. This allows us to control all eigenvalues with modulus smaller than $\lambda_{\max}$. We obtain that exactly two disks are contained in the right-half part of the complex plane, specifically $\B_{r_0}(\lambda_0)$ and $\B_{r_0}(\lambda_0^*)$, and the rest of the disks are fully contained in the left-half part of the complex plane. In particular, we have that $\B_{r_0}(\lambda_0)$ and $\B_{r_0}(\lambda_0^*)$ are both contained in $\B_{\lambda_{\max}}(0)$. This allows us to  invoke Theorem \ref{th : conclusion stability 2}, and we obtain that there exists exactly two eigenvalues with positive real part, namely $\tilde{\lambda}_0$ and $\tilde{\lambda}_0^*$ belonging  to $\B_{r_0}(\lambda_0)$ and $\B_{r_0}(\lambda_0^*)$ respectively. Moreover, we obtain that an upper bound for the real part of the union of the rest of the disks is given by $-375.$
\end{proof}

\begin{appendix}

 \section{Properties for the inverse of the Dirichlet and Neumann Laplacians} \label{App : explicit inverses}


Following the general construction for $\cL^{-1}$ given in Section~\ref{ssec : linear problem}, we give here fully worked out formula for the inverse of the Dirichlet Laplacian, and then for the inverse of the Neumann Laplacian restricted to odd functions.

Let us first construct the inverse of the Dirichlet Laplacian. We start by introducing the (unbounded) linear operators $\tilde{\cL}_0$ on $\ell^1_\nu$ defined as 
\begin{align*}
    (\tilde{\cL}_0U)_n = \begin{cases}
         \cB_{1}(U,0) &\text{ if } n=0,\\
          \cB_{-1}(U,0)  &\text{ if } n=1,\\
          8U_2 &\text{ if } n=2,\\
          2nU_n &\text{ if } n \geq 3.
    \end{cases}
\end{align*}
Then, given $u = \mathcal{G}_0(U)$ and $f = \mathcal{G}_0(F)$ for some $U, F \in \ell^1_\nu$, the following BVP
\begin{align*}
\begin{cases}
    \partial_x^2 u = f,\\
    u(-1) = u(1) = 0,
\end{cases}
\end{align*}
is equivalent to 
\begin{align*}
    \tilde{\cL}_0 U = \Sigma^2\mathcal{C}_{0,2} F. 
\end{align*}
Consequently, we obtain that the inverse of the Dirichlet Laplacian writes
\begin{align*}
    \mathcal{L}_0^{-1} = \tilde{\cL}_0^{-1} \Sigma^2\mathcal{C}_{0,2}.
\end{align*}
In what follows, we provide an explicit expression for $\mathcal{L}_0^{-1}$.
Using~\eqref{def : change of basis before shift} and~\eqref{def : change of basis order k}, we have
\begin{align}\label{eq : change of basis in appendix}
    (\Sigma^2\mathcal{C}_{0,2}U)_n = \begin{cases}
    0 &\text{ if } n=0,1,\\
   U_0 - \frac{4U_2}{3} + \frac{U_4}{3} &\text{ if } n=2,\\
        \frac{U_{n-2}}{2(n-1)} - \frac{nU_n}{n^2-1} + \frac{U_{n+2}}{2(n+1)} &\text{ if } n \geq 3.
    \end{cases} 
\end{align}
Now, we compute the inverse of $\tilde{\cL}_0$. Let $U$ satisfying $\tilde{\cL}_0U = \Phi$ with $\Phi_0 = \Phi_1 = 0$, then we have that 
\begin{align*}
 U_0 + 2 \sum_{n=1}^\infty U_n = 0 \text{, } ~~
  U_0 + 2 \sum_{n=1}^\infty (-1)^n U_n = 0, ~~
  U_2 = \frac{1}{8} \Phi_2 ~ \text{ and } ~
  U_{n} = \frac{1}{2n}\Phi_{n} \text{ for all } n \geq 3.
\end{align*}
Solving the above, one finds that
\begin{align*}
    U_0 = - \frac{\Phi_2}{4}  - \sum_{n=2}^\infty \frac{\Phi_{2n}}{2n}, ~~
    U_1 = - \frac{1}{2}\sum_{n=2}^\infty\frac{ \Phi_{2n-1} }{2n-1}, ~~
    U_2 =  \frac{1}{8} \Phi_2 ~ \text{ and } ~
    U_{n} = \frac{1}{2n}\Phi_{n} \text{ for all } n \geq 3.
\end{align*}
Consequently, combining this with \eqref{eq : change of basis in appendix}, we obtain that 
\begin{align}\label{eq : inverse dirichlet}
    (\cL_0^{-1}U)_n = \begin{cases}
       \displaystyle -\frac{U_0}{4} + \frac{7U_2}{24}  -\frac{3}{2} \sum_{k=2}^\infty  \frac{U_{2k}}{(k-1)(k+1)(2k-1)(2k+1)} &\text{ if } n=0,\\
       \displaystyle  -\frac{U_1}{24} + \frac{U_3}{20} - \frac{3}{4} \sum_{k=3}^\infty  \frac{U_{2k-1}}{(k-1)k(2k-3)(2k+1)} &\text{ if } n=1,\\
        \displaystyle \frac{U_0}{8} - \frac{U_2}{6} + \frac{U_4}{24} &\text{ if } n=2,\\
      \displaystyle   \frac{U_{n-2}}{4n(n-1)} - \frac{U_n}{2(n^2-1)} + \frac{U_{n+2}}{4n(n+1)}  &\text{ if } n \geq 3.
    \end{cases}
\end{align}
Now that we obtained the explicit expression for the inverse of the Dirichlet Laplacian in sequence space, we derive some useful properties for our analysis.

\begin{lemma}\label{lem : properties dirichlet}
    Let $\nu \geq 1$. Then, the following properties hold: for all $k\geq 5$,
\begin{align*}
    \|\mathcal{L}_0^{-1} E^{(1)}_k\|_\nu &\leq  \frac{1}{4\nu(k-2)(k-1)} + \frac{1}{2(k^2-1)} + \frac{\nu}{4(k+1)(k+2)} + \frac{3\nu^{-k+1}}{(k^2-1)(k-2)(k+1)},
 \end{align*}
and
 \begin{align*}
    \|\mathcal{L}_0^{-1}\|_{\nu} &\leq \frac{1+\nu^2}{4}.
\end{align*}
\end{lemma}

\begin{proof}
The first inequality is a direct consequence of~\eqref{eq : inverse dirichlet}. In order to prove the second inequality, we recall that since we use a weighted $\ell^1_\nu$-norm, we have a very explicit formula for the operator norm, namely
\begin{equation}
\label{eq : ell1opnormnu}
    \Vert \cA\Vert_\nu = \sup_{n\in\N_0} \frac{1}{\xi_n \nu^n}\sum_{k \in \mathbb{N}_0} |A_{k,n}| \xi_k \nu_k = \sup_{n\in\N_0} \frac{1}{\xi_n}\nu^n \|A_{\cdot,n}\|_\nu,
\end{equation}
which is the generalization of~\eqref{eq : ell1opnorm} to the case $\nu\geq 1$. The proof is then a direct consequence of~\eqref{eq : inverse dirichlet} combined with~\eqref{eq : ell1opnormnu}.
\end{proof}

\begin{remark}\label{rem : dirichlet and eta}
   Notice that Lemma \ref{lem : properties dirichlet} provides a sharper estimate for $\|\mathcal{L}_0^{-1} E^{(1)}_k\|_\nu$ than 
   $\eta_{0,k}^{(\nu)}$, given in  Proposition \ref{prop : Ki is compact}. This is, of course, a consequence of the fact that we possess an explicit expression for $\mathcal{L}_0^{-1}$, when Proposition \ref{prop : Ki is compact} provides a rougher but more general formula. In practice, especially in Section \ref{ssec : toy problem}, we will use the formulas of Lemma \ref{lem : properties dirichlet} instead of the ones of Proposition \ref{prop : Ki is compact} to benefit from its sharper estimations.
\end{remark}

Now, we turn to the inverse of the Neumann Laplacian. In this case, we use the notations \eqref{def : odd and even restrictions} and restrict our operators to odd functions (respectively we restrict our sequences to $\ell^1_{1,o})$.  In this case, we consider $\mathcal{L}_1$ to be the following (unbounded) linear operator on $\ell^1_{1,o}$
\begin{align*}
    (\tilde{\cL}_1U)_n = \begin{cases}
         0 &\text{ if } n=0,\\
          \cB_{1}(U,0)  &\text{ if } n=1,\\
          0 &\text{ if } n=2,\\
          2nU_n &\text{ if } n \geq 3 \text{ and $n$ is odd},\\
          0 &\text{ otherwise}.
    \end{cases}
\end{align*}
Similarly as above, the inverse of the (odd) Neumann Laplacian writes
\begin{align}
    \mathcal{L}_1^{-1} = \tilde{\cL}_1^{-1} \Sigma^2\mathcal{C}_{0,2},
\end{align}
where $\tilde{\cL}_1^{-1} : \ell^1_{1,o} \to \ell^1_{1,o}$ is the inverse of $\tilde{\cL}_1$ restricted to $\ell^1_{1,o}.$ It remains to find an explicit expression for $\tilde{\cL}_1^{-1}$. Let $U \in \ell^1_{1,o}$ satisfying $\tilde{\cL}_1^{-1} U = \Phi$ with $\Phi \in \ell^1_{1,o}$ is such that $\Phi_1 = 0$. Then, we have that 
\begin{align*}
    \sum_{n \geq 1} U_n n^2 = 0, ~ U_n = \frac{\Phi_n}{2n} \text{ for all $n \geq 3$ odd and } U_n = 0  \text{ for all $n \geq 2$ even.}
\end{align*}
This implies that 
\begin{align*}
    U_1 = - \frac{1}{2}\sum_{n \geq 3, n \text{ odd}}n \Phi_n  = 0, ~ U_n = \frac{\Phi_n}{2n} \text{ for all $n \geq 3$ odd and } U_n = 0  \text{ for all $n \geq 2$ even.}
\end{align*}
Consequently, combining the above with \eqref{eq : change of basis in appendix},  we obtain the following expression for the inverse of the odd Neumann Laplacian
\begin{align}\label{eq : inverse neumann}
    (\cL_1^{-1}U)_n = \begin{cases}
       0 &\text{ if } n=0,\\
       \displaystyle  -3\frac{U_1}{8} +  \sum_{k\geq 3, k \text{ odd}}  \frac{U_{k}}{k^2-1} &\text{ if } n=1,\\
      \displaystyle   \frac{U_{n-2}}{4n(n-1)} - \frac{U_n}{2(n^2-1)} + \frac{U_{n+2}}{4n(n+1)}  &\text{ if } n \geq 3 \text{ and $n$ odd}.
    \end{cases}
\end{align}
Now, similarly as for the Dirichlet case, we derive useful properties for the inverse of the Neumann Laplacian. In particular, these properties are used in the analysis of Section \ref{ssec:KS} for the Kuramoto-Sivashinski PDE.
\begin{lemma}\label{lem : properties operators KS}
    Let $k \geq 4$, then the following properties hold 
    \begin{equation}\label{eq : estimates KS tail}
        \begin{aligned}
\|\cL_1^{-1}E_k^{(1)}\|_{\ell^1_{1,o}} &\leq \frac{1}{k^2-1} + \frac{1}{4(k-2)(k-1)} + \frac{1}{2(k^2-1)} + \frac{1}{4(k+2)(k+1)}, \\
\|\mathcal{L}_1^{-2}E_k^{(1)}\|_{\ell^1_{1,o}} &\leq 0.42\|\cL_1^{-1}E_k^{(1)}\|_{\ell^1_{1,o}},\\
\|\cG_0^{-1}\partial_x \cG_0 \cL_1^{-2}E_k^{(1)}\|_{\ell^1_{1,e}} &\leq \|\cL_1^{-1}E_k^{(1)}\|_{\ell^1_{1,o}}.
    \end{aligned}
    \end{equation}
    Moreover, for all $ k \in \N_0$, we have 
    \begin{align*}
\|\mathcal{L}_1^{-2}E_k^{(1)}\|_{\ell^1_{1,o}} \leq 0.17, \qquad
        \|\cG_0^{-1}\partial_x \cG_0 \cL_1^{-2}E_k^{(1)}\|_{\ell^1_{1,e}}  \leq 0.31.
    \end{align*}
\end{lemma}

\begin{proof}
    First, recall that we have an explicit expression for the decay of $\cL_1^{-1}$ given in \eqref{eq : inverse neumann}. In particular, we obtain that 
    \begin{align}\label{eq : estimates L0inv}
\|\cL_1^{-1}E_k^{(1)}\|_{\ell^1_{1,o}} \leq \frac{1}{k^2-1} + \frac{1}{4(k-2)(k-1)} + \frac{1}{2(k^2-1)} + \frac{1}{4(k+2)(k+1)}
    \end{align}
    for all $k \geq 5$.
Then, using \eqref{eq : ell1opnorm}, note that 
\begin{align}\label{eq : decomposition finite tail}
    \|\cL_1^{-1}\|_{\ell^1_{1,o}} = \max\left\{\|\cL_1^{-1} \pi^{\leq N}\|_{\ell^1_{1,o}}, ~ \|\cL_1^{-1}\pi^{> N}\|_{\ell^1_{1,o}}\right\}.
\end{align}
In fact, using \eqref{eq:Kifinite}, we have that  $\cL_1^{-1} \pi^{\leq N} = \pi^{\leq N+2}\cL_1^{-1} \pi^{\leq N}$ is a finite matrix.
Choosing $N=200$, we use  rigorous numerics in \cite{julia_cadiot} to obtain that  
\begin{align*}
    \|\cL_1^{-1} \pi^{\leq N}\|_{\ell^1_{1,o}} \leq 0.42.
\end{align*}
On the other hand, we have that 
\begin{align*}
    \|\cL_1^{-1}\pi^{> N}\|_{\ell^1_{1,o}} \leq \frac{1}{(N+1)^2-1} + \frac{1}{4(N-1)N} + \frac{1}{2((N+1)^2-1)} + \frac{1}{4(N+3)(N+2)}
    < 0.42
\end{align*}
using \eqref{eq : estimates L0inv}. In particular, we have obtained that $\|\cL_1^{-1}\|_{\ell^1_{1,o}} \leq 0.42.$ This implies that 
\begin{align*}
\|\mathcal{L}_1^{-2}E_k^{(1)}\|_{\ell^1_{1,o}} &\leq \|\mathcal{L}_1\|_{\ell^1_{1,o}} \|\cL_1^{-1}E_k^{(1)}\|_{\ell^1_{1,o}} \leq  0.42 \|\cL_1^{-1}E_k^{(1)}\|_{\ell^1_{1,o}}.
\end{align*}

 Now, let $u = \mathcal{G}_0(U)$ for some $U \in \ell^1_1$ and let $w = \partial_x \mathcal{G}_0\cL_1^{-1}U$. Then, straightforward computations lead to
    \begin{align*}
        w' = u \text{ and } w(-1) = 0.
    \end{align*}
    This implies that $\mathcal{G}_0^{-1}(w) = W = \mathcal{S}U$ where $\mathcal{S}$ is given in \eqref{def : antideriva}. Noticing that $\|\mathcal{S}\|_{\ell^1_{1,o} \to \ell^1_{1,e}} \leq 1$, where $\|\cdot\|_{\ell^1_{1,o} \to \ell^1_{1,e}}$ is the usual operator norm for bounded linear operators from  $\ell^1_{1,o}$ to $\ell^1_{1,e}$, we get $\|\cG_0^{-1}\partial_x \cG_0 \cL_1^{-2}E_k^{(1)}\|_{\ell^1_{1,e}} \leq \|\cL_1^{-1}E_k^{(1)}\|_{\ell^1_{1,o}}$. In order to prove that 
    \[
    \|\mathcal{L}_1^{-2}E_k^{(1)}\|_{\ell^1_{1,o}} \leq 0.17 \quad \mbox{and} \quad \|\cG_0^{-1}\partial_x \cG_0 \cL_1^{-2}E_k^{(1)}\|_{\ell^1_{1,e}}  \leq 0.31,
    \]
    we follow the decomposition \eqref{eq : decomposition finite tail}, compute the finite truncation $\pi^{\leq N}$ thanks to rigorous numerics, and use \eqref{eq : estimates KS tail} for the tail part $\pi^{>N}$. The computer-assisted details are given in \cite{julia_cadiot}.
\end{proof}

\begin{remark}\label{rem : operators KS and eta}
    Similarly to what was observed in Remark \ref{rem : dirichlet and eta}, Lemma \ref{lem : properties operators KS} provides sharper estimates than the general formulas $\eta_{i,k}^{(1)}$ provided in Proposition \ref{prop : Ki is compact}. Specifically, we will benefit from these sharper estimates in Section \ref{ssec:KS}.
\end{remark}

\section{Proof of Lemma~\ref{lem.lambdastarTM2}}\label{App : rough enclosure spectrum KS}

Let $v \in H^4(-1,1)$ be an eigenvector satisfying the boundary conditions~\eqref{eq : KS original}, associated with $\lambda \in \sigma\left(Dg(\tilde{v})\right)$, that is 
\begin{align*}
    \begin{cases}
        -\partial_x^4 v - \alpha\, \partial_x^2 v - \alpha \, v \, \partial_x \tilde{v} - \alpha \tilde{v} \pa_x v = \lambda v ~~ &(t,x) \in (0,T) \times (-1,1),\\
        \partial_x v(t,-1) = \partial_x v(t,1) = \partial_x^3 v(t,-1) = \partial_x^3 v(t,1) =0 ~~ &t \in (0,T),\\
        v(t,-x) = -v(t,x) ~~ &(t,x) \in (0,T)\times[-1,1]. 
    \end{cases}
\end{align*}
Then, multiplying the above equation by $v^*$, the complex conjugate of $v$, and integrating by parts we obtain
\begin{align}\label{tot0}
    \lambda \|v\|^2_{L^2(-1,1)} = - \|\pa^2_x v\|^2_{L^2(-1,1)} - \alpha \int_{-1}^1 \pa_x^2 v \, v^* \, dx -\alpha \int_{-1}^1 |v|^2 \, \pa_x \tilde{v} \, dx - \alpha \int_{-1}^1 \tilde{v} \, \pa_x v \, v^* \, dx.
\end{align}
Hence, since $\tilde{v}$ is a real valued function, we get
\begin{align*}
    \Re(\lambda) \, \|v\|^2_{L^2(-1,1)} = - \|\pa_x^2 v\|^2_{L^2(-1,1)} - \alpha \int_{-1}^1 |v|^2 \, \pa_x \tilde{v} \, dx - \alpha \,\Re\left(\int_{-1}^1 \pa_x^2 v \, v^* \, dx + \int_{-1}^1 \tilde{v} \, \pa_x v \, v^* \, dx \right).
\end{align*}
For an arbitrary $\kappa_1>0$, the elementary relation $ab \leq \kappa_1 a^2 + b^2/(4\kappa_1)$ and the Cauchy-Schwarz inequality yield
\begin{multline*}
    \Re(\lambda) \|v\|^2_{L^2(-1,1)} \leq - \left(1-\alpha \kappa_1\right) \|\pa_x^2 v\|^2_{L^2(-1,1)} +\alpha \|\pa_x \tilde{v}\|_{L^\infty(-1,1)} \, \|v\|^2_{L^2(-1,1)} \\+ \alpha\dfrac{\|v\|^2_{L^2(-1,1)}}{4\kappa_1} +\alpha \|\tilde{v}\|_{L^\infty(-1,1)} \, \|\pa_x v\|_{L^2(-1,1)} \, \|v\|_{L^2(-1,1)}.
\end{multline*}
Since $\partial_xv$ vanishes at $\pm1$, the one-dimensional Dirichlet Poincaré inequality yields
\begin{align}\label{ineq: Poincare}
    \|\pa_x v\|_{L^2(-1,1)} \leq \frac{2}{\pi} \|\pa_x^2 v\|_{L^2(-1,1)}.
\end{align}
Therefore, for any $\kappa_2>0$ we obtain
\begin{multline}\label{toto2}
        \Re(\lambda) \|v\|^2_{L^2(-1,1)} \leq - \left(1-\alpha\kappa_1-\frac{2 \alpha\kappa_2}{\pi}\Vert \tilde{v}\Vert_{L^\infty(-1,1)}\right) \|\pa_x v\|^2_{L^2(-1,1)} \\
    + \alpha\left( \|\pa_x\tilde{v}\|_{L^\infty(-1,1)} + \frac{1}{2\pi\kappa_2}\Vert \tilde{v}\Vert_{L^\infty(-1,1)} + \frac{1}{4\kappa_1} \right) \|v\|^2_{L^2(-1,1)},
\end{multline}
which yields~\eqref{eq:realpart_KS} as soon as the constraint~\eqref{eq:constraints_toto} is satisfied.

Next, coming back to the identity~\eqref{tot0} and using once more~\eqref{ineq: Poincare}, we have
\begin{align*}
    \left|\Im(\lambda)\right| \, \|v\|^2_{L^2(-1,1)} \leq \alpha\left(1  + \frac{2}{\pi}\|\tilde{v}\|_{L^\infty(-1,1)} \right) \,\|\pa_x^2 v\|_{L^2(-1,1)} \, \|v\|_{L^2(-1,1)}.
\end{align*}
Now, if we assume that $\Re(\lambda)\geq -\mu$ and that $\kappa_1$ and $\kappa_2$ satisfy the constraint~\eqref{eq:constraints_toto}, we deduce from~\eqref{toto2} that
\begin{align}\label{toto3}
    \|\pa_x^2 v\|^2_{L^2(-1,1)} \leq \dfrac{\mu +\alpha\left( \|\pa_x\tilde{v}\|_{L^\infty(-1,1)} + \frac{1}{2\pi\kappa_2}\Vert \tilde{v}\Vert_{L^\infty(-1,1)} + \frac{1}{4\kappa_1} \right)}{1-\alpha\kappa_1-\frac{2 \alpha\kappa_2}{\pi}\Vert \tilde{v}\Vert_{L^\infty(-1,1)}} \|v\|^2_{L^2(-1,1)}.
\end{align}
Hence
\begin{align*}
    |\Im(\lambda)| \leq \alpha\left(1  + \frac{2}{\pi}\|\tilde{v}\|_{L^\infty(-1,1)} \right)
    \left(\dfrac{\mu +\alpha\left( \|\pa_x\tilde{v}\|_{L^\infty(-1,1)} + \frac{1}{2\pi\kappa_2}\Vert \tilde{v}\Vert_{L^\infty(-1,1)} + \frac{1}{4\kappa_1} \right)}{1-\alpha\kappa_1-\frac{2 \alpha\kappa_2}{\pi}\Vert \tilde{v}\Vert_{L^\infty(-1,1)}}\right)^{1/2},
\end{align*}
which concludes the proof.

\section*{Acknowledgments}

  The three authors were supported by the ANR project CAPPS: ANR-23-CE40-0004-01. Moreover, MC was also supported by the FMJH :  ANR-22-EXES-0013.

\end{appendix}

\bibliographystyle{abbrv}
\bibliography{biblio}

@article{NakKin09,
  title={On very accurate verification of solutions for boundary value problems by using spectral methods},
  author={Nakao, Mitsuhiro T and Kinoshita, Takehiko},
  journal={JSIAM Letters},
  volume={1},
  pages={21--24},
  year={2009},
  publisher={The Japan Society for Industrial and Applied Mathematics}
}

@article{BreBriJol18,
  title={Validated and numerically efficient {C}hebyshev spectral methods for linear ordinary differential equations},
  author={Br{\'e}hard, Florent and Brisebarre, Nicolas and Jolde{\c{s}}, Mioara},
  journal={ACM Transactions on Mathematical Software (TOMS)},
  volume={44},
  number={4},
  pages={1--42},
  year={2018},
  publisher={ACM New York, NY, USA}
}

@article{Zgl09,
author = {P. Zgliczynski},
title = {Covering relations, cone conditions and the stable manifold theorem},
journal = {J. Differential Equations},
year = {2009},
volume = {246},
number = {5},
pages = {1774-1819},
}

@article{Cad25,
  title={Stability analysis for localized solutions in {PDE}s and nonlocal equations on $\mathbb{R}^m$},
  author={Cadiot, Matthieu},
  journal={arXiv:2505.03091},
  year={2025}
}

@article{BrePayReiTan25,
title={Turing Instability for Nonlocal Heterogeneous Reaction-Diffusion Systems: {A} Computer-Assisted Proof Approach},
  author={Breden, Maxime and Payan, Maxime and Reisch, Cordula and Tang, Bao Quoc},
  journal={Journal of Dynamics and Differential Equations},
  pages={1--48},
  year={2026},
  publisher={Springer}
}

@inproceedings{Bre18bis,
  title={{A Newton-like validation method for Chebyshev approximate solutions of linear ordinary differential systems}},
  author={Br{\'e}hard, Florent},
  booktitle={Proceedings of the 2018 ACM International Symposium on Symbolic and Algebraic Computation},
  pages={103--110},
  year={2018}
}

@article{OlvTow13,
  title={A fast and well-conditioned spectral method},
  author={Olver, Sheehan and Townsend, Alex},
  journal={siam REVIEW},
  volume={55},
  number={3},
  pages={462--489},
  year={2013},
  publisher={SIAM}
}

@misc{DLMF,
         key = "{\relax DLMF}",
       title = "{\it NIST Digital Library of Mathematical Functions}",
howpublished = "\url{https://dlmf.nist.gov/}, Release 1.2.4 of 2025-03-15",
         url = "https://dlmf.nist.gov/",
        note = "F.~W.~J. Olver, A.~B. {Olde Daalhuis}, D.~W. Lozier, B.~I. Schneider,
                R.~F. Boisvert, C.~W. Clark, B.~R. Miller, B.~V. Saunders,
                H.~S. Cohl, and M.~A. McClain, eds."}

@article{Bre23,
  title={A posteriori validation of generalized polynomial chaos expansions},
  author={Breden, Maxime},
  journal={SIAM Journal on Applied Dynamical Systems},
  volume={22},
  number={2},
  pages={765--801},
  year={2023},
  publisher={SIAM}
}

@book{Tre13,
 author = {Trefethen, Lloyd N.},
 title = {Approximation theory and approximation practice},
 fseries = {Other Titles in Applied Mathematics},
 series = {Other Titles Appl. Math.},
 volume = {128},
 isbn = {978-1-611972-39-9},
 year = {2013},
 publisher = {Philadelphia, PA: Society for Industrial {and} Applied Mathematics (SIAM)},
 language = {English},
 zbMATH = {6135818},
 Zbl = {1264.41001}
}

@book{Mar86,
 author = {Markowich, Peter A.},
 title = {The stationary semiconductor device equations},
 series = {Computational Microelectronics},
 isbn = {978-3-211-99937-0},
 year = {1986},
 publisher = {Springer-Verlag, Vienna},
 language = {English}
}

@article{Ala92,
 author = {Alabau, Fatiha},
 title = {A method for proving uniqueness theorems for the stationary semiconductor device and electrochemistry equations},
 fjournal = {Nonlinear Analysis. Theory, Methods \& Applications},
 journal = {Nonlinear Anal., Theory Methods Appl.},
 issn = {0362-546X},
 volume = {18},
 number = {9},
 pages = {861--872},
 year = {1992},
 language = {English},
 doi = {10.1016/0362-546X(92)90227-6},
 zbMATH = {147924},
 Zbl = {0768.34007}
}

@article{ArKoTe05,
 author = {Arioli, Gianni and Koch, Hans and Terracini, Susanna},
 title = {Two novel methods and multi-mode periodic solutions for the {Fermi}-{Pasta}-{Ulam} model},
 fjournal = {Communications in Mathematical Physics},
 journal = {Commun. Math. Phys.},
 issn = {0010-3616},
 volume = {255},
 number = {1},
 pages = {1--19},
 year = {2005},
 language = {English},
 doi = {10.1007/s00220-004-1251-z},
 zbMATH = {2214804},
 Zbl = {1076.70008}
}

@article{DaLeMi07,
 author = {Day, Sarah and Lessard, Jean-Philippe and Mischaikow, Konstantin},
 title = {Validated continuation for equilibria of {PDEs}},
 fjournal = {SIAM Journal on Numerical Analysis},
 journal = {SIAM J. Numer. Anal.},
 issn = {0036-1429},
 volume = {45},
 number = {4},
 pages = {1398--1424},
 year = {2007},
 language = {English},
 doi = {10.1137/050645968},
 zbMATH = {5312192},
 Zbl = {1151.65074}
}

@article{Szw05,
  title={{Orthogonal polynomials and Banach algebras}},
  author={Szwarc, Ryszard},
  journal={Inzell Lectures on Orthogonal Polynomials. Advances in the Theory of Special Functions and Orthogonal Polynomials, Nova Science Publishers},
  volume={2},
  pages={103--139},
  year={2005}
}

@article{BreDesLes15,
  title={Rigorous numerics for nonlinear operators with tridiagonal dominant linear part},
  author={Breden, Maxime and Desvillettes, Laurent and Lessard, Jean-Philippe},
  journal={Discrete \& Continuous Dynamical Systems-A},
  volume={35},
  number={10},
  pages={4765--4789},
  year={2015},
  publisher={American Institute of Mathematical Sciences}
}

@article{julia_olivier,
author = {Olivier Hénot},
title = {RadiiPolynomial.jl},
year = {2022},
URL = {  https://github.com/OlivierHnt/RadiiPolynomial.jl
},
eprint = { https://doi.org/10.1137/141000671
},

 note = {\url{  https://github.com/OlivierHnt/RadiiPolynomial.jl}}
  
}

@article{julia_interval,
author = {L. Benet and D.P. Sanders},
title = {IntervalArithmetic.jl},
year = {2022},
URL = { https://github.com/JuliaIntervals/IntervalArithmetic.jl
},
eprint = { https://doi.org/10.1137/141000671
},
  note = {\url{ https://github.com/JuliaIntervals/IntervalArithmetic.jl}}
}

@article {chebfun,
    AUTHOR = {Battles, Zachary and Trefethen, Lloyd N.},
     TITLE = {An extension of {MATLAB} to continuous functions and
              operators},
   JOURNAL = {SIAM J. Sci. Comput.},
  FJOURNAL = {SIAM Journal on Scientific Computing},
    VOLUME = {25},
      YEAR = {2004},
    NUMBER = {5},
     PAGES = {1743--1770},
      ISSN = {1064-8275,1095-7197},
   MRCLASS = {41-04 (65-04)},
  MRNUMBER = {2087334},
       DOI = {10.1137/S1064827503430126},
       URL = {https://doi.org/10.1137/S1064827503430126},
}

@article{julia_cadiot,
  author = {Matthieu Cadiot},
  title  = {{BVP\_Gegenbauer}.jl},
  URL    = {https://github.com/matthieucadiot/BVP_Gegenbauer.jl},
  note = {\url{ https://github.com/matthieucadiot/BVP_Gegenbauer.jl}},
  year   = {2026}
}

@article {plum_orr_summerfeld,
    AUTHOR = {Lahmann, Jan-Rainer and Plum, Michael},
     TITLE = {A computer-assisted instability proof for the
              {O}rr-{S}ommerfeld equation with {B}lasius profile},
   JOURNAL = {ZAMM Z. Angew. Math. Mech.},
  FJOURNAL = {ZAMM. Zeitschrift f\"ur Angewandte Mathematik und Mechanik.
              Journal of Applied Mathematics and Mechanics},
    VOLUME = {84},
      YEAR = {2004},
    NUMBER = {3},
     PAGES = {188--204},
      ISSN = {0044-2267,1521-4001},
   MRCLASS = {76E05 (34B40 34L16 47A75 47N20 65G40 65L15)},
  MRNUMBER = {2038339},
MRREVIEWER = {Marco\ Marletta},
       DOI = {10.1002/zamm.200310093},
       URL = {https://doi.org/10.1002/zamm.200310093},
}

@book {henry_semilinear,
    AUTHOR = {Henry, Daniel},
     TITLE = {Geometric theory of semilinear parabolic equations},
    SERIES = {Lecture Notes in Mathematics},
    VOLUME = {840},
 PUBLISHER = {Springer-Verlag, Berlin-New York},
      YEAR = {1981},
     PAGES = {iv+348},
      ISBN = {3-540-10557-3},
   MRCLASS = {35K55 (34G20 58D25)},
  MRNUMBER = {610244},
MRREVIEWER = {J.\ A.\ Goldstein},
}

@article{BreChu25,
  title={{Solutions of differential equations in Freud-weighted Sobolev spaces}},
  author={Breden, Maxime and Chu, Hugo},
  journal={arXiv preprint arXiv:2501.13672},
  year={2025}
}

@article{AriKoc10,
author = {G. Arioli and H. Koch},
title = {{Computer-assisted methods for the study of stationary solutions in dissipative systems, applied to the {K}uramoto-{S}ivashinski equation}},
journal = {Arch. Ration. Mech. Anal.},
year = {2010},
volume = {197},
number = {3},
pages = {1033-1051},
}

@article{HunLesMir16,
  title={Rigorous numerics for analytic solutions of differential equations: the radii polynomial approach},
  author={Hungria, Allan and Lessard, Jean-Philippe and Mireles James, J},
  journal={Mathematics of Computation},
  volume={85},
  number={299},
  pages={1427--1459},
  year={2016}
}

@article{Ber17,
  title={Introduction to rigorous numerics in dynamics: general functional analytic setup and an example that forces chaos},
  author={van den Berg, Jan Bouwe},
  journal={Rigorous Numerics in Dynamics, Proc. Sympos. Appl. Math., Amer. Math. Soc},
  volume={74},
  pages={1--25},
  year={2017}
}

@article{MirRei19,
  title={{Fourier--Taylor} parameterization of unstable manifolds for parabolic partial differential equations: formalism, implementation and rigorous validation},
  author={Reinhardt, Christian and Mireles James, JD},
  journal={Indagationes Mathematicae},
  volume={30},
  number={1},
  pages={39--80},
  year={2019},
  publisher={Elsevier}
}

@article{WilZgl20,
  title={A geometric method for infinite-dimensional chaos: {S}ymbolic dynamics for the {K}uramoto--{S}ivashinsky {PDE} on the line},
  author={Wilczak, Daniel and Zgliczy{\'n}ski, Piotr},
  journal={Journal of Differential Equations},
  volume={269},
  number={10},
  pages={8509--8548},
  year={2020},
  publisher={Elsevier}
}

@article{ZglMis01,
  title={Rigorous numerics for partial differential equations: The {K}uramoto--{S}ivashinsky equation},
  author={Zgliczynski, Piotr and Mischaikow, Konstantin},
  journal={Foundations of Computational Mathematics},
  volume={1},
  number={3},
  pages={255--288},
  year={2001},
  publisher={Springer}
}

@article{Zgl02bis,
  title={{Attracting fixed points for the Kuramoto--Sivashinsky equation: A computer assisted proof}},
  author={Zgliczynski, Piotr},
  journal={SIAM Journal on Applied Dynamical Systems},
  volume={1},
  number={2},
  pages={215--235},
  year={2002},
  publisher={SIAM}
}

@article{cadiot2025recentadvancesrigorousintegration,
      title={Recent advances about the rigorous integration of parabolic {PDEs} via fully spectral {Fourier-Chebyshev} expansions}, 
      author={Matthieu Cadiot and Jean-Philippe Lessard},
      year={2025},
      journal = {arXiv:2502.20644},
      eprint={2502.20644},
      archivePrefix={arXiv},
      primaryClass={math.AP},
      url={https://arxiv.org/abs/2502.20644}, 
}

@book {boyd_cheb_fourier,
    AUTHOR = {Boyd, John P.},
     TITLE = {Chebyshev and {F}ourier spectral methods},
   EDITION = {Second},
 PUBLISHER = {Dover Publications, Inc., Mineola, NY},
      YEAR = {2001},
     PAGES = {xvi+668},
      ISBN = {0-486-41183-4},
   MRCLASS = {65M70 (65N35)},
  MRNUMBER = {1874071},
}

@article{WilZgl24,
  title={Symbolic dynamics for the {K}uramoto--{S}ivashinsky {PDE} on the line {II}},
  author={Wilczak, Daniel and Zgliczy{\'n}ski, Piotr},
  journal={arXiv preprint arXiv:2405.17087},
  year={2024}
}

@book {NakPluWat19,
    AUTHOR = {Mitsuhiro T. Nakao and Michael Plum and Yoshitaka Watanabe},
     TITLE = {Numerical Verification Methods and Computer-Assisted Proofs for Partial Differential Equations},
    SERIES = {Springer Series in Computational Mathematics},
    VOLUME = {53},
 PUBLISHER = {Springer Singapore},
      YEAR = {2019},
     PAGES = {xii+467},
      ISBN = {978-981-13-7668-9},
       DOI = {10.1007/978-981-13-7669-6},
}

@article{Ort68,
 author = {Ortega, J. M.},
 title = {The {Newton}-{Kantorovich} theorem},
 fjournal = {American Mathematical Monthly},
 journal = {Am. Math. Mon.},
 issn = {0002-9890},
 volume = {75},
 pages = {658--660},
 year = {1968},
 language = {English},
 doi = {10.2307/2313800},
 zbMATH = {3292326},
 Zbl = {0183.43004}
}

@article{Yam98,
 author = {Yamamoto, Nobito},
 title = {A numerical verification method for solutions of boundary value problems with local uniqueness by {Banach}'s fixed-point theorem},
 fjournal = {SIAM Journal on Numerical Analysis},
 journal = {SIAM J. Numer. Anal.},
 issn = {0036-1429},
 volume = {35},
 number = {5},
 pages = {2004--2013},
 year = {1998},
 language = {English},
 doi = {10.1137/S0036142996304498},
 zbMATH = {1200945},
 Zbl = {0972.65084}
}

@article{BerLes15,
author={van den Berg, J. B. and Lessard, J.-P.},
title={Rigorous numerics in dynamics},
journal={Notices Amer. Math. Soc.},
volume={62},
number={9},
year={2015}
}

@article{Gom19,
  title={Computer-assisted proofs in {PDE}: a survey},
  author={G{\'o}mez-Serrano, Javier},
  journal={SeMA Journal},
  volume={76},
  number={3},
  pages={459--484},
  year={2019},
  publisher={Springer}
}

@article{CAPD,
  title={{CAPD:: DynSys: a flexible C++ toolbox for rigorous numerical analysis of dynamical systems}},
  author={Kapela, Tomasz and Mrozek, Marian and Wilczak, Daniel and Zgliczy{\'n}ski, Piotr},
  journal={Communications in nonlinear science and numerical simulation},
  volume={101},
  pages={105578},
  year={2021},
  publisher={Elsevier}
}

@article{BreChaZur21,
  title={Existence of traveling wave solutions for the {Diffusion Poisson Coupled Model: a computer-assisted proof}},
  author={Breden, Maxime and Chainais-Hillairet, Claire and Zurek, Antoine},
  journal={ESAIM: Mathematical Modelling and Numerical Analysis},
  volume={55},
  number={4},
  pages={1669--1697},
  year={2021}
}

@article{LesRei14,
  author={Lessard, J.-P. and Reinhardt, C.},
  title={Rigorous numerics for nonlinear differential equations using {C}hebyshev series},
  journal={SIAM J. Numer. Anal.},
  volume={52},
  number={1},
  pages={1--22},
  year={2014},
  publisher={SIAM}
}

@article{BerShe21,
  title={Rigorous numerics for {ODE}s using {C}hebyshev series and domain decomposition},
  author={van den Berg, Jan Bouwe and Sheombarsing, Ray},
  journal={Journal of Computational Dynamics},
  volume={8},
  number={3},
  pages={353},
  year={2021},
  publisher={American Institute of Mathematical Sciences}
}

@book{kato2013perturbation,
  title={Perturbation theory for linear operators},
  author={Kato, Tosio},
  volume={132},
  year={2013},
  publisher={Springer Science \& Business Media}
}

@article {cadiot_MMT,
    AUTHOR = {Cadiot, Matthieu and Jaquette, Jonathan and Lessard,
              Jean-Philippe and Takayasu, Akitoshi},
     TITLE = {Validated matrix multiplication transform for orthogonal
              polynomials with applications to computer-assisted proofs for
              {PDE}s},
   JOURNAL = {Commun. Nonlinear Sci. Numer. Simul.},
  FJOURNAL = {Communications in Nonlinear Science and Numerical Simulation},
    VOLUME = {151},
      YEAR = {2025},
     PAGES = {Paper No. 109063, 26},
      ISSN = {1007-5704,1878-7274},
   MRCLASS = {65N35 (33C45 65D32)},
  MRNUMBER = {4927489},
       DOI = {10.1016/j.cnsns.2025.109063},
       URL = {https://doi.org/10.1016/j.cnsns.2025.109063},
}

\end{document}